\documentclass[12pt]{article}
\usepackage[english]{babel}
\usepackage{amsfonts}
\usepackage{amssymb}
\usepackage{amsmath}
\usepackage{amsthm}
\usepackage{latexsym}
\usepackage{bm}
\usepackage{fancyhdr}
\usepackage{subcaption}
\usepackage{tabularx}
\usepackage{booktabs} 
\usepackage{enumerate}
\usepackage{caption}
\usepackage{wrapfig}  
\usepackage{color}
\usepackage{numprint}
\usepackage{mathtools}  
\usepackage{multirow}
\usepackage{multicol}
\usepackage{eqnarray}
\usepackage{enumitem}
\usepackage{makecell}
\usepackage{xcolor}
\usepackage[titletoc]{appendix}

\usepackage[numbers]{natbib}
\usepackage{comment}

\usepackage{bbm} 
\usepackage{url}
\usepackage{color}
\usepackage{float}

\usepackage{hyperref}
\hypersetup{colorlinks=true, citecolor=blue, linkcolor=red, urlcolor=green}

\usepackage{xcolor}
\usepackage[section]{placeins}

\fancypagestyle{plain}{%
  \fancyhf{}
  \fancyfoot[C]{\iffloatpage{}{\thepage}}
  }
\theoremstyle{plain}
\newtheorem{theorem}{Theorem}[section]
\newtheorem{lemma}[theorem]{Lemma}
\newtheorem{corollary}[theorem]{Corollary}
\newtheorem{proposition}[theorem]{Proposition}
\newtheorem{conjecture}[theorem]{Conjecture}

\theoremstyle{definition}

\theoremstyle{remark}
\newtheorem*{remark}{Remark}

\newcommand{\fcfs}[0]{first-come, first-served~}
\newcommand{\Soj}[0]{\mathcal{S}}

\title{When Lines Become Contagious: Analyzing the Spread of Infection in the M/M/1 Queue}
\author{ Laurel Newman   \\ School of Operations Research and Information Engineering \\ Cornell
University \\ {lan66@cornell.edu}
\and 
Jamol Pender \footnote{Corresponding Author}
\\
School of Operations Research and Information Engineering
\\
Cornell University
\\
{jjp274@cornell.edu}
}

\usepackage[final]{graphicx}
\graphicspath{{/images/}}

\makeatletter
\newcommand{\sf@counterlist}{}
\makeatother
\begin{document}

\maketitle

\begin{abstract}
In this paper, we construct and analyze a stochastic model for infectious disease spread in service systems.  Our model considers a $M/M/1$ queue where a proportion, $p$, of the arriving population is infectious and the remaining population is susceptible.  When an infectious person overlaps in the queue with a susceptible customer for more than $\ell$ units of time, then the susceptible person becomes infected. We provide both a personalized perspective of how an individual susceptible customer is infected and a community perspective of how one infectious customer can infect the others around them.  From the personalized perspective, we derive the probability that a susceptible person becomes infected and the probability they are infected by at least $k$ different people.  From the community perspective we compute the mean and higher moments of the number of susceptible customers that are infected by one infectious individual.  Finally, we use stochastic simulation to verify our analytical results and show that our new formulas accurately describe the behavior our stochastic model of infection and to expand our model into the consideration of different arrival and service regimes.

\end{abstract}

\section{Introduction}\label{sec:intro}

Stochastic models serve as essential tools in applied probability, offering a rich framework to understand and analyze random phenomena inherent in various real-world systems. It has never been a more important time to leverage stochastic models in the context of infectious disease due to the Covid-19 pandemic.  One widely used stochastic model is the SEIR model, which tracks individuals' transitions between Susceptible, Exposed, Infectious, and Recovered states. It incorporates key parameters like transmission rate, incubation period, infectious period, and recovery rate to describe disease spread within a population. These models are invaluable for simulating various disease scenarios and assessing the impact of interventions, aiding policymakers and healthcare professionals in decision-making.

Unfortunately, deterministic compartmentalized models have dominated the Covid-19 literature. Recognizing the significance of stochastic infection models is vital for understanding virus transmission see for example \citet{dandekar2021safe, nguemdjo2020simulating, kaplan2020om}. These stochastic models offer insights into how infections spread, guiding interventions, see for example \citet{saavedra1988simple, allen2008introduction, zhang2013stochastic, drakopoulos2017network, derjany2020multiscale, palomo2020flattening, pang2020functional, forien2020epidemic, moein2021inefficiency, daw2022ephemerally}. Measures like air filter installations, transparent barriers, mask mandates, and physical distancing have been implemented to curb virus spread in service systems. 

Current stochastic models assume uniform social interactions across the entire population, neglecting variations due to factors like age, occupation, and time. Such heterogeneity can significantly affect virus dynamics. To combat this, many researchers employ simulations to explore these complexities, yielding practical insights. For example, various studies such as \citet{ying2021modelling, namilae2020multiscale, agnelli2023spatial, johansson2014utilizing, ronchi2020exposed, gosce2014analytical, abadeer2020simulating, rahn2022modelling} have simulated COVID-19 spread in diverse ways, but these simulations may lack mathematical rigor.

Additionally, in service systems, the entire population does not interact simultaneously, emphasizing the need for models accounting for customer overlap. Fortunately, there is a growing body of literature addressing overlap times in queueing models, enabling the modeling of infectious disease spread in retail settings, see for example \citet{palomo2021measuring, ko2022overlapping, xu2023queueing, palomo2021overlap, gao2023overlap, ko2023number, palomo2023maximum, boxma2024maximum, boxma2025overlap}.  Recent work by \citet{kang2021queueing} and \citet{palomo2023overlap} explores how one can calculate the overlap times of customers in multi-server and single server queues respectively.  More specifically, \citet{kang2021queueing} show how to use the overlaps to compute a new $R_0$ value for understanding infection rates in compartmentalized epidemic models and \citet{palomo2023overlap} prove that the overlap distribution is exponential for the M/M/1 queue and show via simulation that the overlap distribution depends on the service time distribution explicitly.  Our analysis is important because it can demonstrate, exactly, how much overlap occurs and can provide distributional information or prediction intervals for possible overlap.  Moreover, it can be used as a tool to prevent large overlaps and as a design tool to construct appropriate overlap by restricting the arrival rate or service distributions.  

Our model is similar to the novel work of \citet{kang2021queueing}, but we will highlight a few differences that motivate our analysis. The first difference is that we explicitly compute the probability a susceptible customer becomes infected in the system.  This is an important metric from a queueing and infectious disease perspective and provides a \emph{personalized} metric of infection.  We also compute the probability that a susceptible customer is infected by at least $k$ different people, which provides a novel way of measuring the severity of the illness when a susceptible is infected.  It is well known that the severity of an illness is dependent on the initial exposure to the virus \citet{memoli2015validation, liu2020viral}.  Thus, our analysis provides a way of not only determining who will get infected, but also how sick they may become.  

A second difference is that our methodology, which uses sample path arguments, allows one to compute higher moments of the number of infected customers.  In particular, the arguments used in \citet{kang2021queueing} only work for computing the mean number infected and our work provides a new formula for the variance, which is essential for producing prediction intervals for understanding community based infection. A third difference is that our model assumes a single, fixed threshold (which can be randomized later) which serves as a random variable for every susceptible customer (in the personalized perspective) rather than as a random variable for every susceptible-infectious pair, as the $R_0^{sys}$ formula of \citet{kang2021queueing}. Notably, this same formula in \citet{kang2021queueing} additionally assumes only a single infectious customer is ever in the system at a time, an assumption we do not make. Moreover, our expression for the mean number infected  highlights the dependence of the mean number infected on the Laplace transform of the threshold distribution.  Finally, our methodology of splitting the analysis into the upon arrival and after arrival settings provides a blueprint for other queueing models and even allows for transient and non-stationary analysis.  

{\color{black} Such single-server epidemic queueing models are especially relevant for bottleneck service systems such as small clinics, testing stations, retail counters, or security checkpoints, where all flow passes through one service channel. Because customers are processed sequentially, each additional arrival increases the waiting time of those behind them, thereby increasing overlap durations and infection probabilities. If the server can also become infected, additional feedback effects arise: reduced service availability lengthens queues, which in turn increases overlap times and transmission risk. Even in this simple single-server setting, the interaction between stochastic arrivals, service times, and overlap-based infection probabilities provides a useful framework for evaluating capacity decisions, arrival control policies, and protective interventions in congestion-sensitive environments.}

In this paper, we analyze the $M/M/1$ queue with two types of customers.  The customers are either infectious or susceptible. Customers arrive to the queue at rate $\lambda$ and customers are labeled as infectious with probability $p$ and the remaining proportion of customers $1-p$ are susceptible.  A susceptible customer becomes infected if they overlap with an infectious customer for longer than a pre-determined threshold $\ell > 0$. We assume that if you become infected in the queue, you cannot infect other susceptible customers who are still in the queue.  This model captures how people interact in retail or service settings and how infectious disease can spread in these systems.  In particular, our analysis emphasizes looking at the queue from susceptible and infectious customers' point of view. For an individual (or ``tagged'') susceptible customer's perspective, they do not want to be infected, thus they care about the probability that they will be infected while waiting for service. Note here the distinction between the proportion of infectious arrivals of the system, $p$, and the infection probability for a single tagged susceptible customer. A tagged infectious customer does not want to infect others.  Thus, they care about the number customers they will infect.  With these personalized and community based perspectives in mind, we provide the following contributions in this work.

\subsection{Contributions of the paper}
Our goal in this work is to answer the following questions
\begin{itemize}
    \item What proportion of the susceptible customers become infected?  
    \item How many infectious customers infect a tagged susceptible customer when they become infected?
    \item What are the moments of the number of customers an infectious customer infects?
    \item Is it more likely for a tagged susceptible customer to get infected from customers already in the queue upon their arrival or from a customer who comes after them?
    \item How does the infection probability or mean number of infections depend on the threshold distribution?
\end{itemize}

\subsection{Comparison to prior work}
For the reader's convenience, we explicitly highlight the distinctions between our work and the prior work by Kang et al.~\cite{kang2021queueing}.
\subsubsection{ Probability a single susceptible customer becomes infected}
This is not explored within \cite{kang2021queueing}. To summarize our work,
\begin{itemize}
    \item We explicitly calculate this over their entire sojourn in the system for the $M/M/1$ (Theorem~\ref{thm:total-infection-prob}), $D/M/1$, $E_2/M/1$, and $H_2/M/1$ systems (Theorem~\ref{thm:gm1infecprob} leveraging Equations~\ref{eqn:rho-deterministic}, \ref{eqn:rho-erlang}, and \ref{eqn:rho-hyperexp}).
    \item We further find a closed form solution for all $G/M/1$ systems, depending only on the Laplace transform of the inter-arrival distribution (Theorem~\ref{thm:gm1infecprob}).
    \item We approximate the complement of this probability for the $M/G/1$ case (Equation~\ref{eqn:mg1probapprox}).
    \item Finally, for the $M/M/1$ setting we additionally provide new distributional results for the upon-arrival and after-arrival perspectives, providing insight into the characterization of the infection dynamics (Theorems~\ref{thm:upon-arrival} and~\ref{thm:after-arrival}).
\end{itemize}
\subsubsection{Expected number of customers an infectious one infects}
From prior work in \cite{kang2021queueing},
\begin{itemize}[left=0pt..1em]
            \item For arbitrary pairwise distributed infection thresholds, Proposition 1 establishes an equation for this, with dependence on the overlap time and queue length distributions for any system \textit{assuming all other customers are susceptible}.
            \item Proposition 2 finds a closed form solution for the $M/M/1$ system, with exponentially distributed pairwise infection thresholds, \textit{assuming all other customers are susceptible}.
        \end{itemize}
In comparison, in our work we expand these to explicit formulations without assumptions on the rest of the customers in the system. In summary,
\begin{itemize}
        \item We explicitly calculate the expectation for the $M/M/1$ (Equation~\ref{eqn:meannuminfec}), $D/M/1$, $E_2/M/1$, and $H_2/M/1$ systems (Corollary~\ref{cor:gm1meaninfec} leveraging Equations~\ref{eqn:rho-deterministic}, \ref{eqn:rho-erlang}, and \ref{eqn:rho-hyperexp}).
        \item We further find a closed form solution for all $G/M/1$ systems, depending only on the Laplace transform of the inter-arrival distribution (Corollary~\ref{cor:gm1meaninfec}).
        \item Finally, for the $M/G/1$ case we bound the expectation (Equation~\ref{eqn:mg1expecbound}), and under the second condition of Conjecture~\ref{conj:symmetry}, we find a non-closed form equation for this, dependent only on the expectation of the sojourn time distribution (Corollary~\ref{conj:mg1meaninfec}).
\end{itemize}
\subsubsection{Variance of the number of customers infected by a single infectious one}
This is not explored within \cite{kang2021queueing}. To summarize our work,
\begin{itemize}
        \item We explicitly calculate the variance for the $M/M/1$ (Equation~\ref{eqn:varnuminfec}), $D/M/1$, $E_2/M/1$, and $H_2/M/1$ systems (Corollary~\ref{cor:gm1varinfec} leveraging Equations~\ref{eqn:rho-deterministic}, \ref{eqn:rho-erlang}, and \ref{eqn:rho-hyperexp}).
        \item We further find a closed form solution for all $G/M/1$ systems, depending only on the Laplace transform of the inter-arrival distribution (Corollary~\ref{cor:gm1varinfec}).
        \item Finally, for the $M/G/1$ case we find an approximation for this, based on the expectation and variance of the sojourn time distribution (Equation~\ref{eqn:mg1varinfec}).
\end{itemize}
\subsubsection{Higher moments of the number of customers infected by a single infectious one}
This is not explored within \cite{kang2021queueing}. In our work we calculate this explicitly for the $M/M/1$ system (Theorem~\ref{thm:moments-num-infect}).

\subsubsection{Number of times a susceptible customer gets infected (as a proxy for infection severity)}
This is not explored within \cite{kang2021queueing}. In our work we calculate this explicitly for the $M/M/1$ system (Theorems~\ref{thm:exactly-k-infections} and \ref{thm:at-least-k}).

\subsection{Notation of the paper}

We consider the M/M/1 queue from the perspective of a tagged customer arriving in steady state.

\begin{itemize}
    \item Let $Q$ be the queue length upon their arrival, $S$ the service time for the tagged customer, $W$ their wait time, $\mathcal{S}$ their sojourn time, and $\mathcal{R}$ their remaining sojourn time.
    \item Indexing from the first customer in the queue upon the tagged customer's arrival, let $S_i$ denote the service times for the $i^{th}$ customer, which are i.i.d.~exponential random variables with rate $\mu$. Let $R_i$ be the  $i^{th}$ customer's remaining service time (note if $i \neq 1$, $R_i = S_i$), $W_i$ their wait time, and $\mathcal{S}_i = S_i+W_i$ their sojourn time.
\end{itemize}

We define some additional notation for use in Section~\ref{sec:prob-infec}.
\begin{itemize}
    \item Let $I_{aa}$ be the event where the tagged customer is susceptible and becomes infected by a customer who arrives after them.  We define as $I_{ua}$ the event where the tagged customer is infected by a customer who was already in the queue upon their arrival. Note specifically that these events are not disjoint, and so define as $I = I_{ua}\cup I_{aa}$ the event the tagged customer gets infected. 
    \item For event $I_{aa}$, we are primarily interested in arrivals after the tagged customer. Let $M_I \coloneqq M_I(\mathcal{S})$, the number of new infectious customers who arrive while the tagged susceptible arrival is still in the system. The associated arrivals, in chronological order, are $A_1, \ldots A_{M_I}$, and we define the inter-arrival time between the $i^{th}$ and $(i+1)^{th}$ customers to be $A_{i+1} - A_i$, which are i.i.d.~exponential random variables with rate parameter $\lambda p$.
    \item For event $I_{ua}$, we are only interested in the customers already in the system upon the arrival of the tagged customer. That is, those who arrived before this customer but who have yet to complete service by this customer's arrival. To this end, we index these customers as $i=1,\ldots,Q$ in chronological order by arrival time (and let $I_{\max}$ be the maximum index amongst these of an infectious customer). Let $C \coloneqq \{c_1,\ldots,c_{Q}\}$ be the customers in the queue upon arrival, and let $C_I\subseteq C$ be the set of all infectious customers, so $c_{I_{max}} \in C_I$. We will additionally let $D_{I}$ denote the minimum distance the tagged customer has to an infectious customer, i.e., the distance to the closest infectious customer such that $Q+1 - D_I = I_{\max}$.
\end{itemize}

We also define some additional notation for Sections~\ref{sec:mult-infec} and specifically~\ref{sec:moments}. 
\begin{itemize}
    \item Consider the event $\mathbf{1}_{\tilde{\mathcal{R}_{j}} > \ell}$ to be the event such that  customer $j$ in the queue has remaining sojourn time at least $\ell$ and that they are additionally the \textit{first} customer in the queue with remaining sojourn time at least $\ell$.
    \item Let $N_{ua}$ denote the number of infection transmissions that occur upon arrival to the queue in steady state. In Section~\ref{sec:mult-infec}, our perspective is of a tagged susceptible customer who is infected by those already in the queue upon their arrival. In Section~\ref{sec:moments} our perspective is of a tagged infectious customer who infects those already in the queue upon their arrival. We similarly define as $N_{aa}$ the number of new infections that occur when the tagged customer is either infected by those who arrived after them (Section~\ref{sec:mult-infec}), or infects those who arrive after them (Section~\ref{sec:moments}). Similarly, let $N_{I}$ be the total number of new infections transmitted respectively to or by the tagged customer.
    \item For the case of $N_{aa}$, we are especially interested in the number of new ``opposite customer type'' arrivals that occur after the tagged customer. Let $N_S(x)$ denote the number of susceptible arrivals over a time period $x$ and $N_I(x)$ the infectious arrivals ($N_S(x)  + N_I(x) = N(x)$). Then the associated arrivals are $A_1, \ldots A_{N_S}$ or $A_1, \ldots A_{N_I}$, with inter-arrival time between the $i^{th}$ and $(i+1)^{th}$ customers $A_{i+1} - A_i$, which are i.i.d.~exponential random variables with rate $\lambda (1-p)$ or $\lambda p$, respectively.
\end{itemize}

\subsection{Organization of paper}
The remainder of the paper is organized as follows.  In Section \ref{sec:stoc_model}, we describe the stochastic model that we will analyze in this paper.  We also describe overlap times and how they are relevant to understanding infection spread. In Section \ref{sec:prob-infec}, we then analyze the \emph{personalized} perspective of the paper by deriving an expression for the probability that a customer is infected in queue. In Section \ref{sec:mult-infec}, we compute the probability for a tagged susceptible customer of being infected by at least $k$ arrivals, offering insight into a way of understanding infection severity. Then, in Section \ref{sec:moments}, we compute the higher moments of the number of infections caused by one infectious arrival.  This provides a \emph{community} perspective of the impact of one tagged infectious customer. We also establish the relationship between the personalized and community perspectives.  In Section \ref{numerics_infections}, we provide some numerical experiments that validate our analytical formulas for the probability of infection, the mean number of infections and the variance of the number of infections, and extend our analysis to the $G/M/1$ and $M/G/1$ systems.  Finally, in Section \ref{conclusion}, we conclude and provide new directions of interest for further analysis.

\section{Stochastic Infection Model} \label{sec:stoc_model}

In this section, we will describe the infection dynamics of our queueing model in detail. Our model is based on an $M/M/1$ queue with an initial empty queue length, arrival rate $\lambda$, and service rate $\mu$. We assume that $\mu > \lambda$ so that our system has a steady state distribution.  To study the spread of infection within the queue, we partition the arrival rate into two customer classes, infectious and susceptible, based on their likelihood of being infected. Specifically, customers arriving with a probability of $p \in (0,1)$ are considered infectious, while those with a probability of $1-p$ are susceptible. We will refer to $p$ as the proportion of infectious arrivals of the system.

To model the spread of infection within the queue, we adopt a simple approach based on an overlap threshold of $\ell \in \mathbb{R}_+$ minutes. Specifically, we consider a susceptible customer to become infected if they overlap with an infectious customer for more than $\ell$ minutes. Note that our model assigns an infection transmission threshold to each \textit{susceptible} customer, rather than specifying an independent threshold for every susceptible-infectious pair, as in \cite{kang2021queueing}. This modelling choice captures customer-specific susceptibility, while avoiding the pairwise threshold assumption adopted in \cite{kang2021queueing}. As a result, our model remains sufficiently adaptable to accommodate heterogeneous infection risk across individuals, while retaining a more parsimonious and analytically tractable framework. We make the assumption that once a customer is infected, they do not further spread the infection within the queue. This assumption is reasonable as many infectious diseases are not contagious upon initial infection. By simulating the progression of infection within the queue, we can gain insights into how the infection rate, arrival rate, service rate, and overlap threshold affect the dynamics of the infection. Our model provides a powerful tool for understanding how infectious diseases spread in service systems.

\subsection{Overlap time distribution}


In this section, we consider the M/M/1 queue starting empty at time 0.  Let $A_n$ be the arrival time of the $n^{th}$ customer, $S_n$ the service time of the $n^{th}$ customer, and $W_n$ the waiting time of the $n^{th}$ customer. Thus, the departure time for the $n^{th}$ customer is given by the following equation
\begin{eqnarray*}
D_n &=& S_{n} + A_{n} + W_{n},
\end{eqnarray*}
Given the arrival time, departure time and waiting time for each customer, it is possible to describe the overlap time between any pair of customers.  The overlap time is defined as the amount of time that both customers ``overlap'' in the queue simultaneously.  Thus, the overlap time between the $n^{th}$ and $(n+k)^{th}$ customers in the single server queue is given by
\begin{eqnarray*}
O_{n,n+k} &=& \left( D_n - A_{n+k} \right)^+ \\ 
&=& \left( S_{n} + A_{n} + W_{n} - A_{n+k} \right)^+ \\
&=&  \left( S_{n} + W_{n} - ( A_{n+k} - A_{n} ) \right)^+ .
\end{eqnarray*}  
It's essential to highlight the distinction between the overlap time equation and the infinite server equation discussed by \citet{palomo2023overlap}. The reason for this distinction lies in the inherent ordering of a single server queue, which follows a \fcfs principle. This sequencing allows for a clear arrangement of departures based on their respective arrival times. Now that we possess a specific expression to calculate the overlap time between two customers, our subsequent discussion will elucidate how this information can be harnessed to compute infection probabilities for customers in the queue who are susceptible.


\section{Susceptible Customers' Probability of Infection}\label{sec:prob-infec}
In this section we explore the probability that a susceptible customer becomes infected while in the queueing system.

\subsection{Probability of infection from customer in queue upon arrival}

In this section, we compute the probability that a tagged susceptible customer will contract the infection from someone who is already in the queue upon their arrival.  This serves as a crucial initial step in our goal of computing the probability susceptible customers get infected in the system.  However, before we proceed to derive the probability of a tagged susceptible customer acquiring the infection from someone already in the queue at the moment of their arrival, we must first compute an important result that quantifies the smallest distance between an infectious customer and the arriving susceptible individual, given that $k$ customers are in the queue upon arrival.  The forthcoming lemma derives this probability.  

\begin{lemma}\label{lem:max_infec}
Suppose there are $k$ customers in the queue before the arrival of a tagged susceptible customer and $j \geq 1$ of them are infectious.  Then, the probability that the distance between the tagged arriving customer and the closest infectious customer is $j+i$ is given by
\begin{eqnarray*}
\mathbb{P} \left( k+1-D_{I} = j+i |Q=k, |C_I| = j\right) &=& \frac{\binom{j+i-1}{j-1}}{\binom{k}{j}}.
\end{eqnarray*}
\begin{proof}
 Suppose the tagged customer is susceptible, and when they arrive there are $k$ in the queue customers (so they are customer $k+1$), and $j$ of these customers are infectious. Then note that the number of (susceptible) customers between the tagged customer and the closest infectious customer (by definition customer $c_{I_{\max}}$) is $D_I$, such that we are finding the probability that $I_{\max} = j+i$. This requires that there are $j-1$ \textit{other} infectious customers than the closest one, who must all occupy positions between 1 and $j+i-1$ (inclusive) in the queue. It immediately follows that there are $j+i-1 \choose j-1$ possible ways these customers may be arranged in the queue.  Finally, we compute the denominator by using the Hockey Stick identity and summing the possible numerators.  This completes the proof.
\end{proof}
\end{lemma}

Lemma \ref{lem:max_infec}, provides the distribution for the distance between an arriving susceptible customer and the closest infectious customer.  This result is useful since it will allow us to compute the probability that the $j^{th}$ person waiting in a queue of length $k$ is the last infectious customer in the queue. The following lemma provides this probability.   

\begin{lemma}\label{lem:cond-max-infec} For an $M/M/1$ queue with queue length $k$, the probability for all $1 \leq j \leq k$ that the $j^{th}$ person is the last infectious customer in the queue is equal to
\begin{eqnarray*}
    \mathbb{P}(I_{\max} = j| Q = k) = p(1-p)^{k-j}.
\end{eqnarray*}
\begin{proof}

For $1 \leq m \leq k$, consider a scenario where a tagged susceptible customer arrives to a queue with $k$ customers, where $m$ of them are infectious. We designate customer $c_j$ as the closest infectious individual to the newly arrived susceptible customer, for $j \leq k$. Utilizing the pigeonhole principle, it's evident that we require $m \leq j$.

If we define $i = j - m$, then we can assert that the closest infectious customer is $c_{m+i}$.  Think of $i$ as a measure of this closest infectious customer's ``closeness." As $m$ represents the furthest spot in the queue from the tagged customer that the last infectious customer in the queue could occupy, $i$ quantifies the last infectious customer's additional proximity to the arriving susceptible customer.
Now, using Lemma~\ref{lem:max_infec}, we see that then \begin{eqnarray*}
    \mathbb{P}(I_{\max} = j| Q = k, |C_I| = m) &=& \mathbb{P}(I_{\max} = m+i| Q = k, |C_I| = m)\\ 
    &=& \mathbb{P}(k+1-D_I = m+i| Q = k, |C_I| = m)\\    &=& \frac{{m+i-1 \choose m-1}}{{k \choose m}}= \frac{{j-1 \choose m-1}}{{k \choose m}}.
\end{eqnarray*}

    Using the above observation,
    \begin{align*}
        \mathbb{P}(I_{\max} = j| Q = k)
        &= \sum_{m=1}^j\mathbb{P}(I_{\max} = j| Q = k, |C_I| = m)\mathbb{P}(|C_I|=m)\\
        &= \sum_{m=1}^j \frac{{j-1 \choose m-1}}{{k \choose m}}\cdot p^m(1-p)^{k-m}{k \choose m}\\
        &= \sum_{m=1}^j {j-1 \choose m-1}p^m(1-p)^{k-m}\\
        &= p(1-p)^{k-j}.
    \end{align*}
This completes the proof.
\end{proof}
\end{lemma}

Our final piece is to find the probability of the tagged susceptible customer getting infected by customers already in the queue upon their arrival given the queue length is $k$ and the index of the last infectious customer is equal to $j$.  This result uses Lemmas \ref{lem:max_infec} and \ref{lem:cond-max-infec} and is given below.  

\begin{lemma}\label{lem:cond-upon-infect} Consider a tagged susceptible arrival to the $M/M/1$ queue in steady state with queue length $k$, where for any $1\leq j \leq k$ the $j^{th}$ person waiting in the queue is the last infectious customer in the queue. The probability that the tagged customer gets infected during their time in the system is equal to
$$\mathbb{P}(I_{ua}|I_{\max} = j, Q = k) =\sum_{n=0}^{j-1} e^{-\mu l} \frac{(\mu l)^n}{n!}.$$
\end{lemma}

\begin{proof}
We begin by noting that for an infection to occur we require both (a) that the tagged customer has sojourn time $(\mathcal{S} = S+W)$ at least $\ell$ and (b) that some infectious customer is also in the system for at least $\ell$. 
    \begin{align*}
\mathbb{P}(I_{ua}|I_{\max} = j, Q = k) 
&= \mathbb{P}\left(S+W \geq l, \max_{i \in C_I}\sum_{m=1}^{i} R_m \geq l|I_{\max} = j, Q = k\right)\\ 
&= \mathbb{P}\left(S+\sum_{m=1}^{k}R_m \geq l, \max_{i \in C_I}\sum_{m=1}^{i} R_m \geq l|I_{\max} = j, Q= k\right) .
\intertext{Since $i \leq k, \forall i \in I$, and $S \geq 0$,}
&= \mathbb{P}\left(\max_{i \in C_I}\sum_{m=1}^{i} R_m \geq l|I_{\max} = j, Q = k\right)\\ 
&= \mathbb{P}\left(\sum_{m=1}^{j} R_m \geq l|I_{\max} = j, Q = k\right)\\ 
&= \mathbb{P}\left(\sum_{m=1}^{j} R_m \geq l\right)\\ 
&=\sum_{n=0}^{j-1} e^{-\mu l} \frac{(\mu l)^n}{n!}.
\end{align*}
\end{proof}

With all these pieces we can find the probability a susceptible customer is infected upon arrival. Since we have the conditional probability of a tagged susceptible customer getting infected, we can sum over the conditions: the distribution of the queue length and the index of the last infectious customer.  This is given below.

\begin{theorem}\label{thm:upon-arrival}
Let a susceptible customer join the $M/M/1$ queue in steady state, then the probability the tagged customer gets infected by a customer who is in queue upon their arrival is equal to
\begin{equation*}
\mathbb{P}\left( I_{ua} \right) = e^{-(\mu-\lambda) \ell} \frac{p \lambda}{\mu-\lambda(1-p)} .
\end{equation*}
\begin{proof}
\begin{align*}
\mathbb{P}(I_{ua}) &= \sum_{k=0}^{\infty} \mathbb{P}(Q =k) \cdot \mathbb{P}(I_{ua} |Q = k)\\
&= \sum_{k=1}^{\infty}\mathbb{P} \left(Q = k \right) \cdot \mathbb{P}\left(I_{ua} |Q = k\right)\\
&= \sum_{k=1}^{\infty}\mathbb{P}\left(Q = k \right) \cdot \sum_{j=1}^{k}\mathbb{P}\left(I_{ua} | I_{\max} = j, Q = k)\mathbb{P}(I_{\max} = j| Q = k \right).
\intertext{By Lemmas~\ref{lem:cond-max-infec} and~\ref{lem:cond-upon-infect}, respectively, we recognize the above expression to be}
&= e^{-\mu \ell}\left(1-\frac{\lambda}{\mu}\right)\sum_{k=1}^{\infty}\left(\frac{\lambda}{\mu}\right)^k\sum_{j=1}^{k}p(1-p)^{k-j}\sum_{n=0}^{j-1}\frac{(\mu \ell)^n}{n!}\\
&= e^{-\mu \ell} \left(1-\frac{\lambda}{\mu}\right)\sum_{n=0}^{\infty}\frac{(\mu \ell)^n}{n!} \sum_{k=n+1}^{\infty}\left(\frac{\lambda}{\mu} \right)^k\sum_{j=n+1}^{k}p(1-p)^{k-j} \\
&= e^{-\mu \ell}\left(1-\frac{\lambda}{\mu}\right)\sum_{n=0}^{\infty}\frac{(\mu \ell)^n}{n!} \sum_{k=n+1}^{\infty}\left(\frac{\lambda}{\mu}\right)^k\left(1-(1-p)^{k-n}\right)\\
&= e^{-\mu \ell} \left(1-\frac{\lambda}{\mu}\right)\sum_{n=0}^{\infty}\frac{(\mu \ell)^n}{n!}\frac{p\left(\frac{\lambda}{\mu}\right)^{n+1}}{(1-\frac{\lambda}{\mu})(p\frac{\lambda}{\mu}-\frac{\lambda}{\mu}+1)} \\
&= e^{-\mu \ell}\left(1-\frac{\lambda}{\mu}\right)\frac{p\lambda\mu}{(\mu -\lambda)(p\lambda-\lambda+\mu)}\sum_{n=0}^{\infty}\frac{(\lambda \ell)^n}{n!} \\
&= e^{-(\mu-\lambda) \ell}\frac{p\lambda}{\mu-\lambda(1-p)} .
\end{align*}
\end{proof}
\end{theorem}

We observe in Theorem~\ref{thm:upon-arrival} that the probability for a susceptible customer of getting infected upon arrival is equal to the product of two probabilities. 

The first probability measures if the tagged customer's sojourn time is larger than $\ell$ i.e. $\mathbb{P}(\mathcal{S} > \ell) = e^{-(\mu-\lambda) \ell} $.  This is necessary as if their sojourn time is not longer than $\ell$, then it is impossible for them to get infected. The second probability is $\mathbb{P}(\mathcal{S} > \text{Arrival of infectious customer} ) = \frac{\lambda p}{(\mu-\lambda) + \lambda p} $.

\subsection{Probability of infection by customer who arrives after}
In this section, we compute the probability that a susceptible customer gets infected by someone who arrives to the queue after their arrival. This is the second important piece in computing the probability of getting infected while in the queue. And specifically note at this point that we are \textit{not} restricting our consideration only to those susceptible customers who were not already infected upon arrival. That is, recall that these two events are not mutually exclusive. 

In computing this probability, we find it easier to compute its complement, that is the likelihood that a tagged susceptible customer will not overlap with an infectious customer who arrives after them long enough to get infected. In doing so, we begin by recognizing that if the first infectious customer after the tagged susceptible customer arrives while the tagged customer has less than $\ell$ of their sojourn remaining, the tagged customer will not get infected. The following lemma computes this useful probability.

\begin{lemma}
\label{lem:integral-prob}
 Consider a susceptible customer arriving to an $M/M/1$ queue in steady state. If $k$ infectious customers arrive while the tagged customer is in the system (that is to say, within the interval $[0,x]$), the probability that the first arrives while the tagged customer has less than $\ell$ sojourn remaining is equal to
    $$\mathbb{P}(A_1>x-\ell|M_I=k) = \left(\frac{\ell}{x}\right)^k.$$
\end{lemma}
\begin{proof}
        We find it easier to calculate the converse, $$\mathbb{P}(A_1>x-\ell|M_I=k) =1 - \mathbb{P}(A_1 \leq x- \ell | M_I = k).$$ By conditioning on $M_I = k$, we now have that the arrival times are distributed uniformly on $[0,x]$. We use the distribution of the CDF for the minimum order statistic of $X \sim U(0,x)$, which is 
        \begin{eqnarray*}
            \sum_{i=1}^k{k \choose i}F_X(x-\ell)^i(1-F_X(x-\ell))^{k-i} &=& \sum_{i=1}^k{k \choose i}\left(\frac{x-\ell}{x}\right)^i\left(\frac{\ell}{x}\right)^{k-i} \\
            &=& 1-\left(\frac{\ell}{x}\right)^k.
             \end{eqnarray*}
It is worth confirming that when $k=0$, that is there are no infectious arrivals within a time window of $x$ after the tagged customer arrived, then $\mathbb{P}(A_1>x-\ell)=1-\left(1-\left(\frac{\ell}{x}\right)^0\right) = 1$.
\end{proof}

We can now consider the conditional probability of the tagged susceptible customer avoiding getting infected by customers who arrive to the queue after their arrival, given a queue length of $k$ upon their arrival. The result is given below.

\begin{lemma}
 \label{lem:splitCompQ}
 Consider a susceptible arrival to an $M/M/1$ queue in steady state when the queue length is $k$. We can split the complement of the after arrival infection probability into two integrals, specifically,
\begin{align*}
\mathbb{P}(I_{aa}^c|Q=k)
=& \int_0^{\ell}\frac{\mu^{k+1}x^ke^{-\mu x}}{k!}\sum_{j=0}^{\infty}\frac{
 e^{-\lambda p x } (\lambda p x )^j }{j!}
 dx\\
&+\int_{\ell}^{\infty}\frac{\mu^{k+1}x^ke^{-\mu x}}{k!}\sum_{j=0}^{\infty}\left(\frac{\ell}{x}\right)^j\frac{
 e^{-\lambda p x } (\lambda p x )^j }{j!}dx.
\end{align*}
\end{lemma}
\begin{proof}
We first recognize that the infection probability depends on the number of infectious customers who overlap with the tagged susceptible customer, as in the upon arrival case, and this number of infection-transmitting after arrival overlaps is dependent on the tagged customer's total time in the system, so we find that 
\begin{align*}
   \lefteqn{ \mathbb{P}(I_{aa}^c|Q=k) } \\
    =& \int_0^{\infty}\mathbb{P}(\mathcal{S}= x|Q=k)\sum_{j=0}^{\infty}\mathbb{P}(I_{aa}^c|M_I = j,Q=k,\mathcal{S}=x)\mathbb{P}(M_I=j|\mathcal{S}=x,Q=k)dx.\intertext{Since the queue is single-server, the tagged customer does \textit{not} get infected by an arrival after them only if either (a) they are not in the system long enough to get infected, or (b) the first infectious arrival after them occurs too late to infect them (thereby meaning \textit{all} infectious customers arriving after the tagged customer arrive too late to infect the tagged customer).}
    =& \int_0^{\infty}\mathbb{P}(\mathcal{S} = x|Q=k)\sum_{j=0}^{\infty}\mathbb{P}(A_{1}>x-l\cup x < \ell|M_I=j)\mathbb{P}(M_I=j|\mathcal{S}=x)dx.\intertext{The probability that the susceptible customer does not get infected is trivially 1 if the customer's sojourn time is less than $\ell$ minutes. So,}
    =& \int_0^{\ell}\mathbb{P}(S = x|Q=k)\sum_{j=0}^{\infty}\mathbb{P}(M_I=j|S=x,Q=k)dx\\
    &+\int_{\ell}^{\infty}\mathbb{P}(S = x|Q=k)\sum_{j=0}^{\infty}\mathbb{P}(A_1>x-l|M_I=j)\mathbb{P}(M_I=j|S=x,Q=k)dx\\
 =& \int_0^{\ell}\frac{\mu^{k+1}x^ke^{-\mu x}}{k!}\sum_{j=0}^{\infty}\frac{
 e^{-\lambda p x } (\lambda p x )^j }{j!}
 dx\\
&+\int_{\ell}^{\infty}\frac{\mu^{k+1}x^ke^{-\mu x}}{k!}\sum_{j=0}^{\infty}\mathbb{P}(A_1>x-l|M_I=j)\frac{
 e^{-\lambda p x } (\lambda p x )^j }{j!}dx.\intertext{And by Lemma~\ref{lem:integral-prob}, this is}
 =& \int_0^{\ell}\frac{\mu^{k+1}x^ke^{-\mu x}}{k!}\sum_{j=0}^{\infty}\frac{
 e^{-\lambda p x } (\lambda p x )^j }{j!}
 dx + \int_{\ell}^{\infty}\frac{\mu^{k+1}x^ke^{-\mu x}}{k!}\sum_{j=0}^{\infty}\left(\frac{\ell}{x}\right)^j\frac{
 e^{-\lambda p x } (\lambda p x )^j }{j!}dx.
    \end{align*}
\end{proof}

In completing our result for the unconditional after arrival infection probability, we encounter two large computations. For the sake of clarity, we separately solve these computations, allowing us to cleanly substitute them in the final proof. 

\begin{lemma}\label{lem:serviceShort}
The following identity holds for all $\lambda < \mu$:
    $$\left(\frac{\mu-\lambda}{\mu}\right)\sum_{k=0}^{\infty} \left(\frac{\lambda}{\mu}\right)^k\int_0^{\ell}\frac{\mu^{k+1}x^ke^{-\mu x}}{k!}\sum_{j=0}^{\infty}\frac{
 e^{-\lambda p x } (\lambda p x )^j }{j!}
 dx=1-e^{-(\mu-\lambda) \ell}.$$
\end{lemma}

\begin{proof}
        \begin{align*}
    &\left(\frac{\mu-\lambda}{\mu}\right)\sum_{k=0}^{\infty} \left(\frac{\lambda}{\mu}\right)^k\int_0^{\ell}\frac{\mu^{k+1}x^ke^{-\mu x}}{k!}\sum_{j=0}^{\infty}\frac{e^{-\lambda p x } (\lambda p x )^j }{j!}
dx\\
    &= \left(\mu-\lambda\right)\int_0^{\ell}e^{-(\mu+\lambda p) x}\sum_{k=0}^{\infty} \frac{(\lambda x)^k}{k!}\sum_{j=0}^{\infty}\frac{ (\lambda p x )^j }{j!}dx\\
    &= \left(\mu-\lambda\right)\int_0^{\ell}e^{-(\mu+\lambda p) x}e^{\lambda x}e^{\lambda p x}dx\\
    &= \left(\mu-\lambda\right)\int_0^{\ell}e^{-(\mu-\lambda) x}dx\\
    &= 1-e^{-(\mu-\lambda) \ell}.
    \end{align*}
\end{proof}

\begin{lemma}\label{lem:serviceLong}
The following identity holds for all $\lambda < \mu$:
    \begin{align*}
        &\left(\frac{\mu-\lambda}{\mu}\right)\sum_{k=0}^{\infty} \left(\frac{\lambda}{\mu}\right)^k\int_{\ell}^{\infty}\frac{\mu^{k+1}x^ke^{-\mu x}}{k!}\sum_{j=0}^{\infty}\left(\frac{\ell}{x}\right)^j\frac{e^{-\lambda p x } (\lambda p x )^j }{j!}dx=\frac{\mu-\lambda}{\mu-(1-p)\lambda}e^{-(\mu-\lambda) \ell}.
    \end{align*}
\end{lemma}

\begin{proof}
\begin{align*}
    &\left(\frac{\mu-\lambda}{\mu}\right)\sum_{k=0}^{\infty} \left(\frac{\lambda}{\mu}\right)^k\int_{\ell}^{\infty}\frac{\mu^{k+1}x^ke^{-\mu x}}{k!}\sum_{j=0}^{\infty}\left(\frac{\ell}{x}\right)^j\frac{e^{-\lambda p x } (\lambda p x )^j }{j!}dx\\
    &=\left(\mu-\lambda\right)\int_{\ell}^{\infty}e^{-(\mu+\lambda p) x}\sum_{k=0}^{\infty} \frac{(\lambda x)^k}{k!}\sum_{j=0}^{\infty}\frac{(\lambda p \ell )^j }{j!}dx\\
    &=\left(\mu-\lambda\right)\int_{\ell}^{\infty}e^{-(\mu+\lambda p) x}e^{\lambda x}e^{\lambda p \ell}dx\\
    &=\left(\mu-\lambda\right)e^{\lambda p \ell}\int_{\ell}^{\infty}e^{-(\mu-(1-p)\lambda) x}dx\\
    &=\left(\mu-\lambda\right)e^{\lambda p \ell}\frac{e^{-(\mu-(1-p)\lambda) \ell}}{\mu-(1-p)\lambda}\\
    &=\frac{\mu-\lambda}{\mu-(1-p)\lambda}e^{-(\mu-\lambda) \ell}.
\end{align*}

\end{proof}

We are now able to sum over the distribution of the queue length to leverage Lemma~\ref{lem:splitCompQ} in solving for the full complement of the after arrival probability. We also employ Lemmas~\ref{lem:serviceShort} and~\ref{lem:serviceLong} for the sake of brevity within the proof of the following theorem.

\begin{theorem}\label{thm:after-arrival}
Let a susceptible customer join the $M/M/1$ queue in steady state, then the probability the customer gets infected by a customer who joins the queue after their arrival is equal to
$$\mathbb{P}(I_{aa})=\frac{p\lambda}{\mu-\lambda(1-p)}e^{-(\mu-\lambda) \ell}.$$
\end{theorem}
\begin{proof}

\begin{align*}
    \mathbb{P}(I_{aa})
    =& 1-\mathbb{P}(I_{aa}^c)\\
    =& 1-\sum_{k=0}^{\infty} \mathbb{P}(Q=k)\mathbb{P}(I_{aa}^c|Q=k)\\
    =& 1-\sum_{k=0}^{\infty}\left(\frac{\lambda}{\mu}\right)^k\left(1-\frac{\lambda}{\mu}\right)\mathbb{P}(I_{aa}^c|Q=k.)\intertext{Using Lemma~\ref{lem:splitCompQ}, we split this probability into two integrals}
    =&1-\frac{\mu-\lambda}{\mu}\sum_{k=0}^{\infty}\left(\frac{\lambda}{\mu}\right)^k\int_0^{\ell}\frac{\mu^{k+1}x^ke^{-\mu x}}{k!}\sum_{j=0}^{\infty}\frac{ e^{-\lambda p x } (\lambda p x )^j }{j!}dx\\
    &-\frac{\mu-\lambda}{\mu}\sum_{k=0}^{\infty}\left(\frac{\lambda}{\mu}\right)^k\int_{\ell}^{\infty}\frac{\mu^{k+1}x^ke^{-\mu x}}{k!}\sum_{j=0}^{\infty}\left(\frac{\ell}{x}\right)^j\frac{
    e^{-\lambda p x } (\lambda p x )^j }{j!}dx.\intertext{We employ Lemma~\ref{lem:serviceShort} and Lemma~\ref{lem:serviceLong} to substitute for these first and second summations, respectively, thus getting,}
    =& 1-\left(1-e^{-(\mu-\lambda) \ell}
 +\frac{\mu-\lambda}{\mu-\lambda(1-p)}e^{-(\mu-\lambda) \ell}\right)\\
 =& \left(1-\frac{\mu-\lambda}{\mu-\lambda(1-p)}\right)e^{-(\mu-\lambda) \ell}\\
 =& \frac{p\lambda}{\mu-\lambda(1-p)}e^{-(\mu-\lambda) \ell} .
    \end{align*}
\end{proof}

We immediately note that this shows us symmetry around the customers who arrive before and after a tagged arrival (and overlap with the tagged arrival), as we see in Theorems~\ref{thm:upon-arrival} and~\ref{thm:after-arrival} that the probabilities of a susceptible arrival being infected by a customer already in the queue upon their arrival and of being infected by a customer who arrives after them are equal. This observation inspires a conjecture for all systems which satisfy a \fcfs requirement.

\begin{conjecture}\label{conj:symmetry}
    For the $G/G/1$ queueing system with the \fcfs discipline, the probabilities of infection upon and after arrival are equivalent. Further, given a known sojourn time we conjecture that the events of infection upon and after arrival are independent and identically distributed.
\end{conjecture}

While we have shown this to be true for the $M/M/1$ system, our conjecture centers around the idea that this is more generally true, so long as the discipline is \fcfs. The intuition behind this conjecture under this discipline is that the probability of infection in either case is understood to be the probability that for a tagged customer, within all overlaps at or surpassing the threshold of infection, at least one is infectious. This is then, for a known number of overlaps, a binomial probability, clearly equivalent without dependence on whether the case is upon or after arrival. As such, this conjecture is that the number of such overlaps is equivalent in the upon and after arrival cases.

The further clause follows immediately from the independence of the conditional events so long as the initial conjecture holds.

\subsection{Calculating the probability of getting infected}

We now wish to find the probability that a customer becomes infected either upon or after arrival (again, recall these events are not disjoint). In doing so, we leverage the conditional independence of the two cases given a fixed sojourn time. To that end, a few lemmas about a fixed sojourn time are useful.

Firstly, we recognize that the sojourn time in steady state is exponentially distributed with parameter $\mu-\lambda.$

\begin{lemma}\label{lem:prob-sojourn-x}
A customer arriving to the $M/M/1$ queue in steady state has an exponential sojourn time probability density function i.e.
    $$f_{\mathcal{S}}(x) = (\mu-\lambda)e^{-(\mu-\lambda) x}.$$
\end{lemma}
\begin{proof}
The proof is well known in the queueing literature.
\end{proof}

A second lemma demonstrates that, given the sojourn time $x$, the length of the queue follows a Poisson distribution with parameter $\lambda x$.
\begin{lemma}\label{lem:queue-given-sojourn}
For a customer arriving to the $M/M/1$ queue in steady state with sojourn time $x$, the likelihood that there are $r$ customers in the system at their time of arrival is Poisson,
    $$\mathbb{P}(Q=r|\mathcal{S} = x) = \frac{e^{-\lambda x}(\lambda x)^r}{r!}.$$
It is additionally convenient for our purposes to note that this result holds identically for $\mathbb{P}(Q=r|\mathcal{R} = x)$.
\end{lemma}

\begin{proof}
The proof of this result appears in \citet{ross2014introduction}.  However, we have included a proof for the reader's convenience.
    \begin{align*}
        \mathbb{P}(Q=r|\mathcal{S} = x) 
        &= \frac{\mathbb{P}(\mathcal{S}=x|Q=r)\mathbb{P}(Q=r)}{\mathbb{P}(\mathcal{S} = x)}\intertext{So by Lemma~\ref{lem:prob-sojourn-x},}
        &= \frac{\frac{\mu^{r+1}x^re^{-\mu x}}{r!}\left(\frac{\lambda}{\mu}\right)^r\left(1-\frac{\lambda}{\mu}\right)}{(\mu-\lambda)e^{-(\mu-\lambda) x}}\\
        &=\frac{e^{-\lambda x}(\lambda x)^r}{r!}.
    \end{align*}
By the memoryless property of the exponential inter-arrival distribution, we find that $\mathbb{P}(\mathcal{R}=x)=\mathbb{P}(\mathcal{S} = x)$, so this proof demonstrates that the result holds identically for both remaining and total sojourn time.
\end{proof}

The insight that the arrival process is a Poisson process allows us to utilize the memoryless property of the exponential distribution (increments of the Poisson are exponential) to only consider the ``snapshot'' of the remaining queue upon a tagged customer's arrival after the infection transmission threshold $\ell$ has passed. In doing so, we are able to restrict exploring the upon arrival case to only those customers who we know will remain in the system long enough to infect the tagged customer.

Leveraging this insight, We now find it useful to write a small lemma regarding the upon arrival case under a fixed sojourn time.

\begin{lemma}\label{lem:closest-infec-upon-cant-infect}
For a susceptible arrival to the $M/M/1$ queue in steady state with sojourn time $x$, the probability that there are no infectious customers in the queue upon arrival with remaining sojourn time long enough to infect them is,
    $$\mathbb{P}(\mathcal{R}_{I_{\max}} < \ell|\mathcal{S} = x) = e^{-\lambda p(x-\ell)}.$$
\end{lemma}

\begin{proof}
We recognize that $\mathcal{R}_{I_{\max}}$ is the remaining sojourn time of the closest infectious customer to the tagged customer. We also note from Lemma~\ref{lem:queue-given-sojourn} that for fixed sojourn time $x$, queue length is Poisson. Thus, we utilize the memoryless property of the exponential distribution to recognize that this is the same as confirming that when the tagged customer has remaining sojourn time $x-\ell$ there are no infectious customers ahead of them in the queue.

    \begin{align*}
        \mathbb{P}(\mathcal{R}_{I_{\max}} < \ell|\mathcal{S} = x) 
        &= \sum_{k=0}^{\infty}\mathbb{P}(Q =k |\mathcal{R} = x-\ell)\mathbb{P}(|C_I| = 0|Q =k)\\
        &= \sum_{k=0}^{\infty}e^{-\lambda(x-\ell)}\frac{(\lambda(x-\ell))^k}{k!}(1-p)^k\\
        &= e^{-\lambda(x-\ell)}\sum_{k=0}^{\infty}\frac{(\lambda(1-p)(x-\ell))^k}{k!}\\
        &= e^{-\lambda p(x-\ell)}.
    \end{align*}
\end{proof}

With this final piece, we now can determine the total probability a susceptible arrival gets infected while in the system.

\begin{theorem} \label{thm:total-infection-prob}
Let $I = I_{ua} \cup I_{aa}$ be the event that a susceptible customer gets infected while in the queue. Suppose a susceptible customer joins the $M/M/1$ queue in steady state, then the probability the customer gets infected is equal to
\begin{eqnarray}\label{eqn:prob-infec}
    \mathbb{P}(I) &=&  \mathbb{P}(I_{ua} \cup I_{aa}) =  \frac{ 2 \lambda p}{\mu - \lambda + 2 \lambda p}  e^{- (\mu - \lambda)  \ell} .
\end{eqnarray}
\end{theorem}
\begin{proof}
\begin{align*}
    \mathbb{P}(I_{ua} \cup I_{aa}) &=
    1 - \mathbb{P}(I_{ua}^c \cap I_{aa}^c)\\
    &= 1- \Big(\mathbb{P}(I_{ua}^c \cap I_{aa}^c | \mathcal{S} < \ell ) \mathbb{P}( \mathcal{S} < \ell ) + \mathbb{P}(I_{ua}^c\cap I_{aa}^c | \mathcal{S} \geq \ell ) \mathbb{P}( \mathcal{S} \geq \ell )\Big) \\
    &= 1-\Big(\mathbb{P}( \mathcal{S} < \ell ) + \mathbb{P}(I_{ua}^c \cap I_{aa}^c | \mathcal{S} \geq \ell ) \mathbb{P}( \mathcal{S} \geq \ell )\Big) \\
    &=  1-\Big(1 - e^{-(\mu - \lambda) \ell} + \mathbb{P}(I_{ua}^c \cap I_{aa}^c , \mathcal{S} \geq \ell )\Big)  \intertext{We note that $I_{ua}^c$ and $I_{aa}^c$ are independent given $\mathcal{S} =x$. For fixed sojourn time $x$, we also recognize that \textit{not} getting infected upon arrival is dependent on the closest infectious customer having remaining sojourn time less than $\ell$. Similarly, \textit{not} getting infected after arrival is dependent on no infectious customers arriving within $x-\ell$.}
    &=  e^{-(\mu - \lambda) \ell} - \left( \int^{\infty}_{\ell} \mathbb{P}(\mathcal{R}_{I_{\max}} < \ell |\mathcal{S} =x) \cdot \mathbb{P}(N_I(x-\ell ) = 0) \cdot  \mathbb{P}(\mathcal{S} =x) dx\right)  \intertext{By Lemmas~\ref{lem:closest-infec-upon-cant-infect} and~\ref{lem:prob-sojourn-x},}
    &=  e^{-(\mu - \lambda) \ell} - \left( \int^{\infty}_{\ell} e^{-2\lambda p (x-\ell)}  \cdot  (\mu - \lambda) e^{-(\mu - \lambda) x} dx \right)  \\
    &= e^{-(\mu - \lambda) \ell} - e^{2 \lambda p \ell} \left( \int^{\infty}_{\ell} e^{-(2\lambda p + \mu - \lambda ) x}  \cdot  (\mu - \lambda) dx \right)  \\
    &=  e^{-(\mu - \lambda) \ell} - e^{- (\mu - \lambda)  \ell} \frac{ \mu - \lambda}{\mu - \lambda + 2 \lambda p}   \\
    &= \frac{ 2 \lambda p}{\mu - \lambda + 2 \lambda p}e^{- (\mu - \lambda)  \ell}.
\end{align*}
\end{proof}

Note that this result is intuitive. We get that the combined probability for a susceptible customer of getting infected by \textit{both} a customer in the system upon their arrival \textit{and} by a customer who arrives after them is the product of two probabilities: the probability of the tagged customer being in the system for the infection transmission threshold and the probability that the exponential of the tagged customer's pass-through rate is more than than the exponential of twice the rate of infectious arrivals to the system (twice, since we have the two cases of those there upon the tagged customer's arrival and those who arrive while the tagged customer remains in the system). This also demonstrates that the coupling of the upon- and after-arrival probabilities of infection occurs through the shared requirement that the tagged customer's sojourn time is over the infection transmission threshold.

One should note that this formula does not match with the expression given in \citet{kang2021queueing} since our formula is exact and is not an asymptotic expression.  Moreover, this result is close to an approximation that was developed using the maximum overlap time tail distribution as Theorem 3.3 in \citet{palomo2023maximum}.

\begin{remark}
    From a practical point of view, a manager of a service system might be interested in understanding what arrival rate would induce a probability of infection $\epsilon$. This has no closed form, but one can give an approximation based on heavy or low traffic.

    For a low traffic setting, where $\lambda \ll \mu$, we can approximate the denominator of Equation \ref{eqn:prob-infec} as $\mu$, such that we have \begin{align*}
        \epsilon &\approx \frac{2p\lambda}{\mu}e^{-(\mu-\lambda)\ell}\intertext{whose terms we rearrange to get,}
        \ell\frac{\mu}{2p} e^{\mu\ell} \epsilon&\approx \ell \lambda e^{\lambda\ell}.\intertext{As the Lambert $W$ function is the inverse of $f(w) = w e^w$, we can thus rewrite this as}
        \lambda(\epsilon) &\approx \frac{1}{\ell}W\left(\frac{\mu \ell \epsilon}{2p} e^{\mu \ell}\right),\intertext{where $W$ is the Lambert $W$ function.}
    \end{align*}

    As the exponential term dominates Equation~\ref{eqn:prob-infec}, this approximation is useful while treating the denominator as $\mu$ is reasonable. In contrast, once we assume $\lambda \approx \mu$, we find it more useful to define $\delta = \mu-\lambda$ such that $\delta \ll \mu$, so that we can rewrite Equation~\ref{eqn:prob-infec} as
    \begin{align*}
        \epsilon &\approx \frac{2p(\mu-\delta)}{\mu +(2p-1)(\mu-\delta)}e^{-\delta\ell}\\
        \epsilon &\approx \frac{2p(\mu-\delta)}{2p(\mu -\delta)+\delta}e^{-\delta\ell}\intertext{The Taylor expansions of the scalar term, $1-\frac{\delta}{2p\mu}-\frac{(2p-1)\delta^2}{(2p\mu)^2}-\frac{(2p-1)^2\delta^3}{(2p\mu)^3}+\cdots$, and of the exponential term, $1-\delta\ell + \frac{(\delta\ell)^2}{2}-\frac{(\delta\ell)^3}{3}+\cdots$, let us rewrite the approximation as}
        \epsilon &\approx 1 - \delta\ell -\frac{\delta}{2p\mu}+O\left(\delta^2\right)\intertext{which we can rearrange to isolate $\delta$ as}
        \delta &\approx \frac{2p\mu(1-\epsilon)}{1+2p\mu\ell}\intertext{So substituting back for $\lambda$ we get our high utility approximation as}
        \lambda(\epsilon) &\approx \mu\left(1-\frac{2p(1-\epsilon)}{1+2p\mu\ell}\right)
    \end{align*}

    Note for these two approximations that their accuracy improves as either $\lambda$ (normalized by $\mu$) or $\delta$, respectively, becomes smaller.
    \end{remark}

\subsection{Numerical results for infection probability}

In this section, we provide some plots for the infection probability to gain some insights about our results.  In Figure \ref{fig:fig_1}, we observe that that for low system utilization, i.e.~small $\rho=\frac{\lambda}{\mu}$, the curve is nearly linear in $p$, and has a low maximum. As the utilization increases, the infection probability increases rapidly even for a low proportion of infectious arrivals of the system.

    \begin{figure}[h]
    \centering
    \includegraphics[width=0.7\textwidth]{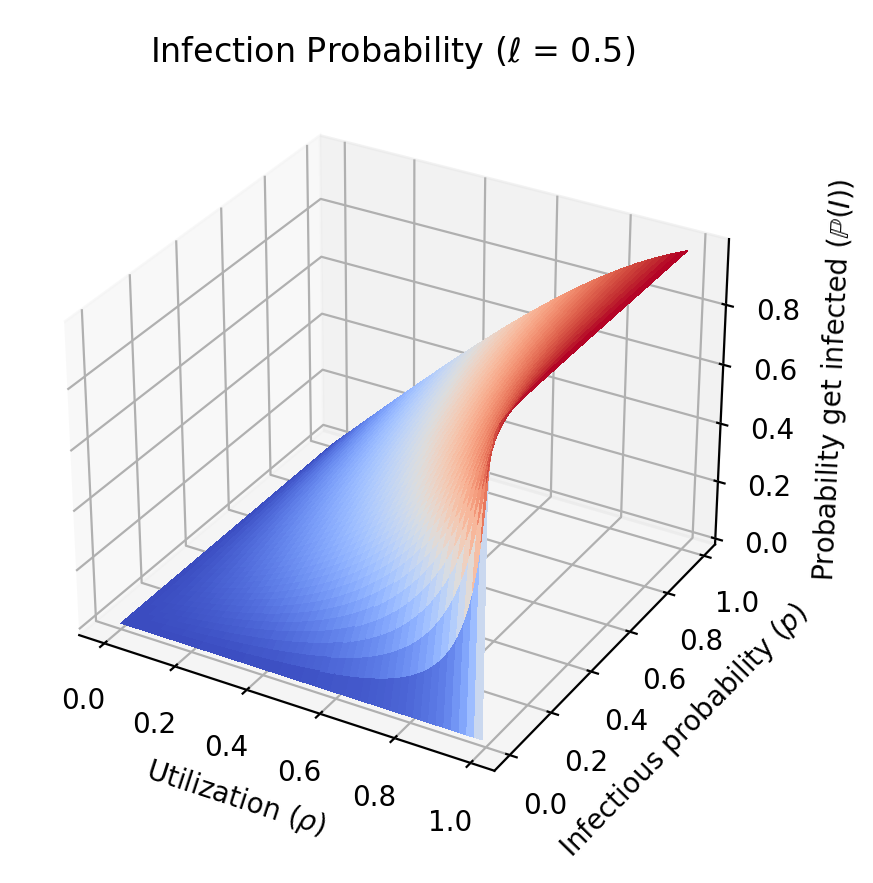}
    \caption{The total probability of becoming infected in queue depending on utilization $\rho = \frac{\lambda}{\mu}$ and proportion of infectious arrivals $p$. }
    \label{fig:fig_1}
\end{figure}

In addition, we observe that the infection probability is invariant to scaling both $\lambda$ and $\mu$ equivalently provided that we also scale the threshold $\ell$ by the inverse of the change.  This implies that if we make $\lambda \to \lambda n $, $\mu \to \mu n $, then we need to scale $\ell \to \frac{\ell}{n}$.  We observe this phenomenon in Figure \ref{fig:scale-var}.  On the left of Figure \ref{fig:scale-var}, we observe that not scaling the threshold yields very small probabilities of infection as higher rates correspond to shorter sojourn times.  On the right of Figure \ref{fig:scale-var}, we observe that if we also scale the threshold appropriately, the infection probability is invariant across different values of scaling.  

\begin{figure}[H] \label{fig:with-var}
    \centering
\includegraphics[width=0.5\textwidth]{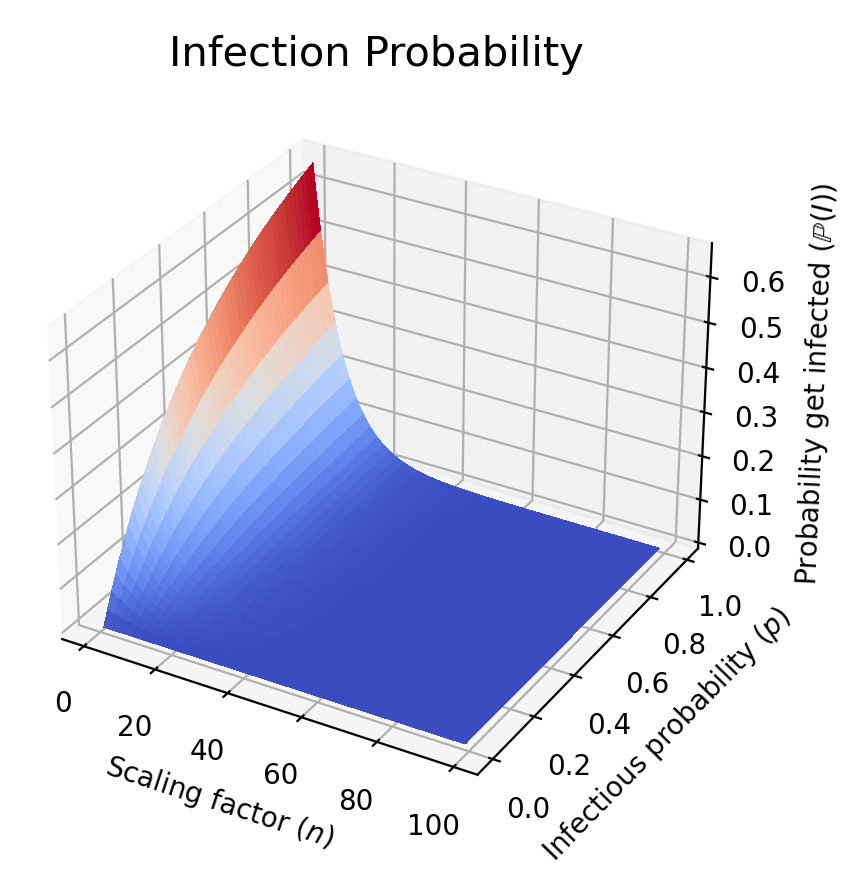}~\includegraphics[width=0.5\textwidth]{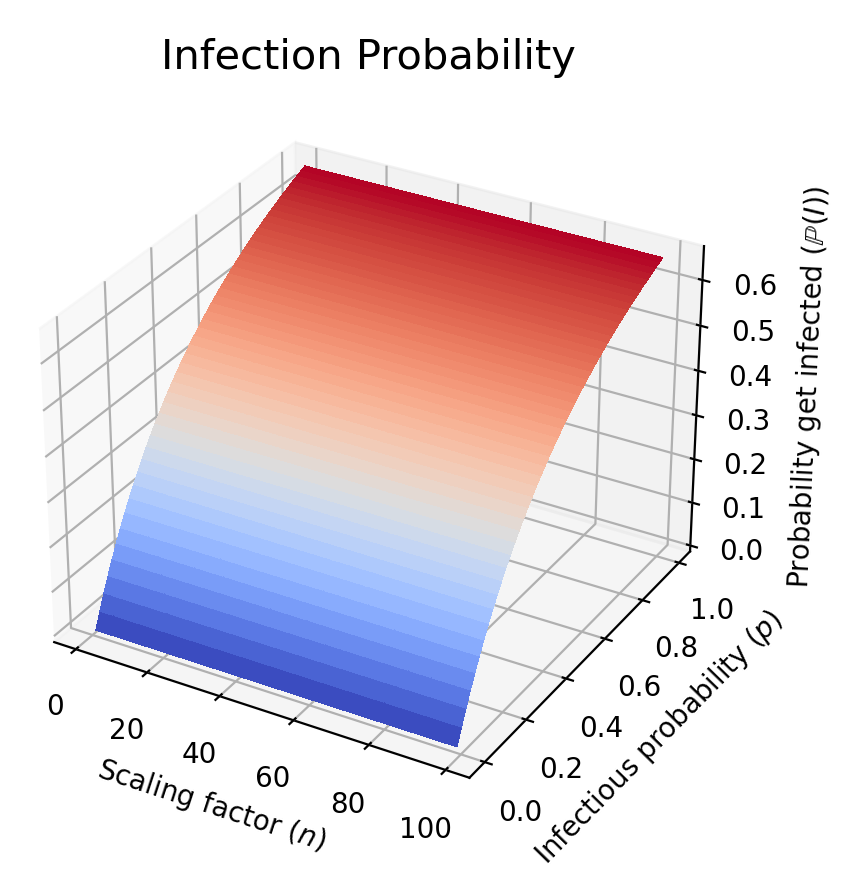}
\caption{Scaling $\lambda$ and $\mu$.  (Left) No Scaling of Threshold.  (Right) Scaling of threshold.  }
\label{fig:scale-var}
\end{figure}

\begin{remark}
   A key observation is that our analytical expressions are initially derived for a 
\emph{fixed} infection transmission threshold $\ell$.  However, this formulation is far more 
general than it may appear at first glance.  If the infection transmission threshold is 
allowed to be random for each susceptible customer, with distribution function $F_\Theta(\ell)$, then the corresponding 
performance measures can be obtained simply by integrating our fixed-threshold 
quantities with respect to $F_\Theta$.  In particular, if $G(\ell)$ denotes any of 
the fixed-threshold expressions derived in our analysis, then the random-threshold 
version is given by
\[
\mathbb{E}\!\left[ G(\Theta) \right] 
    = \int_{0}^{\infty} G(\ell)\, dF_\Theta(\ell).
\]
Because $G(\ell)$ always appears in the form of an exponential evaluated at 
$\mu - \lambda$, this integration step reduces to evaluating the 
Laplace-Stieltjes transform of the threshold distribution at the argument 
$\mu - \lambda$.  That is,
\[
\mathbb{E}\!\left[ e^{-(\mu - \lambda)\Theta} \right]
    = \mathcal{L}_\Theta(\mu - \lambda),
\]
so the random-threshold expressions follow directly by replacing the exponential 
term associated with the fixed threshold by the Laplace-Stieltjes transform of 
the threshold distribution. Consequently, our framework admits 
\emph{arbitrary threshold distributions} for each tagged susceptible customer with no additional structural 
assumptions. This is a manner in which our work can be generalized beyond what is done in \citet{kang2021queueing}. While \citet{kang2021queueing} allows for a pairwise infection threshold for each pair of susceptible-infectious customers it is worth noting two things. Firstly, that their analysis assumes only a single infectious customer in the system at a time, such that it is not substantially different from our consideration of a random threshold for a tagged susceptible customer. Secondly, we conjecture that the random pairwise and random per tagged customer infection threshold cases are not greatly dissimilar, see Tables~\ref{tab:compareaveragemeans} and \ref{tab:compareaveragevars}.
Our approach seamlessly extends to \emph{any} 
distribution, for example, deterministic, Erlang, hyperexponential, heavy-tailed, 
or empirical threshold models merely by substituting the appropriate 
Laplace-Stieltjes transform.  Hence, our results provide a unified treatment 
that accommodates a much broader class of infection mechanisms. 
\end{remark}

\section{Multiple Infections}\label{sec:mult-infec}

We now consider $N_{ua}$ as being the number of infectious customers a susceptible arrival is infected by upon arrival and $N_{aa}$ the number of infectious customers who arrive after the tagged customer and infect them. As a way of exploring the severity of an infection, we want to explore the probability of a tagged customer getting infected by some specific number, $k$, other customers that we can imagine represents the amount needed for a higher degree of infection.

\begin{figure}[H]
\centering
\makebox[\linewidth][c]{
\begin{subfigure}[b]{0.6\textwidth}
        \centering
    \includegraphics[width=.98\textwidth]{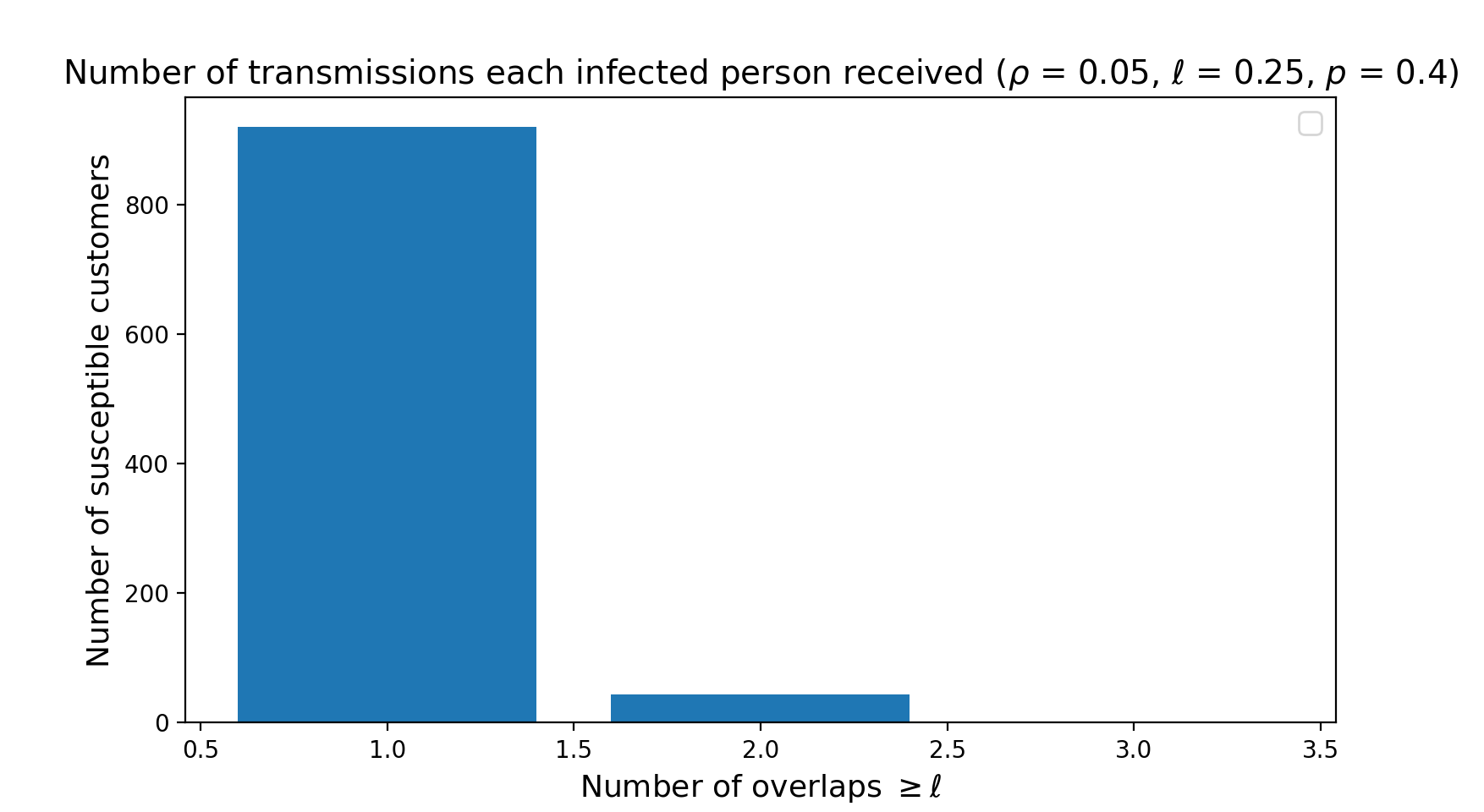}
    \captionsetup{width=.65\linewidth}
\caption{963 susceptible people were infected. Experimentally, only one person infected them 95.43\% of the time.}
\end{subfigure}
\begin{subfigure}[b]{0.6\textwidth}
    \centering
\includegraphics[width=.98\textwidth]{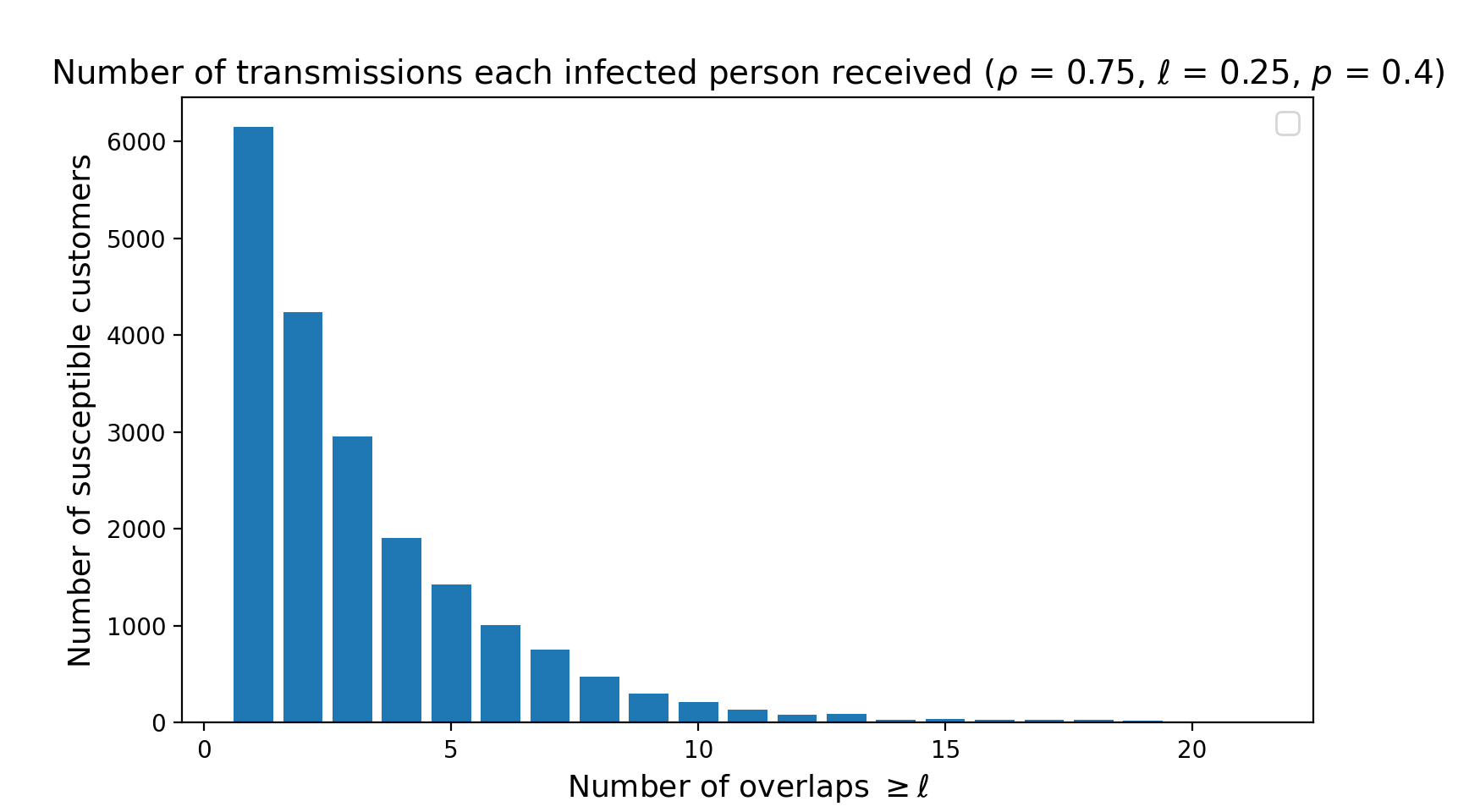}
\captionsetup{width=.65\linewidth}
\caption{19,904 susceptible people were infected. Experimentally, only one person infected them 30.88\% of the time.}
\end{subfigure}}
\caption{Simulating over 50,000 customers, the ``amount'' of infection transmitted to a susceptible customer, as measured by the amount of infectious people whom the susceptible customer overlaps with for at least $\ell$, experimentally depends greatly on the system utilization rate $\rho$.}
\label{fig:transmissions}
\end{figure}

\subsection{Exactly $k$ infections}

We first begin by recognizing that this question is only of novel interest when we consider $k \geq 1$.

\begin{lemma}
For a susceptible arrival to the $M/M/1$ queue in steady state, the probability no customers infect them is,
    $$\mathbb{P}(N_{ua}=0) =\mathbb{P}(N_{aa}=0) = 1-\frac{\lambda p}{\mu-\lambda(1-p)}e^{-(\mu-\lambda)\ell}.$$
\end{lemma}

\begin{proof}
    We recognize that the probability of getting infected by zero people is equivalently the probability of not getting infected at all. This result then follows immediately as the complement of Theorems~\ref{thm:upon-arrival} and~\ref{thm:after-arrival}.
\end{proof}

\subsubsection{Probability of $k$ infection upon arrival}

We next consider the case that $k$ people infect the tagged susceptible customer upon arrival. Our approach will be restricting the entire queue of people ahead of the tagged customer to only those who are in the system long enough to potentially infect the tagged customer. That is, all the people whose remaining sojourn time is at least $\ell$. Of these people, we know \textit{every} customer who is infectious will infect the tagged customer. We will thus consider the case where customer $j$ is the first such potentially infection-transmitting customer: that is event $\mathbf{1}_{\tilde{\mathcal{R}_{j}} > \ell}$.

It is thus clearly essential that we know likelihood of this event, so we first calculate the probability of this event for all $j \leq r$, where $r$ is the length of the queue upon arrival.

\begin{lemma}\label{lem:first-sojourn}
    For the $M/M/1$ queue in steady state, if $k$ people are in the queue, then for all $0<j\leq r$, the probability customer $j$ is the first in the queue with remaining sojourn time at least $\ell$ is $$\mathbb{P}\left(\mathbf{1}_{\tilde{\mathcal{R}_{j}} > \ell}|Q = r\right) = \frac{(\mu\ell)^{j-1}e^{-\mu\ell}}{(j-1)!}.$$
\end{lemma}
\begin{proof}
    First suppose $j = 1$. This is then simply the probability that the first person in the queue has remaining service time at least $\ell$, so $$\mathbb{P}(R_1 \geq \ell|Q=r) =  \mathbb{P}(S_1 \geq \ell) = e^{-\mu\ell}.$$

    Now suppose $j \geq 2$ (for any $j \leq r$). Then we have
    \begin{align*}
        \mathbb{P}\left(\mathbf{1}_{\tilde{\mathcal{R}_{j}} > \ell}|Q = r\right) 
        &= \mathbb{P}\left(\mathbf{1}_{\tilde{\mathcal{R}_{j}} > \ell}\right) \\
        &= \mathbb{P}\left(S_j + \sum_{i=1}^{j-1}R_i \geq \ell, S_{j-1} + \sum_{i=1}^{j-2}R_i < \ell\right)\\
        &= \mathbb{P}\left(\sum_{i=1}^{j}R_i \geq \ell,\sum_{i=1}^{j-1}R_i < \ell\right)\\
        &= \mathbb{P}\left(\sum_{i=1}^{j}R_i \geq \ell\Bigg|\sum_{i=1}^{j-1}R_i < \ell\right)\mathbb{P}\left(\sum_{i=1}^{j-1}R_i < \ell\right)\\
        &= \int_0^{\ell}\mathbb{P}\left(R_j+\sum_{i=1}^{j-1}R_i \geq \ell\Bigg|\sum_{i=1}^{j-1}R_i =x\right)\mathbb{P}\left(\sum_{i=1}^{j-1}R_i  =x\right)dx\\
        &= \int_0^{\ell}\mathbb{P}\left(R_j \geq \ell - x\Bigg|\sum_{i=1}^{j-1}R_i =x\right)\mathbb{P}\left(\sum_{i=1}^{j-1}R_i  =x\right)dx\\
        &= \int_0^{\ell}\mathbb{P}\left(S_j \geq \ell - x\right)\mathbb{P}\left(\sum_{i=1}^{j-1}R_i  =x\right)dx\\
        &= \int_0^{\ell}e^{-\mu(\ell-x)}\frac{\mu^{j-1}x^{j-2}e^{-\mu x}}{(j-2)!}dx\\
        &= \frac{(\mu\ell)^{j-1}e^{-\mu\ell}}{(j-1)!}.
    \end{align*}
\end{proof}

We can now determine the likelihood of $k$ people being able to infect the tagged customer under this assumption, which we do in the following lemma.

\begin{lemma}\label{lem:exactly-k-inner}
    Consider a susceptible arrival to the $M/M/1$ queue in steady state, with queue length $r$ where, for some $j \leq r$,  the $j^{th}$ customer is the first in the line with remaining sojourn time at least $\ell$. For any $k \geq 1$ where $k \leq r-j+1$, the probability $k$ of the customers in the queue infect the tagged new arrival is,
    
    $$\mathbb{P}\left(N_{ua} = k|Q=r,\mathbf{1}_{\tilde{\mathcal{R}_{j}} > \ell}\right) = {r-j+1 \choose k}p^k(1-p)^{r-j+1-k}.$$
\end{lemma}

\begin{proof}
    Under the assumptions that there are $r$ customers in line upon the tagged customer's arrival and only the last $r-j+1$ of them have remaining sojourn time long enough to infect the tagged custoemr, this is simply a binomial distribution with parameters $r-j+1$ and $p$. Thus, the formula follows immediately.
\end{proof}

The probability that the tagged customer is infected by exactly $k$ customers upon their arrival to the queue follows easily from here. 

\begin{theorem}\label{thm:exactly-k-upon-arrival}
    For a susceptible arrival to the $M/M/1$ queue in steady state, and for any $k \geq 1$, the probability they are infected by $k$ customers in the queue upon their arrival is, 
    $$\mathbb{P}(N_{ua} = k) = \left(1-\frac{\lambda p}{\mu-\lambda(1-p)}\right)\left(\frac{\lambda p}{\mu-\lambda(1-p)}\right)^ke^{-(\mu-\lambda)\ell}.$$
\end{theorem}

\begin{proof}
\begin{align*}
        \mathbb{P}(N_{ua} = k) 
        &= \sum_{r=k}^{\infty}\mathbb{P}(Q = r)\sum_{j=1}^{r-k+1}\mathbb{P}\left(\mathbf{1}_{\tilde{\mathcal{R}_{j}} > \ell}|Q=r\right)\mathbb{P}\left(N_{ua} = k|Q=r,\mathbf{1}_{\tilde{\mathcal{R}_{j}} > \ell}\right).\intertext{By Lemmas~\ref{lem:first-sojourn} and~\ref{lem:exactly-k-inner}, this becomes}
        &= \sum_{r=k}^{\infty}\left(\frac{\lambda}{\mu}\right)^r\left(1-\frac{\lambda}{\mu}\right)\sum_{j=1}^{r-k+1}\frac{(\mu\ell)^{j-1}e^{-\mu\ell}}{(j-1)!}{r-j+1 \choose k}p^k(1-p)^{r-j+1-k}\\
        &= \left(\frac{p}{1-p}\right)^k\frac{\mu-\lambda}{\mu}e^{-\mu\ell}\sum_{j=1}^{\infty} \frac{(\mu\ell)^{j-1}}{(j-1)!}\sum_{r=j+k-1}^{\infty}{r-j+1 \choose k}\left(\frac{\lambda}{\mu}\right)^r(1-p)^{r-j+1}\\
        &= \left(\frac{p}{1-p}\right)^k\frac{\mu-\lambda}{\mu}e^{-\mu\ell}\sum_{j=1}^{\infty} \frac{(\mu\ell)^{j-1}}{(j-1)!}\left(\frac{\mu^2(1-p)^k\left(\frac{\lambda}{\mu}\right)^{j+k}\left(\frac{\mu}{\mu-(1-p)\lambda}\right)^k}{\lambda(\mu-\lambda(1-p))}\right)\\
        &= \left(\frac{\lambda p}{\mu-\lambda(1-p)}\right)^k\frac{\mu-\lambda}{\mu-\lambda(1-p)}e^{-\mu\ell}\sum_{j=1}^{\infty} \frac{(\lambda\ell)^{j-1}}{(j-1)!}\\
        &= \left(\frac{\lambda p}{\mu-\lambda(1-p)}\right)^k\frac{\mu-\lambda}{\mu-\lambda(1-p)}e^{-(\mu-\lambda)\ell}.
\end{align*}
\end{proof}

\subsubsection{Probability of $k$ infections after arrival}

For the case in which a susceptible customer is infected by exactly 
$k$ arrivals to the queue after their own arrival, we observe that this probability can be expressed as the likelihood that while the customer has $\ell$ units of time remaining in their sojourn within the system, exactly $k$ infectious arrivals occur. 

\begin{theorem}
    For a susceptible arrival to the $M/M/1$ queue in steady state, and for any $k \geq 1$, the probability they are infected by $k$ customers who arrive to the queue after them is,  
    $$\mathbb{P}(N_{aa} = k) = \left(1-\frac{\lambda p}{\mu-\lambda(1-p)}\right)\left(\frac{\lambda p}{\mu-\lambda(1-p)}\right)^ke^{-(\mu-\lambda)\ell}.$$
\end{theorem}

\begin{proof}
    \begin{align*}
        \mathbb{P}(N_{aa} = k) 
        &= \sum_{r=0}^{\infty}\mathbb{P}(Q=r)\int_{\ell}^{\infty}\mathbb{P}(\mathcal{S}=x|Q=r)\mathbb{P}(M_I(x-l) = k|\mathcal{S}=x,Q=r)dx\\
        &= \sum_{r=0}^{\infty}\left(\frac{\lambda}{\mu}\right)^r\left(1-\frac{\lambda}{\mu}\right)\int_{\ell}^{\infty}\frac{\mu^{r+1}x^re^{-\mu x}}{r!}\frac{e^{-\lambda p (x-\ell)}\left(\lambda p(x-\ell)\right)^k}{k!}dx\\
        &= \frac{\mu-\lambda}{k!}\int_{\ell}^{\infty}e^{-\mu x}e^{-\lambda p (x-\ell)}\left(\lambda p(x-\ell)\right)^k\sum_{r=0}^{\infty}\frac{\left(\lambda x\right)^r}{r!}dx\\
        &= \frac{\mu-\lambda}{k!}\int_{\ell}^{\infty}e^{-(\mu-\lambda) x}e^{-\lambda p (x-\ell)}\left(\lambda p(x-\ell)\right)^kdx\\
        &= \frac{\mu-\lambda}{k!}\left(\frac{k!(\lambda p)^k}{(\mu-\lambda(1-p))^{k+1}}e^{-(\mu-\lambda)\ell}\right)\\
        &= \frac{\mu-\lambda}{\mu-\lambda(1-p)}\left(\frac{\lambda p}{\mu-\lambda(1-p)}\right)^k e^{-(\mu-\lambda)\ell}.
    \end{align*}
\end{proof}

\subsubsection{Probability of $k$ total infections}

We now find the joint probability of $k$ infection transmissions combined between the upon and after arrivals. In doing so, we leverage the conditional independence of the two cases given a fixed sojourn time.

First, we determine the likelihood that exactly $m$ people infect a tagged customer upon arrival, given sojourn time $x.$
\begin{lemma}\label{lem:num-upon-given-sojourn}
For a customer arriving to the $M/M/1$ queue in steady state with sojourn time $x$, the likelihood that they are infected by $m$ customers in the system upon their arrival is,
    $$\mathbb{P}(N_{ua} = m|\mathcal{S} = x) =e^{-\lambda p(x-\ell)}\frac{(\lambda p(x -\ell))^m}{m!}.$$
\end{lemma}
\begin{proof}
    Here, we recognize that any customer who departs within the first $\ell$ units of time after this tagged customer arrives to the system can not possibly infect them. Thus, utilizing the memoryless property of the exponential distribution, it is useful to only consider a ``snapshot'' of the queue after $\ell$ units of time have already passed, or the state of the queue when the tagged customer's remaining sojourn time is $x-\ell$.

    \begin{align*}
        \mathbb{P}(N_{ua} = m|\mathcal{S} = x)  
        &=\sum_{r=m}^{\infty}\mathbb{P}(Q=r|\mathcal{R} = x-\ell)\sum_{j=m}^{r}\mathbb{P}(I_{\max} = j|Q=r)\mathbb{P}(N_{ua} = m|I_{\max} = j)\intertext{Because we consider the queue as now only consisting of those customers whose \textcolor{black}{remaining} sojourn time was at least $\ell$ upon the tagged customer's arrival, any infectious customers amongst the new ``first'' $j$ must successfully transmit their infection. Further, we have assumed that customer $j$ is infected. These facts and Lemma~\ref{lem:queue-given-sojourn} give us}
        &=\sum_{r=m}^{\infty}\frac{e^{-\lambda (x-\ell)}(\lambda (x-\ell))^r}{r!}\sum_{j=m}^{r}p(1-p)^{r-j}{m-1 \choose j-1} p^{m-1}(1-p)^{j-m}\\
        &=\left(\frac{p}{1-p}\right)^me^{-\lambda (x-\ell)}\sum_{r=m}^{\infty}\frac{(\lambda (x-\ell)(1-p))^r}{r!}\sum_{j=m}^{r}{m-1 \choose j-1}\\
        &=\left(\frac{p}{1-p}\right)^me^{-\lambda (x-\ell)}\sum_{r=m}^{\infty}\frac{(\lambda (x-\ell)(1-p))^r}{r!}\frac{(r-m+1){r \choose m-1}}{m}\\
        &=\left(\frac{p}{1-p}\right)^me^{-\lambda (x-\ell)}\frac{e^{\lambda(x-\ell)(1-p)}(\lambda(x-\ell)(1-p))^m}{m!}\\
        &=e^{-\lambda p (x-\ell)}\frac{(\lambda p (x-\ell))^m}{m!}.
    \end{align*}
\end{proof}

We find that the conditional probability of infections from arrivals after the tagged customer achieves an identical distribution, which we expect under Conjecture~\ref{conj:symmetry}.

\begin{lemma}\label{lem:num-after-given-sojourn}
For a customer arriving to the $M/M/1$ queue in steady state with sojourn time $x$, the likelihood that they are infected by $m$ customers who arrive to the system after them is,
    $$\mathbb{P}(N_{aa} = m|\mathcal{S} = x) = e^{-\lambda p(x-\ell)}\frac{(\lambda p(x -\ell))^m}{m!}.$$
\end{lemma}

\begin{proof}
    This follows immediately from recognizing that everyone who infects the tagged customer must arrive at least $\ell$ units of time before the customer exits the system, that is within timeframe $x-\ell$.
\end{proof}

We can now combine these lemmas to find the total probability a susceptible customer is infected by $k$ individuals.

\begin{theorem}\label{thm:exactly-k-infections}
Consider any $k\geq 1.$ For a susceptible arrival to the $M/M/1$ queue in steady state, the probability that exactly $k$ infectious customers transmit infection to them is,
\begin{eqnarray}\label{eqn:exactkinfec}
    \mathbb{P}(N_{I} = k) =\frac{\mu-\lambda}{\mu-\lambda+2\lambda p}\left(\frac{2\lambda p}{\mu-\lambda+2\lambda p}\right)^ke^{-(\mu-\lambda)\ell}.
\end{eqnarray}
\end{theorem}

\begin{proof}
    Utilizing the fact that the total number of people who infect the tagged customer is simply the sum of the number who infect them upon arrival and the number who infect them after arrival, since those sets partition all those who infect the tagged customer, we find that
    \begin{align*}
        \mathbb{P}(N_{I} = k) 
        &= \sum_{m=0}^k \mathbb{P}(N_{ua} = m, N_{aa} = k-m)\\
        &= \sum_{m=0}^k\int_{\ell}^{\infty}\mathbb{P}(\mathcal{S} = x)\mathbb{P}(N_{ua} =m, N_{aa} = k-m|\mathcal{S} = x) dx \intertext{So by Lemmas~\ref{lem:num-upon-given-sojourn} and~\ref{lem:num-after-given-sojourn} and Lemma~\ref{lem:prob-sojourn-x},}
        &=\int_{\ell}^{\infty}(\mu -\lambda)e^{-(\mu-\lambda)x}e^{-2\lambda p(x-\ell)}\sum_{m=0}^k\frac{(\lambda p(x -\ell))^{k}}{m!(k-m)!}dx\\
        &=\int_{\ell}^{\infty}(\mu -\lambda)e^{-(\mu-\lambda)x}e^{-2\lambda p(x-\ell)}\frac{(2\lambda p(x -\ell))^{k}}{k!}dx\\
        &=\frac{(\mu-\lambda)(2\lambda p)^k}{k!}\int_{\ell}^{\infty}e^{-(\mu-\lambda)x}e^{-2\lambda p(x-\ell)}(x -\ell)^{k}dx\\
        &=\frac{(\mu-\lambda)(2\lambda p)^k}{k!}\left(\frac{-e^{-(\mu-\lambda)\ell}}{(\mu-\lambda+2\lambda p)^{k+1}}\Gamma(k+1, (\mu-\lambda+2\lambda p)(x-\ell))\right)\Bigg|_{x=\ell}^{x=\infty}\intertext{making use of the fact that $k+1\in\mathbb{Z}^{+}$, we can express the lower incomplete gamma function in summatio form as}
        &=\frac{(\mu-\lambda)e^{-(\mu-\lambda)\ell}}{\mu-\lambda+2\lambda p}\left(\frac{2\lambda p}{\mu-\lambda+2\lambda p}\right)^k\left(\sum_{n=0}^k\frac{((\mu-\lambda+2\lambda p)(x-\ell))^n}{n! e^{(\mu-\lambda+2\lambda p)(x-\ell)}}\right)\Bigg|^{x=\ell}_{x=\infty}\intertext{Recognizing that in this context $0^0 = 1$, we have then,}
        &=\frac{(\mu-\lambda)e^{-(\mu-\lambda)\ell}}{\mu-\lambda+2\lambda p}\left(\frac{2\lambda p}{\mu-\lambda+2\lambda p}\right)^k\left(1 - 0\right)\\
        &= \frac{\mu-\lambda}{\mu-\lambda+2\lambda p}\left(\frac{2\lambda p}{\mu-\lambda+2\lambda p}\right)^ke^{-(\mu-\lambda)\ell}.
    \end{align*}
\end{proof}

\subsection{At least $k$ infections}
It is also of interest to know not just the likelihood of a specific number of infections, but the likelihood of getting at least that many. As a way of considering severity of infection, this would allow us to know the probability of getting at least some degree of severity.

\begin{figure}[H]
\centering
\makebox[\linewidth][c]{
\begin{subfigure}[b]{0.6\textwidth}
        \centering
    \includegraphics[width=.98\textwidth]{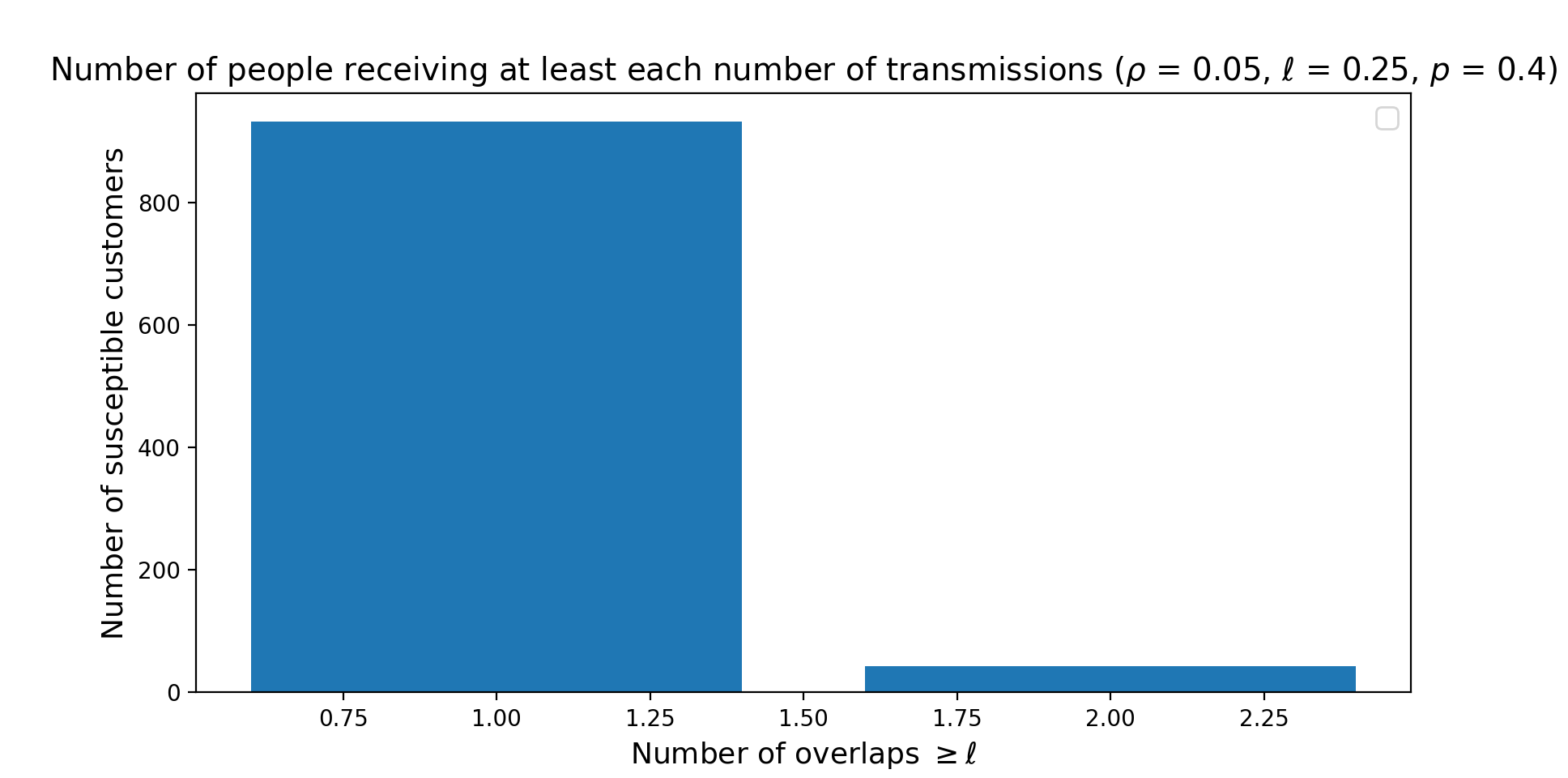}
    \captionsetup{width=.65\linewidth}
\caption{903 susceptible people who became infected. Experimentally, at least two people infected them 4.76\% of the time.}
\end{subfigure}
\begin{subfigure}[b]{0.6\textwidth}
    \centering
\includegraphics[width=.98\textwidth]{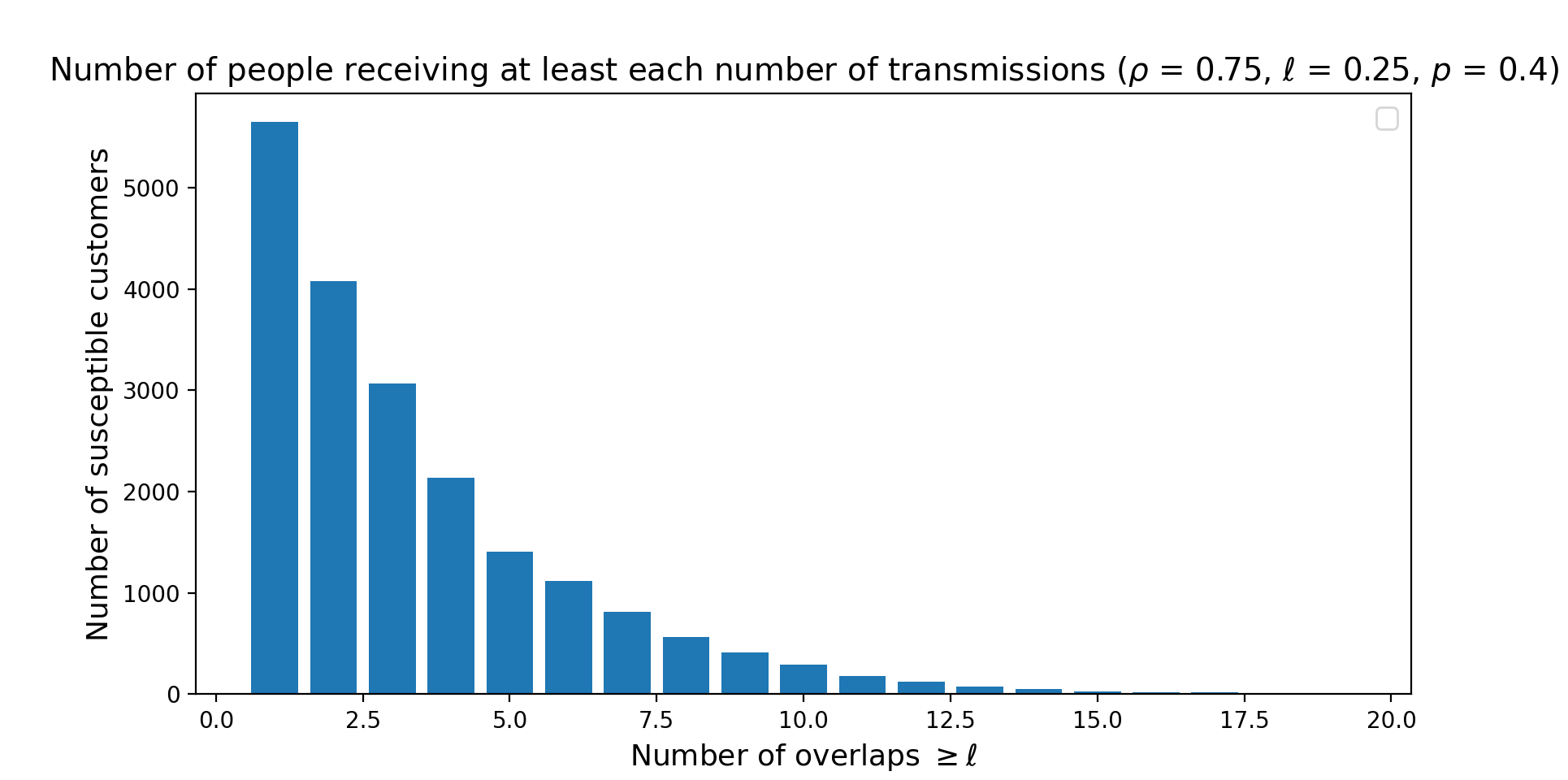}
\captionsetup{width=.65\linewidth}
\caption{19,842 susceptible people who became infected. Experimentally, at least two people infected them 71.04\% of the time.}
\end{subfigure}}
\caption{Simulating over 50,000 customers, we confirm that, given a susceptible customer became infected at all, the likelihood of being infected at least twice is the complement of being infected only once.}
\label{fig:atleasttransmissions}
\end{figure}

We again begin by recognizing a trivial case: requiring at least 0 infections is the trivial case of allowing any outcome that could happen to a susceptible customer entering the system.

\subsubsection{Probability of at least $k$ infections upon arrival}

The probability that at least $k$ infections occur from people in front of the tagged susceptible customer in the queue follows straightforwardly along the same logic as in Theorem~\ref{thm:exactly-k-upon-arrival}, but where we now allow for more than $k$ customers to possibly infect them.

\begin{theorem}
For a susceptible arrival to the $M/M/1$ queue in steady state, and for any $k \geq 1$, the probability that at least $k$ customers in the queue upon their arrival infect them is,
$$\mathbb{P}(N_{ua} \geq k) = \left(\frac{\lambda p}{\mu-\lambda(1-p)}\right)^ke^{-(\mu-\lambda)\ell}.$$
\end{theorem}

\begin{proof}
\begin{align*}
        &\mathbb{P}(N_{ua} \geq k) \\
        &= \sum_{r=k}^{\infty}\mathbb{P}(Q = r)\sum_{j=1}^{r-k+1}\mathbb{P}\left(\mathbf{1}_{\tilde{\mathcal{R}_{j}} > \ell}|Q=r\right)\sum_{i=k}^{r-j+1}\mathbb{P}\left(N_{ua} = i|Q=r,\mathbf{1}_{\tilde{\mathcal{R}_{j}} > \ell}\right)\intertext{Again leveraging Lemmas~\ref{lem:first-sojourn} and~\ref{lem:exactly-k-inner}, we get}
        &= e^{-\mu\ell}\frac{\mu-\lambda}{\mu}\sum_{j=1}^{\infty}\frac{(\mu\ell)^{j-1}}{(j-1)!}\sum_{r=k+j-1}^{\infty}\left(\frac{\lambda}{\mu}\right)^r\sum_{i=k}^{r-j+1}{r-j+1 \choose i}p^i(1-p)^{r-j+1-i}\\
        &= e^{-\mu\ell}\frac{\mu-\lambda}{\mu}\sum_{j=1}^{\infty}\frac{(\mu\ell)^{j-1}}{(j-1)!}\sum_{i=k}^{\infty}\left(\frac{p}{1-p}\right)^i\sum_{r=i+j-1}^{\infty}{r-j+1 \choose i}\left(\frac{\lambda}{\mu}\right)^r(1-p)^{r-j+1}\\
        &= e^{-\mu\ell}\frac{\mu-\lambda}{\mu}\sum_{j=1}^{\infty}\frac{(\mu\ell)^{j-1}}{(j-1)!}\sum_{i=k}^{\infty}\left(\frac{p}{1-p}\right)^i\left(\frac{\mu^2}{\lambda(\mu-\lambda(1-p))}\left(\frac{\lambda}{\mu}\right)^{i+j}\left(\frac{\mu(1-p)}{\mu-\lambda(1-p)}\right)^i\right)\\
        &= e^{-\mu\ell}\frac{\mu-\lambda}{\mu-\lambda(1-p)}\sum_{j=1}^{\infty}\frac{(\lambda\ell)^{j-1}}{(j-1)!}\sum_{i=k}^{\infty}\left(\frac{\lambda p}{\mu-\lambda(1-p)}\right)^i\\
        &= e^{-\mu\ell}\frac{\mu-\lambda}{\mu-\lambda(1-p)}\left(e^{\lambda\ell}\right)\left(\frac{\mu-\lambda(1-p)}{\mu-\lambda}\left(\frac{\lambda p}{\mu-\lambda(1-p)}\right)^k\right)\\
        &= \left(\frac{\lambda p}{\mu-\lambda(1-p)}\right)^ke^{-(\mu-\lambda)\ell}.
    \end{align*}
\end{proof}

\subsubsection{Probability of at least $k$ infections after arrival}

For the case of the tagged susceptible customer getting infected by at least $k$ arrivals after them, we recognize that this no longer requires exactly $k$ arrivals \textcolor{black}{while the tagged customer's remaining sojourn time is at least $\ell$}, but rather requires that the $k^{th}$ arrival must have occurred by then (and by placing no restrictions on potential $k+i^{th}$ arrivals, we implicitly allow more than $k$ infection transmissions).

\begin{lemma}\label{lem:uniform-k-stat}
Consider a susceptible arrival to the $M/M/1$ queue in steady state, with sojourn time $x$. For some $k \geq 1$, then for any $j \geq k$, if $j$ infectious arrivals occurred during their sojourn, the likelihood that at least $k$ of them transmitted infection to the tagged customer is,
$$\mathbb{P}(A_{(k)} \leq x-\ell|M_I = j, \mathcal{S}=x) = \left(\frac{\ell}{x}\right)^j\sum_{i=k}^j{j \choose i}\left(\frac{x-\ell}{\ell}\right)^i.$$
\end{lemma}

\begin{proof}
    Given $M_I=j$ and $\mathcal{S}=x$, $A_{i} \sim U(0,x).$ So, 
    \begin{align*}
        \mathbb{P}(A_{(k)} \leq x-\ell|M_I = j, \mathcal{S}=x) 
        &= \sum_{i=k}^j{j \choose i}\left(\frac{x-\ell}{x}\right)^i\left(1-\frac{x-\ell}{x}\right)^{j-i}\\
        &= \left(\frac{\ell}{x}\right)^j\sum_{i=k}^j{j \choose i}\left(\frac{x-\ell}{x}\right)^i\left(\frac{\ell}{x}\right)^{-i}\\
        &=\left(\frac{\ell}{x}\right)^j\sum_{i=k}^j{j \choose i}\left(\frac{x-\ell}{\ell}\right)^i.    \end{align*}
\end{proof}

Leveraging this lemma, the probability of at least $k$ infections after arrival is again straightforward.

\begin{theorem}
        For a susceptible arrival to the $M/M/1$ queue in steady state, and for any $k \geq 1$, the probability that at least $k$ arrivals to the queue after them infect the tagged customer is,
$$\mathbb{P}(N_{aa} \geq k) = \left(\frac{\lambda p}{\mu-\lambda(1-p)}\right)^ke^{-(\mu-\lambda)\ell}.$$
\end{theorem}
\begin{proof}
    \begin{align*}
        &\mathbb{P}(N_{aa} \geq k) \\
        &= \sum_{r=0}^{\infty}\mathbb{P}(Q=r)\int_{\ell}^{\infty}\mathbb{P}(\mathcal{S}=x|Q=r)\sum_{j=k}^{\infty}\mathbb{P}(A_{(k)} \leq x-\ell|M_I = j,\mathcal{S}=x)\mathbb{P}(M_I=j|\mathcal{S}=x)dx.\intertext{By Lemma~\ref{lem:uniform-k-stat},}
        &=\sum_{r=0}^{\infty}\left(\frac{\lambda}{\mu}\right)^r\left(1-\frac{\lambda}{\mu}\right)\int_{\ell}^{\infty}\frac{\mu^{r+1}x^re^{-\mu x}}{r!}\sum_{j=k}^{\infty}\frac{e^{-\lambda p x}(\lambda p x)^j}{j!}\left(\frac{\ell}{x}\right)^j\sum_{i=k}^j{j \choose i}\left(\frac{x-\ell}{\ell}\right)^i dx\\
        &=(\mu-\lambda)\int_{\ell}^{\infty}e^{-\mu x}e^{-\lambda p x}\sum_{r=0}^{\infty}\frac{\left(\lambda x\right)^r}{r!}\sum_{j=k}^{\infty}\frac{(\lambda p \ell)^j}{j!}\sum_{i=k}^j{j \choose i}\left(\frac{x-\ell}{\ell}\right)^i dx\\
        &=(\mu-\lambda)\int_{\ell}^{\infty}e^{-\mu x}e^{-\lambda p x}\sum_{r=0}^{\infty}\frac{\left(\lambda x\right)^r}{r!}\sum_{i=k}^{\infty} \left(\frac{x-\ell}{\ell}\right)^i\sum_{j=i}^{\infty}{j \choose i}\frac{(\lambda p \ell)^j}{j!} dx\\
        &=(\mu-\lambda)\int_{\ell}^{\infty}e^{-\mu x}e^{-\lambda p x}\sum_{r=0}^{\infty}\frac{\left(\lambda x\right)^r}{r!}\sum_{i=k}^{\infty} \left(\frac{x-\ell}{\ell}\right)^i\left( \frac{e^{\lambda p \ell}(\lambda p \ell)^i}{i!}\right)dx\\
        &=(\mu-\lambda)e^{\lambda p \ell}\sum_{i=k}^{\infty} \frac{\left(\lambda p\right)^i}{i!} \int_{\ell}^{\infty}(x-\ell)^i e^{-\mu x}e^{-\lambda p x}\sum_{r=0}^{\infty}\frac{\left(\lambda x\right)^r}{r!}dx\\
        &=(\mu-\lambda)e^{\lambda p \ell}\sum_{i=k}^{\infty} \frac{\left(\lambda p\right)^i}{i!} \int_{\ell}^{\infty}(x-\ell)^i e^{-(\mu-\lambda(1-p)) x}dx\\
        &=(\mu-\lambda)e^{\lambda p \ell}\sum_{i=k}^{\infty} \frac{\left(\lambda p\right)^i}{i!} \left(e^{-(\mu-\lambda(1-p)) \ell}\frac{i!}{(\mu-\lambda(1-p))^{i+1}}\right)\\
        &=\frac{\mu-\lambda}{\mu-\lambda(1-p)}e^{-(\mu-\lambda) \ell}\sum_{i=k}^{\infty} \left(\frac{\lambda p}{\mu-\lambda(1-p)}\right)^i\\
        &=\frac{\mu-\lambda}{\mu-\lambda(1-p)}e^{-(\mu-\lambda) \ell}\left(\frac{\mu-\lambda(1-p)}{\mu-\lambda}\left(\frac{\lambda p}{\mu-\lambda(1-p)}\right)^k\right)\\
        &=\left(\frac{\lambda p}{\mu-\lambda(1-p)}\right)^ke^{-(\mu-\lambda) \ell}.
    \end{align*}
\end{proof}

\subsubsection{Probability of at least $k$ infections}

In finding the probability of a susceptible customer receiving at least $k$ infection transmissions, we recognize that for $k =1$, this should be identical to the probability of the tagged customer getting infected at all. Indeed we find that,

\begin{theorem}\label{thm:at-least-k}
    Consider any $k\geq 1.$ For a susceptible arrival to the $M/M/1$ queue in steady state, the probability that at least $k$ infectious customers transmit infection to them is,
    $$\mathbb{P}(N_I \geq k) = \left(\frac{2\lambda p}{\mu-\lambda+2\lambda p}\right)^ke^{-(\mu-\lambda)\ell}.$$
\end{theorem}

\begin{proof}
    We leverage Theorem~\ref{thm:exactly-k-infections}.
    \begin{align*}
        \mathbb{P}(N_I \geq k) 
        &= \sum_{i=k}^{\infty}\mathbb{P}(N_I = k)\\
        &= \sum_{i=k}^{\infty}\frac{\mu-\lambda}{\mu-\lambda+2\lambda p}\left(\frac{2\lambda p}{\mu-\lambda+2\lambda p}\right)^ie^{-(\mu-\lambda)\ell}\\
        &=\left(\frac{2\lambda p}{\mu-\lambda+2\lambda p}\right)^ke^{-(\mu-\lambda)\ell}.
    \end{align*}

    Note that for $k=1$ this matches the result from Theorem~\ref{thm:total-infection-prob}.
\end{proof}

\subsection{Moments of the number infected}\label{sec:moments}
We reframe the problem to now consider the arrival of a single \textit{infectious customer}, such that now we let $N_I$ be the number of susceptible customer this tagged customer infects.

In this section, we describe how to compute the moments of the number of infections attributed to one infected customer.  These new results will leverage our previous results. Despite our change in perspective, since our results simply leverage the effect of overlapping for sufficiently long with an customer of opposite infection type, we can substitute infectious arrivals at rate $\lambda p$ with susceptible arrivals at rate $\lambda(1-p)$, to see that, without loss of generality, our previous results compute the distribution of the number of customers infected by one tagged infectious arrival. 

\begin{theorem}\label{thm:moments-num-infect}
Let $N_I$ be the number of susceptible customers that are infected by one infectious customer. Then the $m^{th}$ moment of $N_I$ is given by 
\begin{eqnarray*}
    \mathbb{E} \left[ N_I^m \right] &=& \frac{\mu - \lambda}{\mu - \lambda + 2\lambda (1-p)} \operatorname{Li}_{-m}\left( \frac{2\lambda (1-p)}{\mu - \lambda + 2\lambda (1-p)} \right )
\end{eqnarray*}
where $\operatorname{Li}_{m}(x)$ is the polylogarithm function and is defined by 
\begin{equation*}
    \operatorname{Li}_{m}(x) = \sum^{\infty}_{j=1} \frac{x^j}{j^m}.
\end{equation*}
    \begin{proof}
        \begin{eqnarray*}
    \mathbb{E} \left[ N_I^m \right] &=& \sum_{j=0}^{\infty} j^m \cdot \mathbb{P} \left( N_I = m \right) \\
    &=& \sum_{j=0}^{\infty} j^m \cdot \left( \frac{\mu - \lambda}{\mu - \lambda + 2\lambda (1-p)} \right) \left( \frac{2 \lambda (1-p)}{\mu - \lambda + 2\lambda (1-p)} \right)^j e^{-(\mu-\lambda) \ell}  \\
    &=& \frac{\mu - \lambda}{\mu - \lambda + 2\lambda (1-p)} \operatorname{Li}_{-m}\left( \frac{2\lambda (1-p)}{\mu - \lambda + 2\lambda (1-p)} \right ) e^{-(\mu-\lambda) \ell} .
\end{eqnarray*}
    \end{proof}
\end{theorem}

It is worth noting that, for $\beta = \frac{\mu - \lambda}{\mu - \lambda + 2\lambda (1-p)}$, then $$\mathbb{E}[N_I^m] = \beta \operatorname{Li}_{-m}(1-\beta) e^{-(\mu-\lambda)\ell}.$$ This intuitively makes the moments of the number infected by one infectious customer the product of two terms: the exponential that the tagged customer's sojourn time is at least the overlap and the moment of a geometric with parameter $\beta$. Here, $\beta$ represents the probability that the exponential for the tagged customer's sojourn time is more than that of having a susceptible arrival either upon or after their arrival.

In particular, the polylogarithm function has the following special values,
\begin{eqnarray*}
\operatorname{Li}_{-1}(x) &=& \frac{x}{(1-x)^2} \\
\operatorname{Li}_{-2}(x) &=& \frac{x (1+x)}{(1-x)^3}.
\end{eqnarray*}
Thus, the mean and variance of the number infected by one infectious customer is equal to
\begin{eqnarray}
 \mathbb{E} \left[ N_I \right] &=& \frac{2\lambda(1-p)}{\mu-\lambda} e^{-(\mu-\lambda) \ell},\text{ and}    \label{eqn:meannuminfec}\\
  \mathrm{Var} \left[ N_I \right] &=& \frac{2 \lambda (1-p) \left( \mu - \lambda + 4 \lambda (1-p) \right)}{(\mu - \lambda)^2} e^{-(\mu-\lambda) \ell}\nonumber\\
  &&- \frac{4\lambda^2 (1-p)^2}{(\mu-\lambda)^2} e^{-2(\mu-\lambda) \ell}. \label{eqn:varnuminfec}
\end{eqnarray}

\subsection{New infections vs. transmissions}



Now that we have formulas for the mean number of customers that an infectious customer will infect and the probability of infection for a susceptible customer, we can compare the mean new number of infections to the expected number of transmissions from infectious customers.  

If we wish to find the expected number of total new infections, we recognize that for fixed time frame $t$, we expect $\lambda(1-p) t$ susceptible arrivals and $\lambda p t$ infectious arrivals.  Thus, we expect $\lambda(1-p) \mathbb{P}(I) t$ new infections and we expect $\lambda p \mathbb{E}[N_I] t$ transmissions of infection from infectious customers.  The ratio of new infections and number of transmissions is equal to
\begin{eqnarray*}
    \frac{\text{New Infections}}{\text{Total Number of Transmissions}} &=& \frac{\mu - \lambda}{\mu - \lambda + 2\lambda p} = 1 - \frac{2\lambda p}{\mu - \lambda + 2\lambda p} .
\end{eqnarray*}

Note that the ratio is equal to a value that is between zero and one.  We observe that as the proportion of infectious arrivals of the system, that is $p$, goes to zero, then the ratio of infections and transmissions goes to one.  Moreover, when the probability of being infectious goes to one, then the ratio converges to $\frac{\mu - \lambda}{\mu + \lambda}$.  This demonstrates how, for each infectious customer, counting the number of susceptible customers they would infect on average causes us to overcount duplicate transmissions of infection from multiple infectious customers to a single susceptible customer.

We note (see Figure~\ref{fig:overlaps-transmissions}) that this overcounting is related to $\rho$, or more directly to the expected number of overlaps.

\begin{figure}[!htbp] 
\centering
\makebox[\linewidth][c]{
\includegraphics[width=0.6\textwidth]{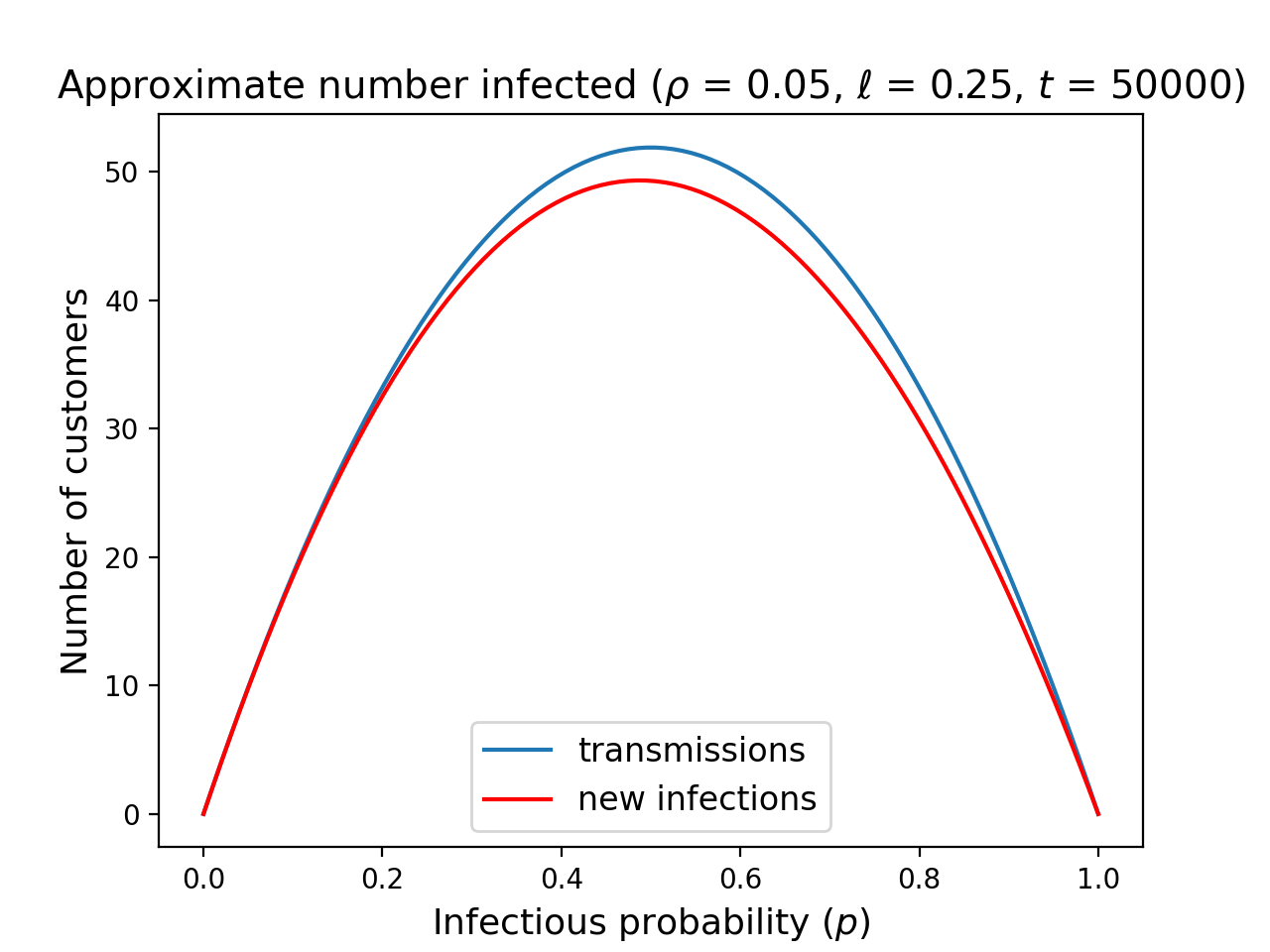}~\includegraphics[width=0.6\textwidth]{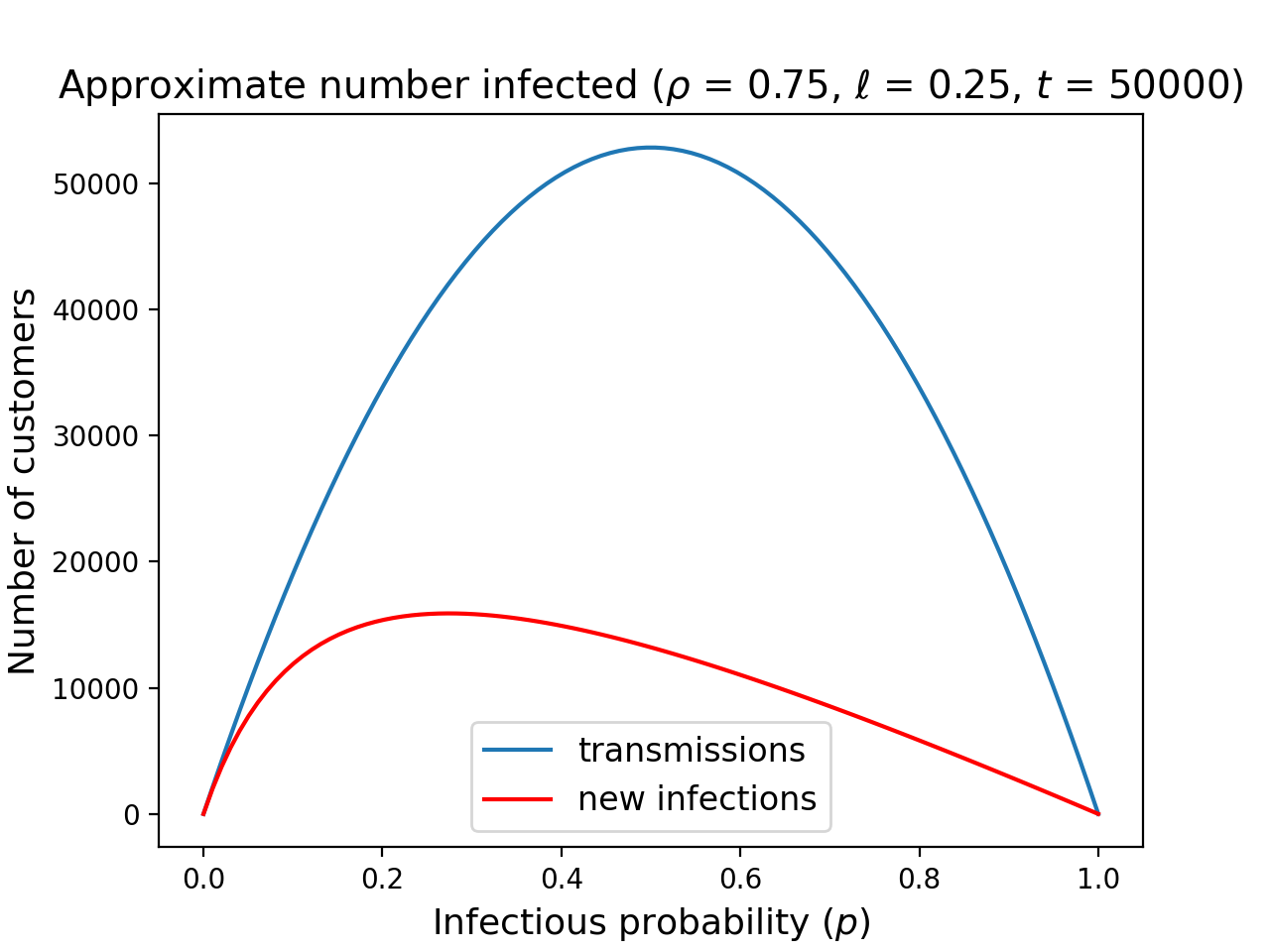}}
\caption{Higher server utilization, $\rho = \frac{\lambda}{\mu}$, causes these formulas to diverge.}
\label{fig:overlaps-transmissions}
\end{figure}


We see that as the utilization increases, the number of total infection transmissions diverges from the number of new infections.  This is because infectious customers are infecting the \textit{same} susceptible customers when the utilization is high.  This also suggests that in higher utilization settings that people might get ``even sicker'' if we were measuring this based on the number of people that infected them.   

{\color{black} \begin{remark}
We also emphasize that all of the results developed in this section naturally extend to the random-threshold setting introduced in Section~\ref{sec:prob-infec}. The key structural feature of our fixed-threshold expressions is that they depend on the threshold only through exponential terms of the form $e^{-(\mu - \lambda) \ell}$. As shown earlier, replacing this exponential dependence by the Laplace–Stieltjes transform of the threshold distribution immediately yields the corresponding formulas for a random threshold. {\color{black} Further, the context of considering a random threshold for each susceptible customer described in Section~\ref{sec:prob-infec} can, by our same logic in adapting the earlier results from Section~\ref{sec:mult-infec} to Section~\ref{sec:moments}, be reimagined for the moments as a different random threshold for each infectious customer. While this does not allow for quite the same degree of individual susceptible-infectious pair threshold randomness as allowed for in \cite{kang2021queueing} (with the caveat that their mathematical results assume a single infectious customer in the system at once), we again refer the reader to Tables~\ref{tab:compareaveragemeans} and \ref{tab:compareaveragevars} for insight into our conjecture that this setting is likely highly similar to the system explored here.} Thus, each of the moment expressions derived here, including the probability of infection, the mean number of exposures, and higher-order moments can be generalized in a completely parallel fashion by applying the same Laplace–Stieltjes transform argument used in Section~3. This provides a unified and highly flexible analytical framework, allowing our moment analysis to accommodate arbitrary threshold distributions without requiring any additional derivations or assumptions. If a practitioner specifies a threshold distribution, the random-threshold results follow directly by evaluating its Laplace–Stieltjes transform at $(\mu - \lambda)$, thereby demonstrating the broad applicability and generality of our approach. 
\end{remark}}

\section{Numerical Insights and Extensions} \label{numerics_infections}

{\color{black}We explore comparisons of our theoretical results with simulations and consider extensions to the $G/M/1$ and $M/G/1$ cases.}

\subsection{$M/M/1$ numerical results}

{\color{black} Our simulation framework is as follows:}

For each set of parameters we average over 5 simulations, with 500,000 customers for each simulation. We always normalize our utilization by fixing $\mu = 1$ for all experiments.  Note that as $\frac{\lambda}{\mu}\rightarrow 1$, it takes longer to reach a steady state, {\color{black}hence we expect to see less accurate simulations as $\lambda$ increases, especially for the higher moments}. We finally begin by holding $\ell = 0.5$ constant across our initial experiments. 

In Table~\ref{tab:probs} we consider the infection probability of a susceptible customer. We recognize that the probability of becoming infected (at all) is effectively the probability of getting at least 1 infection, but we also explore the probabilities of getting exactly $k$ infections in Table~\ref{tab:exactk}, where we consider the severity-$k$ infection probabilities. Next, as insight into the impact on infectious customers on the system, in Tables~\ref{tab:meanvar}, \ref{tab:elllowvar}, and \ref{tab:ellhighvar} we consider the mean and variance of the number infected {\color{black} for our initial and then two additional infection thresholds.  Finally, in Section~\ref{sec:random-thresh} we conduct experiments exploring random threshold settings.

\subsubsection{Fixed infection threshold}

We begin with the experimental and theoretical infection probabilities for a susceptible customer.}

\begin{table}[!htbp]
\centering
\small
\begin{tabular}{|c|c|c|c|}
\hline
$p$ & $\lambda$& \makecell{$\mathbb{P}(I)$\\ (Simulation)}  & \makecell{$\mathbb{P}(I)$\\ (Equation~\ref{eqn:prob-infec})}\\\hline
.01 & .50 & .0156 & .0153 \\

.01 & .75 & .0493 & .0500 \\

.01 & .90 & .145 & .145 \\\hline

.02 & .50 & .0299 & .0300 \\

.02 & .75 & .0947 & .0946 \\

.02 & .90 & .251 & .252 \\\hline

.10 & .50 & .130 & .130 \\

.10 & .75 & .331 & .331 \\

.10 & .90 & .614 & .612 \\\hline

.50 & .50 & .389 & .389 \\

.50 & .75 & .661 & .662 \\

.50 & .90 & .857 & .856 \\\hline

\end{tabular}
\caption{For a susceptible arrival in steady state, the theoretical and simulated probabilities that this customer is infected.}
\label{tab:probs}
\end{table} 

{\color{black} We note two nice and intuitive insights from Table~\ref{tab:probs}. Firstly, the infection probability becomes higher in more crowded systems, as customers must wait longer and thus have a higher incidence rate of sufficiently long overlaps that their probability at least one such overlap is with an infectious customer increases. Further, the infection probability rises with the proportion of infectious arrivals of the system.

\FloatBarrier
We will now consider the number of infections such a customer receives, as a proxy for severity of infection in Table~\ref{tab:exactk}. For this simulation, note we will only consider 100,000 customer per each experiment.}

\begin{table}[h]
    \centering
    \small
    \begin{tabular}{|c|c|c|c|c|c|c|c|}
    \hline
   $p$ & $\lambda$ &  \makecell{$\mathbb{P}(I)$\\$k=1$\\(Sim.)} & \makecell{$\mathbb{P}(I)$\\$k=1$\\(Eqn.~\ref{eqn:exactkinfec})}&  \makecell{$\mathbb{P}(I)$\\$k=2$\\(Sim.)} &\makecell{$\mathbb{P}(I)$\\$k=2$\\(Eqn.~\ref{eqn:exactkinfec})}& \makecell{$\mathbb{P}(I)$\\$k=3$\\(Sim.)}& \makecell{$\mathbb{P}(I)$\\$k=3$\\(Eqn.~\ref{eqn:exactkinfec})} \\\hline
.01 & .50 & .0148 & .0150         & .00035    & .00029            & .00002    & .00001            \\
.01 & .75   & .0466 & .0471         & .00244    & .00267            & .00009    & .00015            \\
.01 & .90   & .124  & .123          & .0192 & .0188         & .00346    & .00286            \\ \hline
.02 & .50   & .0293 & .0288         & .00121    & .00111            & .00008    & .00004            \\
.02 & .75   & .0838 & .0844         & .00887    & .00905            & .00088    & .00097            \\
.02 & .90   & .182  & .185          & .0499 & .0490         & .0113 & .0130         \\ \hline
.10 & .50   & .107  & .108          & .0180 & .0180         & .00308    & .00300            \\
.10 & .75   & .207  & .207          & .0786 & .0776         & .0298 & .0291         \\
.10 & .90   & .218  & .218 & .142      & .140    & .0883     & .0903   \\ \hline
.50 & .50       & .195      & .195    & .0970     & .0974   & .0490     & .0487   \\
.50 & .75       & .165      & .165    & .124      & .124    & .0923     & .0931   \\
.50 & .90       & .0859     & .0856   & .0771     & .0771   & .0688     & .0693   \\ \hline    \end{tabular}
    \caption{For a susceptible arrival in steady state, the theoretical and simulated probabilities that they are infected by $k$ different customers.}
    \label{tab:exactk}
\end{table}

{\color{black} We again can draw some intuitive insights from the results in Table~\ref{tab:exactk}. We note the geometric relationship between the results, as we saw in Figure~\ref{fig:transmissions}, which we understand as naturally coming from this process being analogous to confirming whether the time between multiple infectious arrivals occurs within the tagged customer's infection window. We additionally note that as the queue utility increases, the probability of different numbers of infection transmissions becomes more uniform, as customer overlaps increase in frequency and length, such that there is less of a sharp contrast between the likelihoods of overlapping with one customer and of overlapping with more.

\FloatBarrier
Related to this perspective of infection transmitting overlaps for a susceptible arrival is the perspective of an infectious arrival, so we now explore the moments of the number of infections transmitted by a single infectious customer in Table~\ref{tab:meanvar}.

\begin{table}[h]
\centering
\small
\begin{tabular}{|c|c|c|c|c|c|}
\hline
 $p$ & $\lambda$  &\makecell{$\mathbb{E}[N_I]$\\ (Simulation)}& \makecell{$\mathbb{E}[N_I]$\\ (Equation~\ref{eqn:meannuminfec})}& \makecell{Var$[N_I]$\\(Simulation)} & \makecell{Var$[N_I]$\\ (Equation~\ref{eqn:varnuminfec})}  \\\hline
.01 & .50 & 1.573 & 1.542 & 5.422 & 5.271 \\
.01 & .75 & 5.274 & 5.242 & 39.889 & 40.038      \\
.01 & .90 & 17.354 & 16.951 & 371.666 & 333.748  \\ \hline
.02 & .50 & 1.532 & 1.526 & 5.259 & 5.180 \\
.02 & .75 & 5.225 & 5.189 & 40.449          & 39.286 \\
.02 & .90 & 16.780 & 16.780 & 317.466 & 327.209 \\ \hline
.10 & .50 & 1.392 & 1.402 & 4.442 & 4.483 \\
.10 & .75 & 4.729 & 4.765 & 32.889 & 33.523 \\
.10 & .90 & 15.502 & 15.410 & 281.787 & 277.226 \\ \hline
.50 & .50 & 0.780 & 0.779 & 1.725 & 1.730 \\
.50 & .75 & 2.669 & 2.647 & 11.601 & 11.523 \\
.50 & .90 & 8.533 & 8.561 & 88.752 & 89.368       \\ \hline\end{tabular}
\caption{For an infectious arrival in steady state, the theoretical and simulated mean and variance of the number of customers they infect.}
\label{tab:meanvar}
\end{table} 

{\color{black}
In Table~\ref{tab:meanvar} we note a few significant features. Firstly, we see, as expected, a more pronounced accuracy decrease as the utility increases than for the probabilities of infection in Table~\ref{tab:probs}. Similarly, we note that the moments are much closer for lower utility, which follows straightforwardly from the lower volatility of lower utility systems. Finally, we note that these values decrease as the proportion of infectious arrivals of the system increases. This is intuitive, as a higher proportion of infectious arrivals causes there to be fewer susceptible customers in the system for any given infectious customer to infect. 

\FloatBarrier
We additionally want to highlight how these results hold consistently even for infection thresholds that are low or high in comparison to average sojourn time. To that end, we define $\ell_{\text{low}} = 0.1$ and $\ell_{\text{high}} = 1$ and repeat our prior experiment in Tables~\ref{tab:elllowvar} and \ref{tab:ellhighvar} respectively.}

\begin{table}[h]
\centering
\small
\color{black}
\begin{tabular}{|c|c|c|c|c|c|}
\hline
 $p$ & $\lambda$  &\makecell{$\mathbb{E}[N_I]$\\ (Simulation)}& \makecell{$\mathbb{E}[N_I]$\\ (Equation~\ref{eqn:meannuminfec})}& \makecell{Var$[N_I]$\\(Simulation)} & \makecell{Var$[N_I]$\\ (Equation~\ref{eqn:varnuminfec})}  \\\hline
.01 & .50 & 1.901 & 1.883 & 5.845  & 5.795 \\
.01 & .75 & 5.808 & 5.793 & 40.673 & 41.055 \\
.01 & .90 & 18.345 & 17.643 & 376.002 & 335.164 \\ \hline
.02 & .50 & 1.867 & 1.864 & 5.729 & 5.697 \\
.02 & .75 & 5.739 & 5.735 & 40.175 & 40.288 \\
.02 & .90 & 17.593 & 17.464 & 337.260 & 328.603 \\ \hline
.10 & .50 & 1.719 & 1.712 & 4.979 & 4.945 \\
.10 & .75 & 5.203 & 5.267 & 33.000 & 34.409 \\
.10 & .90 & 15.769 & 16.039 & 268.624 &  278.453 \\ \hline
.50 & .50 & 0.946 & 0.951 & 1.927 & 1.949 \\
.50 & .75 & 2.939 & 2.926 & 12.169 & 11.920 \\
.50 & .90 & 8.933 & 8.910 & 91.037 & 89.902 \\ \hline
\end{tabular}\captionsetup{font={color=black}}\caption{For an infectious arrival in steady state with infection transmission threshold $\ell_{\text{low}}$,  the theoretical and simulated mean and variance of the number of customers they infect.}
\label{tab:elllowvar}
\end{table} 

\begin{table}[h]
\centering
\small
\color{black}
\begin{tabular}{|c|c|c|c|c|c|}
\hline
 $p$ & $\lambda$  &\makecell{$\mathbb{E}[N_I]$\\ (Simulation)}& \makecell{$\mathbb{E}[N_I]$\\ (Equation~\ref{eqn:meannuminfec})}& \makecell{Var$[N_I]$\\(Simulation)} & \makecell{Var$[N_I]$\\ (Equation~\ref{eqn:varnuminfec})}  \\\hline
.01 & .50 & 1.227 & 1.201 & 4.670 & 4.514 \\
.01 & .75 & 4.598 & 4.626 & 37.130 & 38.183 \\
.01 & .90 & 16.045 & 16.124 & 337.237 & 330.801 \\ \hline
.02 & .50 & 1.184 & 1.189 & 4.356 & 4.436 \\
.02 & .75 & 4.565 & 4.579 & 37.031 & 37.462 \\
.02 & .90 & 16.058 & 15.961 & 335.921 & 324.313 \\ \hline
.10 & .50 & 1.086 & 1.092 & 3.787 & 3.830 \\
.10 & .75 & 4.206 & 4.206 & 31.652 & 31.939 \\
.10 & .90 & 14.626 & 14.658 & 276.681 & 274.722 \\ \hline
.50 & .50 & 0.609 & 0.607 & 1.467 & 1.452 \\
.50 & .75 & 2.378 & 2.336 & 11.285 & 10.896 \\
.50 & .90 & 8.224 & 8.144 & 89.401 & 88.410 \\ \hline
\end{tabular}\captionsetup{font={color=black}}\caption{For an infectious arrival in steady state with infection transmission threshold $\ell_{\text{high}}$,  the theoretical and simulated mean and variance of the number of customers they infect.}
\label{tab:ellhighvar}
\end{table} 

{\color{black}
Comparing Tables~\ref{tab:elllowvar}, \ref{tab:meanvar}, and \ref{tab:ellhighvar}, note that both moments decrease as $\ell$ increases, which we again intuitively understand since higher thresholds cause greater difficulty in transmitting infection. Additionally note the higher accuracy of the simulation results for higher thresholds, due to the lower volatility of the number of sufficiently long overlaps when we restrict our consideration to a smaller subset of the right tail of the overlap distribution.
\FloatBarrier
It is now of interest to move beyond this consideration of different fixed thresholds to experiments with varying and random thresholds.

\subsubsection{Random infection threshold}\label{sec:random-thresh}
We have discussed throughout the paper the manner in which our results can be extended to settings with a random infection threshold per tagged customer (which throughout our analysis has generally meant a susceptible customer worrying about how many customers will infect them, but can analogously mean an infectious customer exploring how many different customers they will infect). We explore this phenomenon experimentally in the following simulations by comparing five different experiments using the two threshold values $\ell_{\text{low}}$ and $\ell_{\text{high}}$:
\begin{enumerate}
    \item \label{exp:1} Each susceptible customer's infection threshold is randomly chosen to be $\ell_{\text{low}}$ or $\ell_{\text{high}}$ with equal probability;
    \item \label{exp:2} Each infectious customer's infecting threshold is randomly chosen to be $\ell_{\text{low}}$ or $\ell_{\text{high}}$ with equal probability;
    \item \label{exp:3} The average across two experiments, one of which has constant threshold $\ell_{\text{low}}$, the other of which has constant threshold $\ell_{\text{high}}$ (note for finding the moments here we directly combine our results from Tables~\ref{tab:elllowvar} and \ref{tab:ellhighvar};
    \item \label{exp:4} An experiment where we use a fixed constant threshold value which is the mean of $\ell_{\text{low}}$ and $\ell_{\text{high}}$, that is $\ell_{\text{mean}} = 0.55$;
    \item \label{exp:5} And an experiment where a random infection threshold is chosen pairwise for each individual susceptible-infectious customer pairing. With equal probability this threshold is chose to be $\ell_{\text{low}}$ or $\ell_{\text{high}}$.
\end{enumerate}
For each experiment we take the average of five simulations over 500,000 customers across each set of parameters. We again keep $\mu=1$ for all experiments.

In Table~\ref{tab:compareaverageprobs} we explore the probability of infection across the five different experiments. In Table~\ref{tab:compareaveragemeans} we explore the expected number of customer infected across the different experiments, and in Table~\ref{tab:compareaveragevars} we find the variance of this value. Finally, in Table~\ref{tab:comparedeltapercents}, using Experiment~\ref{exp:4} as a control we consider the accuracy of the other four experiments.

Within this exploration, the following conjecture is worth noting.

\begin{conjecture}\label{conj:notinfectedprob}
    For a queueing system satisfying \fcfs discipline with random infection threshold whose distribution $L$ is allowed to vary for each susceptible-infectious pair of customers, the probability of \textit{not} becoming infected is \begin{eqnarray}\label{eqn:conjecture-pairwise}
    \mathbb{E}[e^{-2\lambda p(\mathcal{S}-L)^+}].\end{eqnarray}
\end{conjecture}

 The intuition behind this conjecture is as follows. When we compare the thinned infectious arrival process, then $\mathbb{E}\left[e^{-\lambda p(\mathcal{S}-L)^+}\right]$ gives us the after arrival probability that no infectious customers arrive while the tagged customer has remaining sojourn time of at least $\ell$. Under Conjecture~\ref{conj:symmetry}, the upon arrival probability of no infection is the same, and so we get the proposed equation. Further, by Jensen's Inequality, for a given mean $\ell^{\star}$ of a threshold distribution, a  constant threshold of $\ell^{\star}$ is the best we can do.}

{\begin{table}[h]
\centering
\makebox[\linewidth][c]{
\footnotesize
\color{black}
\begin{tabular}{|c|c|c|c|c|c|c|}
\hline
 $p$ & $\lambda$  &\makecell{Exp.~\ref{exp:1} $\mathbb{P}(I)$ \\(Random suscep.)}& \makecell{Exp.~\ref{exp:2} $\mathbb{P}(I)$ \\(Random infec.)}& \makecell{Exp.~\ref{exp:3} $\mathbb{P}(I)$\\(Avg. constant)}& \makecell{Exp.~\ref{exp:4} $\mathbb{P}(I)$\\(Fixed mean)}& \makecell{Exp.~\ref{exp:5} $\mathbb{P}(I)$\\(Random pairwise)} \\\hline
.01 & .50 & .0154 & .0153 & .0152 & .0147 &.0152 \\
.01 & .75 & .0497 & .0490 & .0498 & .0496 &.0501 \\
.01 & .90 & .143 & .146 & .143 & .144 &.148 \\ \hline
.02 & .50 & .0295 & .0297 & .0300 & .0298 & .0304 \\
.02 & .75 & .0943 & .0935 & .0934 & .0946 & .0937 \\
.02 & .90 & .250 & .248 & .250 & .251 & .251 \\\hline
.10 & .50 & .130 & .130 & .130 & .127 & .131 \\
.10 & .75 & .327 & .330 & .329 & .325 & .332 \\
.10 & .90 & .610 & .608 & .609 & .611 & .610 \\ \hline
.50 & .50 & .390 & .397 & .389 & .380 & .397 \\
.50 & .75 & .658 & .668 & .658 & .654 & .667 \\
.50 & .90 & .852 & .858 & .853 & .852 & .860 \\ \hline
\end{tabular}}\captionsetup{font={color=black}}\caption{For a susceptible arrival in steady state the simulated probability they get infected across different experiment set-ups.}
\label{tab:compareaverageprobs}
\end{table}}

{\color{black}
From Table~\ref{tab:compareaverageprobs}, note that across all five experiments we achieve strikingly similar probabilities of susceptible customers becoming infected. Comparing against Experiment~\ref{exp:4}, where we have a fixed threshold of $\ell_{\text{mean}} = 0.55$ for all customers, we note that the maximum average percentage delta between this experiment and another occurs for Experiment~\ref{exp:5} where it is .0192 and the minimum average percentage delta between this experiment and another occurs for Experiment~\ref{exp:1} where it is .0115.}

{\begin{table}[h]
\centering
\makebox[\linewidth][c]{
\footnotesize
\color{black}
\begin{tabular}{|c|c|c|c|c|c|c|}
\hline
 $p$ & $\lambda$  &\makecell{Exp.~\ref{exp:1} $\mathbb{E}[N_I]$\\(Random suscep.)}& \makecell{Exp.~\ref{exp:2} $\mathbb{E}[N_I]$\\(Random infec.)}& \makecell{Exp.~\ref{exp:3} $\mathbb{E}[N_I]$\\(Avg. constant)}& \makecell{Exp.~\ref{exp:4} $\mathbb{E}[N_I]$\\(Fixed mean)}& \makecell{Exp.~\ref{exp:5} $\mathbb{E}[N_I]$\\(Random pairwise)} \\\hline
.01 & .50 & 1.558 & 1.524 & 1.564 & 1.504 & 1.571 \\
.01 & .75 & 5.193 & 5.133 & 5.203 & 5.207 & 5.208 \\
.01 & .90 & 17.143 & 17.061 & 17.195 & 16.877 & 16.919 \\ \hline
.02 & .50 & 1.524 & 1.531 & 1.526 & 1.481 & 1.544 \\
.02 & .75 & 5.143 & 5.180 & 5.152 & 5.088 & 5.082 \\
.02 & .90 & 16.696 & 16.655 & 16.826 & 16.405 & 16.652 \\ \hline
.10 & .50 & 1.393 & 1.402 & 1.403 & 1.369 & 1.399 \\
.10 & .75 & 4.697 & 4.763 & 4.705 & 4.695 & 4.740 \\
.10 & .90 & 15.250 & 15.229 & 15.198 & 15.487 & 15.569 \\ \hline
.50 & .50 & .778 & .776 & .778 & .758 & .778 \\
.50 & .75 & 2.606 & 2.630 & 2.659 & 2.639 & 2.617\\
.50 & .90 & 8.413 & 8.632 & 8.579 & 8.605 & 8.323 \\ \hline
\end{tabular}}\captionsetup{font={color=black}}\caption{For an infectious arrival in steady state the simulated mean number of customers they infect across different experiment set-ups.}
\label{tab:compareaveragemeans}
\end{table}}
{\color{black}
We again see in Table~\ref{tab:compareaveragemeans} similar expected numbers of infection transmissions across the five experiments. Using Experiment~\ref{exp:4} as the control, we note that the maximum average percentage delta between this experiment and another occurs for Experiment~\ref{exp:3} where it is .0176 and the minimum average percentage delta between this experiment and another occurs for Experiment~\ref{exp:2} where it is .0159.}

{\begin{table}[h]
\centering
\makebox[\linewidth][c]{
\footnotesize
\color{black}
\begin{tabular}{|c|c|c|c|c|c|c|}
\hline
 $p$ & $\lambda$  & \makecell{Exp.~\ref{exp:1} Var$[N_I]$\\(Random suscep.)} & \makecell{Exp.~\ref{exp:2} Var$[N_I]$\\(Random infec.)} & \makecell{Exp.~\ref{exp:3} Var$[N_I]$\\(Avg. constant)}  & \makecell{Exp.~\ref{exp:4} Var$[N_I]$\\(Fixed mean)}  & \makecell{Exp.~\ref{exp:5} Var$[N_I]$\\(Random pairwise)} \\\hline
.01 & .50 & 5.270 & 5.268 & 5.258 & 5.165 & 5.167 \\
.01 & .75 & 39.589 & 39.413 & 38.902 & 40.922 & 38.925 \\
.01 & .90 & 352.070 & 339.784 & 356.620 & 326.246 & 324.391 \\ \hline
.02 & .50 & 4.951 & 5.165 & 5.043 & 5.012 & 5.182 \\
.02 & .75 & 38.886 & 39.698 & 38.603 & 38.559 & 37.006 \\
.02 & .90 & 329.550 & 316.548 & 336.591 & 312.194 & 317.909 \\ \hline
.10 & .50 & 4.340 & 4.471 & 4.383 & 4.418 & 4.344 \\
.10 & .75 & 32.967 & 34.182 & 32.326 & 33.505 & 32.833 \\
.10 & .90 & 280.738 & 283.554 & 272.653 & 281.133 & 291.078 \\ \hline
.50 & .50 & 1.682 & 1.727 & 1.697 & 1.710 & 1.679 \\
.50 & .75 & 11.153 & 11.460 & 11.727 & 11.619 & 11.227 \\
.50 & .90 & 86.061 & 91.264 & 90.219 & 92.411 & 81.949 \\ \hline
\end{tabular}}\captionsetup{font={color=black}}\caption{For an infectious arrival in steady state the simulated variance of the number of customers they infect across different experiment set-ups.}
\label{tab:compareaveragevars}
\end{table}}

{\color{black}
In Table~\ref{tab:compareaveragevars}, we again achieve similar expected number of infection transmissions across the five experiment. Using Experiment~\ref{exp:4} as the control again, we note that the maximum average percentage delta between this experiment and another occurs for Experiment~\ref{exp:5} where it is .0321 and the minimum average percentage delta between this experiment and another occurs for Experiment~\ref{exp:2} where it is .0208.

A summary of the average percent deltas between  Experiment~\ref{exp:4} and the other experiments is in Table~\ref{tab:threshdeltapercents}.

\begin{table}[H]
\centering
\small
\color{black}
\begin{tabular}{|c|c|c|c|}
\hline
 Experiment & $\delta$ of $\mathbb{P}(I)$ & $\delta$ of $\mathbb{E}[N_I]$ & $\delta$ of Var$[N_I]$ \\\hline
 Random susceptible threshold & .0115 & .0172 & .0307 \\
 Random infectious threshold & .0176 & .0159 & .0208 \\
 Average over two constant thresholds & .0115 & .0176 & .0300 \\
 Random pairwise thresholds & .0192 & .0175 & .0321 \\ \hline
\end{tabular}\captionsetup{font={color=black}}\caption{Across our experiments, comparing the average percentage difference ($\delta$) between theoretical and experimental results to explore relative accuracy across different settings.}
\label{tab:threshdeltapercents}
\end{table}

We note from Table~\ref{tab:threshdeltapercents} that the accuracy for each case is within the same order of magnitude between the different experiment, with Experiment~\ref{exp:5}, the pairwise threshold case, having consistently higher percent difference and the single random threshold experiments having near-consistent lower percent differences. This demonstrates how the influence of the threshold distribution on the spread of infection within the queueing system is largely determined by the expected value of the distribution.}

{\color{black}
\subsection{Extension to $G/M/1$}
We can extend our work to the case of the G/M/1 queue. Our main insight is that the G/M/1 queue also has a geometric($\rho^*$) steady state queue length distribution. However, the parameter $\rho^*$ is the solution to the following equation 
\begin{eqnarray*}
\rho^* &=& \mathcal{L}_A(\mu - \mu \rho^*) .
\end{eqnarray*} 
where $\mathcal{L}_A(x)$ is the Laplace transform of the inter-arrival distribution. }

{\color{black}
\begin{remark}
In the case where the arrival process is an Erlang(2,$\lambda$), yielding an $E_2/M_{\mu}/1$ queue, we have the fixed point for \begin{eqnarray}\label{eqn:rho-erlang}
    \rho^* = \frac{\lambda}{\mu} + \frac{1}{2} - \sqrt{\frac{\lambda}{\mu} + \frac{1}{4}}.
\end{eqnarray}

Moreover, when the arrival process is deterministic, yielding a $D/M_{\mu}/1$ queue, we have the fixed point for \begin{eqnarray}\label{eqn:rho-deterministic}
    \rho^* = -\frac{W_{-1}( - e^{-\Delta \mu} \Delta \mu )}{\Delta \mu}
\end{eqnarray} where $W_{-1}(x)$ is defined as the principal branch of the Lambert $W$ function when $-1/e \leq x \leq 0$. This is the case because
\begin{eqnarray*}
\rho^* &=& \mathcal{L}_A(\mu - \mu \rho^*) \\
&=& e^{-\Delta (\mu - \mu \rho^*)} \\
&=& e^{-\Delta \mu} e^{\Delta \mu \rho^*},
\end{eqnarray*}
and by a generalization of the Lambert $W$ function, we have the above result for $\rho^*$.

Finally, when the arrival process is Hyper-Exponential$(2,\lambda_1,\lambda_2)$, yielding a $H_2/M_{\mu}/1$ queue, we get
\begin{eqnarray}\label{eqn:rho-hyperexp}
    \rho^* = \frac{\lambda_1+\lambda_2+\mu - \sqrt{\lambda_1^2+\lambda_2^2+\mu^2-2\lambda_1\lambda_2-2(2p_1-1)(\lambda_1-\lambda_2)\mu}}{2\mu}.
\end{eqnarray}This is the case as,
\begin{eqnarray*}
    \rho^* &=& \mathcal{L}_A(\mu - \mu \rho^*)\\
        &=& \int_0^{\infty}p_1\lambda_1e^{-\lambda_1 x}e^{-(\mu-\mu\rho^*)x}dx+\int_0^{\infty}(1-p_1)\lambda_2e^{-\lambda_2 x}e^{-(\mu-\mu\rho^*)x}dx\\
        &=& \frac{\lambda_1p_1}{\mu-\mu\rho^*+\lambda_1}+\frac{\lambda_2(1-p_1)}{\mu-\mu\rho^*+\lambda_2}. \end{eqnarray*}
Solving this cubic, we get only one real solution, giving us the above result for $\rho^*$.
\end{remark}

\begin{theorem}\label{thm:gm1infecprob}
For a G/M/1 queue in steady-state the probability of infection for a susceptible customer is,
\begin{eqnarray*}
    \mathbb{P}(I) &=&  \mathbb{P}(I_{ua} \cup I_{aa}) =  \frac{ 2 \rho^* p}{1 - \rho^* + 2 \rho^* p}  e^{- \mu (1 - \rho^*)  \ell},
\end{eqnarray*}
\end{theorem}
\begin{proof}
    We see that this extension is clear in the upon arrival case. Since the analysis depends upon there existing an infectious customer with sufficient remaining sojourn time (that is a sum of \textit{service} times) to infect the tagged customer, changing the arrival regime does not undermine the analysis.
    
    For the after arrival case, we can no longer first thin the arrival process. Rather, we recognize that we can count the arrivals within the timeframe of $\ell$ less than the tagged customer's sojourn window. At this time we then binomially thin the count to see the probability even a single one of those customers is infectious. We can think of this as a process that is geometric over the inter-arrival times: finding the first inter-arrival which is an infectious person, and then determining if this occurs within $\ell$ less than the sojourn time.
    
    That the upon and after arrival cases are still independent given the sojourn time is evident by the analysis in the $M/M/1$ case, completing the claim.
\end{proof}}
{\color{black}
For the mean and variance of the number of infections we have,
\begin{corollary} \label{cor:gm1meaninfec}For a $G/M/1$ queue in steady-state, the expected number of infections a single infectious customer causes is
    \begin{eqnarray*}
 \mathbb{E} \left[ N_I \right] &=& \frac{2\rho^*(1-p)}{1-\rho^*} e^{-\mu(1-\rho^*) \ell},\text{ and}    
\end{eqnarray*}
\end{corollary}
\begin{corollary}\label{cor:gm1varinfec}
    For a $G/M/1$ queue in steady-state, the variance of the number of infections a single infectious customer causes is
    \begin{eqnarray*}
          \mathrm{Var} \left[ N_I \right] &=& \frac{2 \rho^* (1-p) \left( 1- \rho^* + 4 \rho^* (1-p) \right)}{(1-\rho^*)^2} e^{-\mu(1-\rho^*) \ell} \nonumber \\&&- \frac{4\rho^{*^2} (1-p)^2}{(1-\rho^*)^2} e^{-2\mu (1-\rho^*) \ell}. 
    \end{eqnarray*}
\end{corollary}
For the moments of the number of infection transmissions, the validity of substituting $\rho^*$ for $\lambda$ follows straightforwardly from these being observation counts.

\subsubsection{Numerical results}

{\color{black} Our simulation framework is as follows:}

For each inter-arrival distribution, we simulated the queue over a lifetime of 500,000 customers. For each set of parameters, we considered the average of five simulations. For all cases, we fixed $\mu = 1$ and $\ell = 0.5$. {\color{black}Again, note that as $\frac{\lambda}{\mu}\rightarrow 1$, it takes longer to reach a steady state, hence we expect to see less accurate simulations as $\lambda$ increases, especially for the higher moments. In Table~\ref{tab:deterministic} we explore the $G/M/1$ queue with deterministic($\lambda$) inter-arrivals. In Table~\ref{tab:erlang} we explore the $G/M/1$ queue with Erlang$\left(2,\frac{\lambda}{2}\right)$ inter-arrivals. In Table~\ref{tab:hyper} we explore the $G/M/1$ queue with Hyper-Exponential$(2,\lambda_1=1,\lambda_2=0.25)$ inter-arrivals. Finally, in Table~\ref{tab:comparedeltapercents} we compare the across the four different arrival processes the average percent error between the simulation and theoretical results.}

Firstly, for a deterministic inter-arrival distribution we simply allowed the parameter $\Delta$ to take on the same values as $\lambda$ did for the $M/M/1$ experiments, in order to preserve the same means for the sake of comparison.

\begin{table}[h]
\centering
\small
\begin{tabular}{|c|c|c|c|c|c|c|c|}
\hline
$p$ & $\Delta$ & \makecell{$\mathbb{P}(I)$\\ (Sim.)}&  \makecell{$\mathbb{P}(I)$\\ (Thm.~\ref{thm:gm1infecprob})}&\makecell{$\mathbb{E}[N_I]$\\ (Sim.)} & \makecell{$\mathbb{E}[N_I]$\\ (Cor.~\ref{cor:gm1meaninfec})}  & \makecell{Var$[N_I]$\\ (Sim.)}  & \makecell{Var$[N_I]$\\ (Cor.~\ref{cor:gm1varinfec})} \\\hline
.01 & .50 & .00337 & .00341 & .337 & .339 & .657 & .566\\

.01 & .75 & .0187 & .0187 & 1.899 & 1.894 & 8.197 & 7.313\\

.01 & .90 & .0713 & .0700 & 7.719 & 7.512 & 84.515 & 75.387\\\hline

.02 & .50 & .00677 & .00678 & .334 & .336 & .628 & .558\\

.02 & .75 & .0362 & .0365 & 1.864 & 1.875 & 7.864 & 7.185\\

.02 & .90 & .128 & .130 & 7.380 & 7.436 & 73.380 & 73.947\\\hline

.10 & .50 & .0318 & .0326 & .305 & .308 & .562 & .496\\

.10 & .75 & .152 & .154 & 1.717 & 1.722 & 6.793 & 6.200\\

.10 & .90 & .412 & .413 & 6.928 & 6.829 & 66.584 & 62.924\\\hline

.50 & .50 & .129 & .136 & .171 & .171 & .249 & .229\\

.50 & .75 & .416 & .435 & .957 & .957 & 2.542 & 2.339\\

.50 & .90 & .719 & .733 & 3.844 & 3.794 & 22.253 & 21.107\\\hline
\end{tabular}
    \caption{Simulation results for an arrival in the $D/M_{\mu}/1$ queue in steady state.}
    \label{tab:deterministic}
\end{table}

{\color{black} Note from Table~\ref{tab:deterministic} that the values are consistently lower than in the $M/M/1$ case from Tables~\ref{tab:probs} and \ref{tab:meanvar}, and, for lower utilization especially, the values are \textit{substantially} lower. This is intuitive as the lessened volatility from lower utilization combines with the stability of deterministic arrivals to ensure a lower incidence of long (and thus potentially infecting) overlaps.}

Next, for an Erlang-2 distribution, in order to preserve the means between different arrival processes, we scaled the values of $\lambda$ from the $M/M/1$ experiments by a factor of 2.

\begin{table}[h]
\centering
\small
\begin{tabular}{|c|c|c|c|c|c|c|c|}
\hline
$p$ & $\frac{\lambda}{2}$ & \makecell{$\mathbb{P}(I)$\\ (Sim.)}&  \makecell{$\mathbb{P}(I)$\\ (Thm.~\ref{thm:gm1infecprob})}&\makecell{$\mathbb{E}[N_I]$\\ (Sim.)} & \makecell{$\mathbb{E}[N_I]$\\ (Cor.~\ref{cor:gm1meaninfec})}  & \makecell{Var$[N_I]$\\ (Sim.)}  & \makecell{Var$[N_I]$\\ (Cor.~\ref{cor:gm1varinfec})} \\\hline
.01 & .50 & .00899 & .00896 & .902 & .898 & 2.455 & 2.290\\

.01 & .75 & .0344 & .0343 & 3.537 & 3.533 & 21.143 & 20.393\\

.01 & .90 & .107 & .109 & 12.096 & 12.213 & 177.829 & 181.688\\\hline

.02 & .50 & .0176 & .0177 & .889 & .889 & 2.377 & 2.253\\

.02 & .75 & .0658 & .0659 & 3.507 & 3.498 & 21.240 & 20.018\\

.02 & .90 & .194 & .195 & 11.979 & 12.090 & 170.818 & 178.158\\\hline

.10 & .50 & .0802 & .0808 & .819 & .817 & 2.086 & 1.967\\

.10 & .75 & .249 & .251 & 3.198 & 3.212 & 17.605 & 17.145\\

.10 & .90 & .531 & .532 & 11.180 & 11.103 & 153.741 & 151.165\\\hline

.50 & .50 & .275 & .280 & .453 & .454 & .840 & .809\\

.50 & .75 & .568 & .576 & 1.796 & 1.785 & 6.328 & 6.085\\

.50 & .90 & .807 & .813 & 6.200 & 6.168 & 50.596 & 49.397\\\hline

\end{tabular}
    \caption{Simulation results for an arrival in the $E_2/M_{\mu}/1$ queue in steady state.}
    \label{tab:erlang}
\end{table}

{\color{black}
We can again see from Table~\ref{tab:erlang} that the Erlang-2 arrival process ensures lower infection probabilities and moments of the number of infections as compared to a Poisson process demonstrated in Tables~\ref{tab:probs} and \ref{tab:meanvar}, while as expected the values are nonetheless higher than in a deterministic arrival process, as in Table~\ref{tab:deterministic}. We overall expect these results due to deterministic processes being most stable, while the Erlang-2 generalization of the exponential arrival process increases the variability above that of a deterministic setting, but still not as much as in the exponential cases.}

Finally, for Hyper-Exponential(2) inter-arrival times, with probability $p_1$ we drew from an exponential distribution with parameter $\lambda_1=1$, and respectively drew with probability $1-p_1$ with parameter $\lambda_2=0.25$, such that we vary $p_1$ as $2/3, 8/9,$ and $26/27$ to achieve equivalent means as in the $M/M/1$ experiments.

\begin{table}[h]
\centering
\small
\begin{tabular}{|c|c|c|c|c|c|c|c|}
\hline
$p$ & $p_1$ & \makecell{$\mathbb{P}(I)$\\ (Sim.)}&  \makecell{$\mathbb{P}(I)$\\ (Thm.~\ref{thm:gm1infecprob})}&\makecell{$\mathbb{E}[N_I]$\\ (Sim.)} & \makecell{$\mathbb{E}[N_I]$\\ (Cor.~\ref{cor:gm1meaninfec})}  & \makecell{Var$[N_I]$\\ (Sim.)}  & \makecell{Var$[N_I]$\\ (Cor.~\ref{cor:gm1varinfec})} \\\hline
.01 & $2/3$ & .0250 & .0249 & 2.544 & 2.544 & 11.266 & 11.803\\

.01 & $8/9$ & .0718 & .0716 & 7.673 & 7.699 & 76.343 & 78.758\\

.01 & $26/27$ & .174 & .173 & 21.010 & 20.891 & 523.514 & 494.469\\\hline

.02 & $2/3$ & .0481 & .0484 & 2.508 & 2.518 & 11.132 & 11.591\\

.02 & $8/9$ & .131 & .133 & 7.550 & 7.621 & 74.114 & 77.252\\

.02 & $26/27$ & .293 & .293 & 20.541 & 20.680 & 471.773 & 484.740\\\hline

.10 & $2/3$ & .196 & .196 & 2.316 & 2.312 & 9.647 & 9.965\\

.10 & $8/9$ & .422 & .419 & 7.069 & 6.999 & 66.725 & 65.725\\

.10 & $26/27$ & .658 & .659 & 18.941 & 18.992 & 406.854 & 410.379\\\hline

.50 & $2/3$ & .508 & .502 & 1.290 & 1.285 & 3.573 & 3.646\\

.50 & $8/9$ & .742 & .737 & 3.912 & 3.888 & 21.850 & 22.014\\

.50 & $26/27$ & .882 & .879 & 10.539 & 10.551 & 128.780 & 131.350\\\hline

\end{tabular}
    \caption{Simulation results for an arrival in the $H_2/M_{\mu}/1$ queue in steady state.}
    \label{tab:hyper}
\end{table}

{\color{black} Looking at Table~\ref{tab:hyper}, the Hyper-Exponential(2) arrival process clearly gives us the highest infection probabilities and moments of the number of infections in comparison to the other three arrival processes explored in Tables~\ref{tab:probs}, \ref{tab:meanvar}, \ref{tab:deterministic}, and \ref{tab:erlang}. This continues the analysis that demonstrates low stability and higher variability within an arrival process contributes predominantly to increase the infection probability in a system. We see this occurring as a consequence of lighter left-tails and heavier right tails in the inter-arrival distributions causing an increase in the variance in the overlap times, and especially in the right tail of the overlap time distribution (so there are both more and longer overlaps), thus increasing the number of overlaps that are long enough to cause an infection should the customers form a susceptible-infectious pair.

As an additional consequence of this volatility, we expect a similar trend in the accuracy of the simulations after 500,000 arrivals; that is in part, the reasonableness of assuming that 500,000 customers allows the queue to achieve a steady-state. A brief insight into this question is handled in Table~\ref{tab:comparedeltapercents} where, for each of the four inter-arrival disciplines, we compare the average of the normalized absolute difference between the theoretical and simulation values of the infection probability and of both moments of the number infected reported across Tables~\ref{tab:probs}, \ref{tab:meanvar}, \ref{tab:deterministic}, \ref{tab:erlang}, and \ref{tab:hyper}.

\begin{table}[h]
\centering
\small
\color{black}
\begin{tabular}{|c|c|c|c|}
\hline
 Setting & $\delta$ of $\mathbb{P}(I)$ & $\delta$ of $\mathbb{E}[N_I]$ & $\delta$ of Var$[N_I]$ \\\hline
 $M/M/1$ & .00399 & .00786 & .0235 \\
 $D/M/1$ & .0175 & .00822 & .0955 \\
 $E_2/M/1$ & .00778 & .00452 & .0412 \\
 $H_2/M/1$ & .00539 & .00456 & .0287 \\ \hline
\end{tabular}\captionsetup{font={color=black}}\caption{Across our experiments, comparing the average percentage difference ($\delta$) between theoretical and experimental results to explore relative accuracy across different settings.}
\label{tab:comparedeltapercents}
\end{table}

The key note from these comparisons to take from Table~\ref{tab:comparedeltapercents} is that the accuracy is largely within the same order of magnitude between the different regimes, with a slight bias towards the higher variability settings.}

\subsection{Extension to $M/G/1$}

We can also obtain an approximate expression for the infection probability in the $M/G/1$ setting by leveraging the structural decomposition of the $M/M/1$ case.  In the $M/M/1$ queue, the infection probability splits into two components:
\begin{enumerate}
    \item the probability that the sojourn time exceeds the fixed threshold~$\ell$, and
    \item the probability that the sojourn time exceeds an exponential random variable 
    with rate $2\lambda p$, which is equivalently the sojourn-time MGF evaluated at 
    $2\lambda p$:
    \[
        \mathbb{P}(\Soj > X)
            = \mathbb{E}\!\left[e^{-2\lambda p \Soj}\right]
            = M_{\Soj}(2\lambda p).
    \]
\end{enumerate}
This observation provides a natural approximation strategy for the $M/G/1$ queue, which implies that we retain the same decomposition, replacing the $M/M/1$ MGF by the Pollaczek-Khinchine MGF for the $M/G/1$, and approximate 
$\mathbb{P}(\Soj > \ell)$ using standard tail approximations.  This yields a simple and accurate approximation consistent with the structure of the exact $M/M/1$ formula.

Motivated by this structure, we approximate the M/G/1 infection probability using the same two ingredients:  
the tail probability $\mathbb{P}(\mathcal{S} > \ell)$ and the MGF of the sojourn time at $2\lambda p$.  
This yields the approximation
\begin{eqnarray}\label{eqn:mg1probapprox}    \mathbb{P}(I^c)
    \;\approx\;
    \mathbb{E}\!\left[ e^{-2\lambda p\,(\mathcal{S}-\ell)^+} \right].
\end{eqnarray}

This expression has a natural intuitive interpretation.  
From the perspective of customers who arrive after the tagged susceptible customer, the only arrivals that can infect the tagged customer are those who enter during the time window $(\mathcal{S}-\ell)^+$, where $\mathcal{S}$ is the sojourn time of the tagged customer.  
To avoid infection, no infectious customers may arrive during this window.  
Since infectious customers arrive according to a Poisson process with rate $\lambda p$, we require
\[
    \mathbb{P}\!\left( N_{\lambda p}\big( (\mathcal{S}-\ell)^+ \big) = 0 \right).
\]

We must also consider infectious customers who may have arrived \emph{before} the tagged customer, leading to an additional independent contribution of the same form.  
Combining the two independent sources yields
\begin{eqnarray*}
 \mathbb{P}\!\left(
        N_{2\lambda p}\big( (\mathcal{S}-\ell)^+ \big) = 0
    \right)
    \;=\;
    \mathbb{E}\!\left[ e^{-2\lambda p\,(\mathcal{S}-\ell)^+} \right].   
\end{eqnarray*}

Before exploring the higher moments, we first bound the expected sojourn time.\begin{proposition} \label{prop:mg1bound}
Let $\mathcal{S}$ be a nonnegative random variable and let $\ell \ge 0$.  
Then the following bounds hold:
\[
\left( \mathbb{E}[\mathcal{S}] - \ell \right)^+
\;\le\;
\mathbb{E}\!\left[(\mathcal{S} - \ell)^+\right]
\;\le\;
\mathbb{E}[\mathcal{S}] - \ell\,\mathbb{P}(\mathcal{S} > \ell)
\;\le\;
\mathbb{E}[\mathcal{S}].
\]
\begin{proof}
We start from the identity
\[
(\mathcal{S} - \ell)^+ 
= \mathcal{S} \mathbf{1}_{\{\mathcal{S} > \ell\}} 
- \ell\,\mathbf{1}_{\{\mathcal{S} > \ell\}}.
\]

Taking expectations gives
\[
\mathbb{E}\!\left[(\mathcal{S} - \ell)^+\right]
= \mathbb{E}\!\left[\mathcal{S} \mathbf{1}_{\{\mathcal{S} > \ell\}}\right]
- \ell\,\mathbb{P}(\mathcal{S} > \ell).
\]

Since 
\[
\ell\,\mathbb{P}(\mathcal{S} > \ell) 
\le 
\mathbb{E}\!\left[\mathcal{S} \mathbf{1}_{\{\mathcal{S} > \ell\}}\right]
\le 
\mathbb{E}[\mathcal{S}],
\]
we immediately obtain the upper bound
\[
0 \le 
\mathbb{E}\!\left[(\mathcal{S} - \ell)^+\right]
\le 
\mathbb{E}[\mathcal{S}] - \ell\,\mathbb{P}(\mathcal{S} > \ell).
\]

For a lower bound, note that
\[
\mathbb{E}\!\left[\mathcal{S} \mathbf{1}_{\{\mathcal{S} > \ell\}}\right]
= \mathbb{E}[\mathcal{S}]
- \mathbb{E}\!\left[\mathcal{S} \mathbf{1}_{\{\mathcal{S} \le \ell\}}\right]
\ge 
\mathbb{E}[\mathcal{S}] - \ell\big(1 - \mathbb{P}(\mathcal{S} > \ell)\big),
\]
which implies
\[
\mathbb{E}\!\left[(\mathcal{S} - \ell)^+\right]
\ge 
\mathbb{E}[\mathcal{S}] - \ell.
\]

Combining both results yields the stated inequalities.
\end{proof}
\end{proposition}}

{\color{black}For the mean number infected, we now conjecture the following expression, 
\begin{conjecture}\label{conj:mg1meaninfec}For an infectious arrival to the $M/G/1$ queue in steady state, the number of susceptible customers this tagged customer infects is
\begin{eqnarray*}
    \mathbb{E}\left[ N_I \right] &=& 2 \lambda (1-p) \mathbb{E} \left[ (\mathcal{S} - \ell)^+ \right].
\end{eqnarray*}
   If Conjecture~\ref{conj:symmetry} is true, and especially noting its second clause, then this result follows from this type of argument:
   \begin{eqnarray*}
    \mathbb{E}\left[ N_I \right] &=& \sum^{\infty}_{k=0} \mathbb{E}\left[ N_I | Q = k, \mathcal{S}_k = x \right] \mathbb{P}\left[ Q = k, \mathcal{S}_k = x \right]  \nonumber\\
    &=& 2  \sum^{\infty}_{k=0} \int^{\infty}_{\ell} \mathbb{E}\left[ N_{aa} | Q = k, \mathcal{S}_k = x \right] \mathbb{P}\left[ Q = k \right] f_{\mathcal{S}_k}(x) dx  \quad(\text{Conj.~\ref{conj:symmetry}})\nonumber\\
    &=& 2 \sum^{\infty}_{k=0} \left( \int^{\infty}_{\ell} \lambda (1-p)  (x-\ell)  f_{\mathcal{S}_k}(x) dx \right) \mathbb{P}\left[ Q = k \right] \nonumber\\
    &=& 2 \lambda (1-p)  \sum^{\infty}_{k=0} \mathbb{E} \left[ (\mathcal{S}_k - \ell)^+ \right] \mathbb{P}\left[ Q = k \right] \nonumber\\
    &=& 2 \lambda (1-p) \mathbb{E} \left[ (\mathcal{S} - \ell)^+ \right]. 
\end{eqnarray*}  
\end{conjecture}

Supporting this, note that it follows from Proposition \ref{prop:mg1bound} that 
\begin{eqnarray}\label{eqn:mg1expecbound}
  2 \lambda (1-p) \left( \mathbb{E}\left[ \mathcal{S} \right] - \ell \right)^+  &\leq& \mathbb{E}\left[ N_I \right] \leq 2 \lambda (1-p) \left( \mathbb{E}\left[ \mathcal{S} \right] - \ell \mathbb{P} \left( \mathcal{S} > \ell \right) \right) . \nonumber
\end{eqnarray}

Finally, for the variance we have the following approximation, 
\begin{eqnarray}
    \mathrm{Var}\left[ N_I \right] &\approx& \mathrm{Var}\left[ N_{2\lambda (1-p)}\big( (\mathcal{S}-\ell)^+ \big) \right] \nonumber\\
    &=& 2 \lambda (1-p) \mathbb{E} \left[ (\mathcal{S}-\ell)^+ \right] + 4 \lambda^2 (1-p)^2 \mathrm{Var}\left [(\mathcal{S}-\ell)^+ \right].\label{eqn:mg1varinfec}
\end{eqnarray}}

{\color{black}
\subsubsection{Numerical results}

{\color{black} Our simulation framework is as follows:}

For each service distribution, we simulated the queue over a lifetime of 500,000 customers. For each set of parameters, we considered the average of five simulations. For all cases, we fixed $\mu = 1$ and $\ell = 0.5$. {\color{black}Again, note that as $\frac{\lambda}{\mu}\rightarrow 1$, it takes longer to reach a steady state, hence we expect to see less accurate simulations as $\lambda$ increases, especially for the higher moments. In Table~\ref{tab:deterministic-serve} we explore the $M/G/1$ queue with deterministic service times. In Table~\ref{tab:erlang-serve} we explore the $M/G/1$ queue with Erlang$\left(2,\frac{\mu}{2}\right)$ service times. In Table~\ref{tab:hyper-serve} we explore the $M/G/1$ queue with Hyper-Exponential$(2,\mu_1=2,\lambda_2=0.25)$ service times. In Table~\ref{tab:twopoint} we explore the $M/G/1$ queue with Symmetric 2-Point Distribution service times. Finally, in Section~\ref{sec:Qstable} we compare the across the four different service time distributions when we fix the average queue lengths to be the same for each distribution.}

For each service distribution, we simulate the queue over a lifetime of 100,000 customers. For each set of parameters, we considered the average of 5 simulations for each set of parameters. For all cases, we set $\ell = 0.5$. 

We additionally calculated the sojourn times over 50,000,000 customers using the Lindley recursion, and using this simulated the approximation equations for the probability of infection, mean number infected, and variance of the number infected. 

Firstly, for a deterministic service distribution, we simply allowed the parameter $\Delta = 1$ as in the $M/M/1$ queue, in order to preserve the same mean.

\begin{table}[h]
\centering
\begin{tabular}{|c|c|c|c|c|c|c|c|}
\hline
 $p$ & $\lambda$ & \makecell{$\mathbb{P}(I)$\\ (Sim.)}& \makecell{$\mathbb{P}(I)$\\(Eqn.~\ref{eqn:mg1probapprox})} & \makecell{$\mathbb{E}[N_I]$\\ (Sim.)}  & \makecell{$\mathbb{E}[N_I]$\\ (Conj.~\ref{conj:mg1meaninfec})} & \makecell{Var$[N_I]$\\(Sim.)}& \makecell{Var$[N_I]$\\ (Eqn.~\ref{eqn:mg1varinfec})}\\\hline
.01 & .50 & .00974 & .00992 & .996 & .990 & 1.847 & 1.561\\

.01 & 75 & .0288 & .0292 & 2.951 & 2.970 & 11.105 & 10.136\\

.01 & .90 & .0836 & .0828 & 9.029 & 8.893 & 91.608 & 82.231\\\hline

.02 & .50 & .0198 & .0197 & .984 & .980 & 1.849 & 1.540\\

.02 & .75 & .0565 & .0569 & 2.934 & 2.941 & 11.020 & 9.964\\

.02 & .90 & .156 & .153 & 8.976 & 8.811 & 85.727 & 80.900\\\hline

.10 & .50 & .0914 & .0927 & .900 & .900 & 1.602 & 1.372\\

.10 & .75 & .232 & .236 & 2.706 & 2.699 & 9.838 & 8.611\\

.10 & .90 & .474 & .482 & 8.039 & 8.091 & 70.934 & 68.915\\\hline

.50 & .50 & .339 & .358 & .500 & .500 & .720 & .646\\

.50 & .75 & .610 & .636 & 1.504 & 1.500 & 3.659 & 3.330\\

.50 & .90 & .825 & .843 & 4.604 & 4.496 & 26.138 & 23.270\\\hline

\end{tabular}
\caption{Simulation results for an arrival in the $M_{\lambda}/D/1$ queue in steady state.}
    \label{tab:deterministic-serve}
\end{table}

{\color{black} We note from Table~\ref{tab:deterministic-serve} that higher utility decreases the accuracy of the formulas; though we see a lesser degree of this than we might expect across other regimes, due to the low variability in overlap times and frequency created by the fact that the deterministic regime allows for no variance in service (and thus remaining) service times.}

For Erlang-2 service times, in order to preserve the mean, we scaled $\mu$ by a factor of 2; that is, $\mu = 2$.

\begin{table}[h]
\centering
\begin{tabular}{|c|c|c|c|c|c|c|c|}
\hline
 $p$ & $\lambda$ & \makecell{$\mathbb{P}(I)$\\ (Sim.)}& \makecell{$\mathbb{P}(I)$\\(Eqn.~\ref{eqn:mg1probapprox})} & \makecell{$\mathbb{E}[N_I]$\\ (Sim.)}  & \makecell{$\mathbb{E}[N_I]$\\ (Conj.~\ref{conj:mg1meaninfec})} & \makecell{Var$[N_I]$\\(Sim.)}& \makecell{Var$[N_I]$\\ (Eqn.~\ref{eqn:mg1varinfec})}\\\hline
.01 & .50 & .0125 & .0126 & 1.266 & 1.265 & 3.390 & 3.209\\

.01 & .75 & .0397 & .0397 & 4.114 & 4.107 & 23.219 & 22.791\\

.01 & .90 & .116 & 0.115 & 13.068 & 12.923 & 197.346 & 187.310\\\hline

.02 & .50 & .0247 & .0249 & 1.259 & 1.252 & 3.346 & 3.158\\

.02 & .75 & .0759 & .0763 & 4.056 & 4.067 & 22.960 & 22.433\\

.02 & .90 & .209 & .206 & 13.054 & 12.799 & 194.511 & 184.309\\\hline

.10 & .50 & .111 & .112 & 1.150 & 1.150 & 2.906 & 2.759\\

.10 & .75 & .286 & .288 & 3.680 & 3.733 & 19.097 & 19.191\\

.10 & .90 & .556 & .559 & 11.736 & 11.757 & 160.549 & 156.001\\\hline

.50 & .50 & .366 & .373 & .639 & .639 & 1.185 & 1.135\\

.50 & .75 & .640 & .648 & 2.077 & 2.073 & 7.138 & 6.847\\

.50 & .90 & .843 & .849 & 6.584 & 6.536 & 54.840 & 51.246\\\hline

\end{tabular}
\caption{Simulation results for an arrival to the $M_{\lambda}/E_2/1$ queue in steady state.}
    \label{tab:erlang-serve}
\end{table}

{\color{black} We again have in Table~\ref{tab:erlang-serve} lower accuracy at higher utility for the formulas; largely as a consequence of it taking longer to reach the steady state in these dynamics. This similarly shows up in the variance being most inaccurate, since there is higher variability before the steady state, such that in the simulation the system is not necessarily reaching the steady state level of lower behavioral variance within these 500,000 customers.}

For Hyper-Exponential(2) service times, we draw from an exponential distribution with parameter $\mu_1=2$ with probability $p_1$ and respectively draw with probability $1-p_1$ from an exponential with parameter $\mu_2=0.25$, such that $p_1 = 6/7$ to achieve the equivalent mean as in the prior cases.

\begin{table}[h]
\centering
\begin{tabular}{|c|c|c|c|c|c|c|c|}
\hline
 $p$ & $\lambda$ & \makecell{$\mathbb{P}(I)$\\ (Sim.)}& \makecell{$\mathbb{P}(I)$\\(Eqn.~\ref{eqn:mg1probapprox})} & \makecell{$\mathbb{E}[N_I]$\\ (Sim.)}  & \makecell{$\mathbb{E}[N_I]$\\ (Conj.~\ref{conj:mg1meaninfec})} & \makecell{Var$[N_I]$\\(Sim.)}& \makecell{Var$[N_I]$\\ (Eqn.~\ref{eqn:mg1varinfec})}\\\hline
.01 & .50 & .0297 & .0291 & 3.062 & 3.055 & 26.348 & 30.658\\

.01 & .75 & .104 & .103 & 11.926 & 11.939 & 252.186 & 265.519\\

.01 & .90 & .281 & 0.282 & 40.637 & 41.063 & 2122.171 & 2186.947\\\hline

.02 & .50 & .0552 & .0550 & 2.993 & 3.024 & 26.539 & 30.042\\

.02 & .75 & .184 & .181 & 11.959 & 11.800 & 247.087 & 259.090\\

.02 & .90 & .432 & .427 & 41.638 & 40.692 & 2280.029 & 2168.548\\\hline

.10 & .50 & .198 & .193 & 2.756 & 2.776 & 22.682 & 25.569\\

.10 & .75 & .469 & .455 & 11.148 & 10.857 & 222.709 & 220.026\\

.10 & .90 & .731 & .730 & 35.357 & 37.351 & 1510.140 & 1814.379\\\hline

.50 & .50 & .420 & .406 & 1.560 & 1.544 & 8.018 & 8.589\\

.50 & .75 & 0.685 &	.675 & 6.043 & 6.041 & 66.620 & 70.920\\

.50 & .90 & .868 & .862 & 20.862 & 20.657 & 559.507 & 568.866\\\hline
\end{tabular}
\caption{Simulation results for an arrival to the $M_{\lambda}/H_2/1$ queue in steady state.}
    \label{tab:hyper-serve}
\end{table}

{\color{black} The results shown in Table~\ref{tab:hyper-serve} are consistent with the earlier $M/G/1$ cases, though the much higher level of variance inherent to this service distribution is apparent in both the greater values of the results and the lower accuracy of the simulations over the same number of arrivals.}

Finally, for service times distributed according to a symmetric 2-point distribution, we use $\delta \in \{0.25, 0.50, 0.90\}$ such that $\mu_1 \in \{0.75, 0.50, 0.10\}$ and respectively $\mu_2 \in \{1.25, 1.50, 1.90\}$, thus maintaining the same mean as prior experiments. Here for the reader's ease we split up our table into four separate sub-tables according to the proportion of infectious arrivals of the system.}

{\begin{table}[p]
{\begin{tabular}{|c|c|c|c|c|c|c|c|}
\hline
$\lambda$ & $\delta$  & \makecell{$\mathbb{P}(I)$\\ (Sim.)}& \makecell{$\mathbb{P}(I)$\\(Eqn.~\ref{eqn:mg1probapprox})} & \makecell{$\mathbb{E}[N_I]$\\ (Sim.)}  & \makecell{$\mathbb{E}[N_I]$\\ (Conj.~\ref{conj:mg1meaninfec})} & \makecell{Var$[N_I]$\\(Sim.)}& \makecell{Var$[N_I]$\\ (Eqn.~\ref{eqn:mg1varinfec})}\\\hline

.50 & .25 & .0103 & .0102 & 1.025 & 1.021 & 1.997 & 1.746\\

.50 & .50 & .0111 & .0111 & 1.111 & 1.114 & 2.550 & 2.312\\

.50 & .90 & .0150 & .0149 & 1.497 & 1.498 & 4.057 & 3.869\\\hline

.75 & .25 & .0310 & .0305 & 3.131 & 3.109 & 13.067 & 11.464\\

.75 & .50 & .0349 & .0344 & 3.551 & 3.527 & 16.948 & 15.685\\

.75 & .90 & .0465 & .0467 & 4.858 & 4.858 & 30.516 & 29.625\\\hline

.90 & .25 & .0880 & .0870 & 9.459 & 9.412 & 98.046 & 93.732\\

.90 & .50 & .0992 & .0994 & 10.845 & 10.920 & 132.551 & 128.883\\

.90 & .90 & .136 & .135 & 15.407 & 15.439 & 256.670 & 261.104\\\hline
\end{tabular}
    \subcaption{$p = 0.01.$}
    \label{tab:01twopoint}}
{\begin{tabular}{|c|c|c|c|c|c|c|c|}
\hline
$\lambda$ & $\delta$  & \makecell{$\mathbb{P}(I)$\\ (Sim.)}& \makecell{$\mathbb{P}(I)$\\(Eqn.~\ref{eqn:mg1probapprox})} & \makecell{$\mathbb{E}[N_I]$\\ (Sim.)}  & \makecell{$\mathbb{E}[N_I]$\\ (Conj.~\ref{conj:mg1meaninfec})} & \makecell{Var$[N_I]$\\(Sim.)}& \makecell{Var$[N_I]$\\ (Eqn.~\ref{eqn:mg1varinfec})}\\\hline

.50 & .25 & .0200 & .0202 & 1.008 & 1.011 & 2.013 & 1.722 \\

.50 & .50 & .0219 & .0220 & 1.101 & 1.102 & 2.507 & 2.278\\

.50 & .90 & .0294 & .0293 & 1.483 & 1.482 & 3.979 & 3.802\\\hline

.75 & .25 & .0597 & .0593 & 3.082 & 3.078 & 12.388 & 11.271 \\

.75 & .50 & .0656 & .0665 & 3.443 & 3.490 & 15.981 & 15.398\\

.75 & .90 & .0885 & .0891 & 4.800 & 4.807 & 29.825 & 29.038\\\hline

.90 & .25 & .160 & .161 & 9.268 & 9.327 & 94.102 & 92.034\\

.90 & .50 & .181 & .181 & 10.799 & 10.823 & 129.761 & 126.758\\

.90 & .90 & .237 & .237 & 15.275 & 15.251 & 258.585 &  253.285\\\hline
\end{tabular}
\subcaption{$p = 0.02.$}
    \label{tab:02twopoint}}
{    \begin{tabular}{|c|c|c|c|c|c|c|c|}
\hline
$\lambda$ & $\delta$  & \makecell{$\mathbb{P}(I)$\\ (Sim.)}& \makecell{$\mathbb{P}(I)$\\(Eqn.~\ref{eqn:mg1probapprox})} & \makecell{$\mathbb{E}[N_I]$\\ (Sim.)}  & \makecell{$\mathbb{E}[N_I]$\\ (Conj.~\ref{conj:mg1meaninfec})} & \makecell{Var$[N_I]$\\(Sim.)}& \makecell{Var$[N_I]$\\ (Eqn.~\ref{eqn:mg1varinfec})}\\\hline

.50 & .25 & .0938 & .0948 & .929 & .928 & 1.771 & 1.528\\

.50 & .50 & .101 & .101 & 1.019 & 1.013 & 2.228 & 2.004\\

.50 & .90 & .130 & .131 & 1.358 & 1.362 & 3.428 & 3.319\\\hline

.75 & .25 & .239 & .243 & 2.822 & 2.826 & 10.652 & 9.736\\

.75 & .50 & .262 & .263 & 3.233 & 3.206 & 14.101 & 13.256\\

.75 & .90 & .323 & .326 & 4.368 & 4.416 & 24.463 & 24.880\\\hline

.90 & .25 & .487 & .493 & 8.576 & 8.556 & 82.042 & 78.000 \\

.90 & .50 & .522 & .526 & 9.937 & 9.944 & 110.261 & 108.082\\

.90 & .90 & .602 & .604 & 13.911 & 14.051 & 211.570 & 216.321\\\hline
\end{tabular}
    \subcaption{$p = 0.1$.}
\label{tab:10twopoint}}    \caption{Simulation results for an arrival in the $M_{\lambda}/G/1$ queue in steady state with Symmetric 2-Point Distribution service times.}
\end{table}}
{\begin{table}[H] 
  {\begin{tabular}{|c|c|c|c|c|c|c|c|}
\hline
$\lambda$ & $\delta$  & \makecell{$\mathbb{P}(I)$\\ (Sim.)}& \makecell{$\mathbb{P}(I)$\\(Eqn.~\ref{eqn:mg1probapprox})} & \makecell{$\mathbb{E}[N_I]$\\ (Sim.)}  & \makecell{$\mathbb{E}[N_I]$\\ (Conj.~\ref{conj:mg1meaninfec})} & \makecell{Var$[N_I]$\\(Sim.)}& \makecell{Var$[N_I]$\\ (Eqn.~\ref{eqn:mg1varinfec})}\\\hline

.50 & .25 & .341 & .358 & .514 & .516 & .769 & .701\\

.50 & .50 & .347 & .358 & .561 & .562 & .924 & .869\\

.50 & .90 & .416 & .417 & .757 & .756 & 1.406 & 1.360\\\hline

.75 & .25 & .612 & .636 & 1.565 & 1.570 & 3.989 & 3.701\\

.75 & .50 & .622 & .636 & 1.782 & 1.782 & 5.165 & 4.889\\

.75 & .90 & .687 & .685 & 2.478 & 2.453 & 9.066 & 8.769\\\hline

.90 & .25 & .828 & .843 & 4.791 & 4.758 & 27.935 & 26.267\\

.90 & .50 & .834 & .843 & 5.514 & 5.509 & 36.440 & 35.553\\

.90 & .90 & .869 & .868 & 7.760 & 7.798 & 67.609 & 70.273\\\hline
\end{tabular}
        \stepcounter{figure}
        \setcounter{caption@flags}{4}
        \setcounter{subfigure}{10}
    \subcaption{$p =0.5.$}
 \label{tab:50twopoint} }
    \caption{Simulation results for an arrival in the $M_{\lambda}/G/1$ queue in steady state with Symmetric 2-Point Distribution service times.}
\label{tab:twopoint}
\end{table}}

{\color{black} The results in Table~\ref{tab:twopoint} are unsurprising following the earlier $M/G/1$ cases, and we notably see less variance than the with the Hyper-Exponential(2) service distribution, which similarly draws from two different distributions at random, but adds more variance by having those two distributions additionally be non-deterministic. This also highlights that in systems with higher rate of infectious customers there is both a greater chance of susceptible customers becoming infected (since there are more infectious customers to possibly interact with them) and also lower moments of the number each infectious customer infects (since there are fewer susceptible customers for them to potentially interact with).}

\subsubsection{Average queue length stability}\label{sec:Qstable}
{\color{black} This exploration provides us with insight into Conjecture~\ref{conj:mg1meaninfec}, by considering the behavior across different service distributions where we create consistency between the distributions in the form of stabilizing the average queue length. Our simulation framework is as follows:}

Firstly, $\ell =0.5$ is fixed for all experiments. We now define $\mu$ such that the average service time is 1 across all experiments and distributions.  We can now fix the average queue length by determining $\lambda$ using the Pollaczek–Khinchine formula. This allows us to explore the relationship between infection probability and queue length. We average over 5 simulations with 100,000 customers per simulation.

\begin{eqnarray*}
    \mathbb{E}[Q] &=& \lambda\mathbb{E}[S] + \frac{\lambda^2\mathbb{E}^2[S]+\lambda^2\text{Var}[S]}{2(1-\lambda\mathbb{E}[S])}.
\end{eqnarray*}

\begin{table}[h]
\footnotesize
{\begin{tabular}{|c|l|c|c|c|c|c|c|c|}
\hline
  $\mathbb{E}[Q]$    & \makecell{Sim.\\Avgs.}        & \makecell{Deter-\\ministic} & \makecell{Erlang\\2}  & \makecell{Erlang\\20} & \makecell{Hyper-\\Exponential} & \makecell{2-Point\\(0.25)} & \makecell{2-Point\\(0.5)} & \makecell{2-Point\\(0.9)} \\ \hline
 \multirow{3}{*}{1} & $\mathbb{P}(I)$   &  .0133       &  .0146   &  .0146   &  .0155        &  .0146        &  .0151      &  .0184      \\
       & $\mathbb{E}[N_I]$ & 1.354       & 1.488   & 1.452   & 1.608        & 1.463       & 1.562      & 1.818      \\ 
       &    Var$[N_I]$        & 2.957       & 4.374   & 3.385   & 9.987        & 3.483       & 4.370      & 5.668      \\ \cline{1-9} 
       \multirow{3}{*}{2} & $\mathbb{P}(I)$   &  .0316       &  .0319   &  .0327   &  .0339        &  .0339       &  .0345      &  .0399      \\   
       & $\mathbb{E}[N_I]$ & 3.279       & 3.261   & 3.313   & 3.638        & 3.365        & 3.511      & 4.037      \\   
      &    Var$[N_I]$        & 14.340      & 15.498  & 14.163  & 35.786       & 13.764      & 15.551     & 21.109     \\ \cline{1-9} 
       \multirow{3}{*}{4}  & $\mathbb{P}(I)$   &  .0675        &  .0668   &  .0692   &  .0682        &  .0684        &  .0714      &  .0786      \\   
       & $\mathbb{E}[N_I]$ & 7.027       & 6.948   & 7.404    & 7.626        & 7.275       & 7.474      & 8.686       \\   
      &    Var$[N_I]$        & 55.252      & 56.077  & 58.532  & 113.834      & 58.193      & 61.022     & 84.361     \\ \cline{1-9} 
       \multirow{3}{*}{8} & $\mathbb{P}(I)$   &  .134       &  .130   &  .135   &  .125        &  .138       &  .140      &  .164      \\   
  & $\mathbb{E}[N_I]$ & 14.869      & 15.088  & 15.469  & 14.316       & 15.595      & 16.837     & 19.030     \\   
      &    Var$[N_I]$        & 258.026     & 237.033 & 244.034 & 329.908      & 255.875     & 283.715    & 343.632    \\ \hline

\end{tabular}
\subcaption{$p=0.01.$}}
\label{tab:avg-length01}
{\begin{tabular}{|c|l|c|c|c|c|c|c|c|}
\hline
 $\mathbb{E}[Q]$  & \makecell{Sim.\\Avgs.}      & \makecell{Deter-\\ministic} & \makecell{Erlang\\2}  & \makecell{Erlang\\20} & \makecell{Hyper-\\Exponential} & \makecell{2-Point\\(0.25)} & \makecell{2-Point\\(0.5)} & \makecell{2-Point\\(0.9)} \\ \hline
\multirow{3}{*}{1} & $\mathbb{P}(I)$   &  .0275       &  .0283   &  .0281   &  .0316        &  .0281       &  .0297      &  .0358      \\   
       & $\mathbb{E}[N_I]$ & 1.374         & 1.446   & 1.416    & 1.619        & 1.441       & 1.546      & 1.787      \\   
      &    Var$[N_I]$        & 3.035       & 4.196   & 3.248  & 10.417       & 3.516       & 4.278      & 5.220      \\ \cline{1-9} 
       \multirow{3}{*}{2} & $\mathbb{P}(I)$   &  .0610       &  .0631   &  .0629   &  .0627        &  .0628       &  .0661      &  .0737      \\   
       & $\mathbb{E}[N_I]$ & 3.174        & 3.240   & 3.255   & 3.373        & 3.186       & 3.421      & 4.021      \\   
      &    Var$[N_I]$        & 12.669      & 15.946  & 13.592  & 30.956       & 12.743      & 15.526      & 19.978      \\ \cline{1-9} 
       \multirow{3}{*}{4} & $\mathbb{P}(I)$   &  .127       &  .127     &  .128    &  .119        &  .128       &  .135      &  .150      \\   
       & $\mathbb{E}[N_I]$ & 7.140       & 7.299   & 7.222   & 6.974        & 7.107        & 7.697      & 8.933      \\   
      &    Var$[N_I]$        & 58.601      & 64.449  & 57.108  & 97.821       & 56.813      & 69.530     & 94.048      \\ \cline{1-9} 
       \multirow{3}{*}{8} & $\mathbb{P}(I)$   &  .226       &  .236   &  .246   &  .216        &  .234       &  .245      &  .267       \\   
       & $\mathbb{E}[N_I]$ & 14.262      & 15.064  & 16.145  & 14.705       & 14.686      & 15.420     & 18.080     \\   
      &    Var$[N_I]$        & 229.584     & 264.046 & 301.633 & 335.643      & 202.051     & 239.803    & 350.918    \\ \hline
\end{tabular}
        \stepcounter{figure}
        \setcounter{caption@flags}{4}
        \setcounter{subfigure}{10}
\subcaption{$p=0.02.$}}
\label{tab:avg-length02}
\caption{Simulations for the $M/G/1$ queue when average queue length is fixed by varying $\lambda$.}
\end{table}

\begin{table}[h]
\footnotesize
{\begin{tabular}{|c|l|c|c|c|c|c|c|c|}
\hline
 $\mathbb{E}[Q]$     &\makecell{Sim.\\Avgs.}      & \makecell{Deter-\\ministic} & \makecell{Erlang\\2}  & \makecell{Erlang\\20} & \makecell{Hyper-\\Exponential} & \makecell{2-Point\\(0.25)} & \makecell{2-Point\\(0.5)} & \makecell{2-Point\\(0.9)} \\ \hline

 \multirow{3}{*}{1} & $\mathbb{P}(I)$   &  .125       &  .129   &  .129   &  .125         &  .126       &  .130       &  .154       \\   
       & $\mathbb{E}[N_I]$ & 1.280       & 1.369   & 1.335   & 1.558        & 1.295       & 1.355      & 1.646      \\   
      &    Var$[N_I]$        & 2.766       & 3.826    & 3.074   & 9.388        & 2.929       & 3.445      & 4.548      \\ \cline{1-9} 
       \multirow{3}{*}{2} & $\mathbb{P}(I)$   &  .248       &  .243   &  .251   &  .215        &  .251       &  .256      &  .289      \\   
       & $\mathbb{E}[N_I]$ & 2.964       & 2.963   & 2.960   & 3.187          & 2.974       & 3.153      & 3.689      \\   
      &    Var$[N_I]$        & 11.538      & 13.399  & 11.055  & 29.768       & 11.695      & 13.473     & 18.130     \\ \cline{1-9} 
       \multirow{3}{*}{4} & $\mathbb{P}(I)$   &  .420        &  .405   &  .422   &  .356        &  .434       &  .431      &  .470      \\   
       & $\mathbb{E}[N_I]$ & 6.441       & 6.397   & 6.703   & 6.760        & 6.795       & 6.712       & 8.193      \\   
      &    Var$[N_I]$        & 46.632      & 50.622   & 52.908   & 97.561       & 51.381       & 50.405     & 77.350     \\ \cline{1-9} 
       \multirow{3}{*}{8} & $\mathbb{P}(I)$   &  .598       &  .593   &  .609   &  .520        &  .614        &  .620       &  .653      \\   
       & $\mathbb{E}[N_I]$ & 13.620       & 13.071  & 14.164  & 14.136       & 14.808      & 14.838     & 17.071     \\   
      &    Var$[N_I]$        & 216.012     & 171.044 & 224.333 & 364.995      & 234.967     & 231.473    & 317.673    \\ \hline
      \end{tabular}
              \stepcounter{figure}
        \setcounter{caption@flags}{4}
        \setcounter{subfigure}{10}
  \subcaption{$p=0.1.$}}
  \label{tab:avg-length1}
    {\begin{tabular}{|c|l|c|c|c|c|c|c|c|}
\hline
 $\mathbb{E}[Q]$      & \makecell{Sim.\\Avgs.}     & \makecell{Deter-\\ministic} & \makecell{Erlang\\2}  & \makecell{Erlang\\20} & \makecell{Hyper-\\Exponential} & \makecell{2-Point\\(0.25)} & \makecell{2-Point\\(0.5)} & \makecell{2-Point\\(0.9)} \\ \hline
  
\multirow{3}{*}{1} & $\mathbb{P}(I)$   &  .420       &  .402   &  .422   &  .309        &  .417       &  .422      &  .458      \\   
       & $\mathbb{E}[N_I]$ &  .705        &  .748   &  .724   &  .831        &  .717       &  .777      &  .917      \\   
      &    Var$[N_I]$        & 1.163       & 1.479   & 1.242   & 3.103         & 1.229       & 1.455      & 1.881      \\ \cline{1-9} 
       \multirow{3}{*}{2} & $\mathbb{P}(I)$   &  .628       &  .586   &  .629    &  .445        &  .623       &  .619      &  .647      \\   
       & $\mathbb{E}[N_I]$ & 1.616        & 1.668   & 1.659   & 1.748        & 1.676       & 1.768      & 2.051      \\   
      &    Var$[N_I]$        & 4.043       & 5.019   & 4.303   & 9.329        & 4.553       & 5.068      & 6.500      \\ \cline{1-9} 
       \multirow{3}{*}{4} & $\mathbb{P}(I)$   &  .787       &  .750   &  .792   &  .591        &  .784       &  .780      &  .793        \\   
       & $\mathbb{E}[N_I]$ & 3.531       & 3.620   & 3.687   & 3.640         & 3.643         & 3.865      & 4.337      \\   
      &    Var$[N_I]$        & 15.650      & 18.160  & 17.161  & 28.346       & 16.861      & 20.507     & 23.162     \\ \cline{1-9} 
       \multirow{3}{*}{8} & $\mathbb{P}(I)$   &  .888       &  .862   &  .887   &  .725        &  .884       &  .887      &  .891      \\   
       & $\mathbb{E}[N_I]$ & 7.337       & 7.430   & 7.640   & 7.465         & 7.390       & 8.796      & 10.104      \\   
      &    Var$[N_I]$        & 56.261      & 63.164  & 65.055  & 97.696       & 55.121      & 94.298     & 114.934    \\ \hline
\end{tabular}
\subcaption{$p=0.5.$}}
\label{tab:avg-length5}
\caption{Simulations for the $M/G/1$ queue when average queue length is fixed by varying $\lambda$.}
\label{tab:avg-length}
\end{table}
\FloatBarrier
{\color{black} Table~\ref{tab:avg-length} provides insight for our conjecture about the important of average queue length on impacting the spread of infection in the system. The first thing we can note is the greater consistency across the service regimes of the probabilities and moments of infection spread. In particular for the distributions with lower variance and thus tighter clustering of the randomly drawn service times around their means, the greatest contributor to ensuring over the lifespan of the queue that the queue lengths are more consistently near the expected queue length.}

\section{Conclusion} \label{conclusion}
In this paper, we study the $M/M/1$ queue with infectious and susceptible customers.  We analyze a model where susceptible customers can become infected if they overlap with an infectious customer for more than $\ell$ units of time in the queue.  We provide new formulas for the probability of infection and expected number of infections an infectious customer might cause while in the system.  These new formulas and expressions provide new insights for disease mitigation in the context of service systems.  

There are many areas for further research on this topic.  The first would be to extend this model to many server queues.  This is more challenging as the overlap time for even two servers is more complicated because the ordering of departures is lost.  The second extension would be to extend this current model to one with time varying arrival rates or even analyzing the transient behavior of the model.  This is challenging since we need to know the transient or time varying behavior of the $M/M/1$ queue.  Fortunately, work by \citet{massey2023transient} might help in this regard.    Finally, it would be interesting to consider more complex thresholds that could take into account the amount of time that a customer spends with any infectious customer. These more complicated threshold would enable modeling of different ways that a susceptible customer could get infected from infectious customers.

\section*{Acknowledgements}
Jamol Pender would like to acknowledge the gracious support of the National Science Foundation DMS Award \# 2206286.  

\section{Appendix} \label{appendix}

\bibliographystyle{plainnat}
\bibliography{threshold}

\end{document}